\documentclass[10pt,times,letter]{article}

\usepackage{amsmath, amsfonts,amsthm,amssymb,bm}
\usepackage{dsfont}

\usepackage{xcolor}
\definecolor{backgroundcolor}{RGB}{199,238,206}

\usepackage[english]{babel}

\usepackage{latexsym,amssymb,amsfonts,graphicx}
\usepackage{amsmath,dsfont}
\usepackage{verbatim}
\usepackage{mathrsfs}
\usepackage{bm}
\usepackage{color}
\usepackage{epsfig}

\usepackage{url}
\usepackage{bm}
\usepackage{natbib}
 \usepackage[T1]{fontenc}
\usepackage[colorlinks,linkcolor=blue,citecolor=blue,anchorcolor=blue]{hyperref}

\def\esssup_#1{\underset{#1}{\mathrm{ess\,sup\, }}}
\def\essinf_#1{\underset{#1}{\mathrm{ess\,inf\, }}}
\def\argmax_#1{\underset{#1}{\mathrm{arg\,max\, }}}
\def\argmin_#1{\underset{#1}{\mathrm{arg\,min\, }}}

\newtheorem{theorem}{Theorem}[section]

\newtheorem{proposition}[theorem]{Proposition}
\newtheorem{remark}[theorem]{Remark}
\newtheorem{lemma}[theorem]{Lemma}

\newtheorem{assumption}[theorem]{Assumption}
\newtheorem{property}[theorem]{Property}

\numberwithin{equation}{section}

\title{A particle system approach towards the global well-posedness of master equations for potential mean field games of controls}
\author{Huafu Liao \thanks{Email: hfliao@dlut.edu.cn, School of Mathematical Sciences, Dalian University of Technology.}\and
Chenchen Mou \thanks{Email: chencmou@cityu.edu.hk, Department of Mathematics, City University of Hong Kong.}
}
\date{}                                           
\begin{document}
\maketitle
\begin{abstract}
This paper studies $N$-particle systems as well as HJB/master equations for a class of generalized mean field control (MFC) problems and the corresponding potential mean field games of controls (MFGC). Such kind of HJB equations are nonlinearly coupled with the distribution induced by the Wasserstein derivatives of solutions. A local in time classical solution for the HJB equation is generated via a probabilistic approach based on the mean field maximum principle. Given an extension of the displacement convexity condition, we obtain the uniform estimates on the HJB equation for the $N$-particle system. Such estimates imply the displacement convexity/semi-concavity and thus the a priori estimates on the solution to the HJB equation for generalized MFC problems.  The global well-posedness of HJB/master equations for generalized MFC/potential MFGC is then proved via the local well-posedness and the a priori estimates. Thanks to the displacement convexity condition, such global well-posedness is also true for the degenerated case. Our analysis on $N$-particle systems also induces Lipschitz approximators to optimal feedback functions in  generalized MFC/potential MFGC where an algebraic convergence rate is obtained. An alternative approximate Markovian Nash equilibrium is constructed based on the $N$-particle system, where the approximation error is also quantified.
\vspace{0.3 cm}


\vspace{0.3 cm}

\end{abstract}

\section{Introduction}
\subsection{Context and motivation}
In this note we carefully study the following HJB equation for $N$-particle systems ($N\geq1$)
\begin{align}\label{HJB-N}
  \left\{\begin{aligned}
    &\partial_tV_N(t,x)+\frac12\sum_{i=1}^N{\rm tr}\big(D^2_{x_ix_i}V_N(t,x)\sigma\sigma^\top\big)+\frac12\sum_{i,j=1}^N{\rm tr}\big(D^2_{x_ix_j}V_N(t,x)\sigma_0\sigma^\top_0\big)+H_N\big(x,\nabla_x V_N(t,x)\big)=0,\\
    &V_N(T,x)=U(\mu_x),\quad(t,x)\in[0,T)\times(\mathbb R^d)^N,
  \end{aligned}\right.
\end{align}
as well as its interaction with the corresponding mean field limit. In \eqref{HJB-N}, $\sigma$, $\sigma_0\geq0$ are constant matrices and for $(x,p)=(x_1,\ldots,x_N,p_1,\ldots,p_N)\in(\mathbb R^d)^N\times(\mathbb R^d)^N$ we have defined $\mu_x=\frac1N\sum_{i=1}^N\delta_{x_i}$ as well as the Hamiltonian 
\begin{align}\label{def-HN}
  H_N(x,p)&:=\mathcal H\bigg(\frac1N\sum_{i=1}^N\delta_{(x_i,Np_i)}\bigg),
\end{align}
where
\begin{align}\label{extended-hamiltonian}
 \mathcal H:\quad\mathcal P_2(\mathbb R^d\times\mathbb R^d)\ \rightarrow\ \mathbb R,
\end{align}
 is another generalized Hamiltonian defined on the set of probability measures with finite second moments $\mathcal P_2(\mathbb R^d\times\mathbb R^d)$. The mean field limits of such particle systems give rise to the HJB equation for generalized mean field control problems
 {\small\begin{align}\label{HJB-MF}
  \left\{\begin{aligned}
    &\partial_tV(t,\mu)+\frac12\int_{\mathbb R^d}{\rm tr}\big(D_x\partial_\mu V(t,\mu,x)\big)\mu(dx)+\frac12\int_{\mathbb R^d}\int_{\mathbb R^d}{\rm tr}\big(\partial^2_{\mu\mu}V(t,\mu,x,y)\sigma_0\sigma^\top_0\big)\mu(dx)\mu(dy)+\mathcal H(\tilde\mu)=0,\\
    &V(T,\mu)=U(\mu),\quad (t,\mu)\in[0,T)\times\mathcal P_2(\mathbb R^d),
  \end{aligned}\right.
\end{align}}
where
\begin{align*}
\tilde\mu:=\big(Id,\partial_\mu V(t,\mu,\cdot)\big)\sharp\mu,
\end{align*}
and $\partial_t,\ D_x$ are standard temporal and spatial derivatives, $\partial_\mu,\ \partial^2_{\mu\mu}$ are $\mathcal W_2$-Wasserstein derivatives. 

Mean field games theory, pioneered independently by \cite{Huang06,Lasry07}, studies the strategic interactions among a large population of agents, where the behavior of population is described via a measure flow. Its twin, mean field control theory, studies the optimal control problems involving the measure flow of the controlled process. Since its inception, MFG/MFC theory and their applications have been intensively studied for over a decades. The intersection of MFG and MFC is the potential mean field games. It has been indicated in the seminal work of Lasry and Lions \cite{Lasry07} that the optimizers of MFC problems induce the solutions to a particular class of MFG, which is later called the potential MFG and studied in \cite{Briani18,Cardaliaguet15,Cardaliaguet17-1,Cardaliaguet15-1,Cardaliaguet17,Cecchin2022,Graber25,Guo2023,Hofer24} and so on. For the same reason, classical solutions to HJB equation \eqref{HJB-MF} for MFGC further yield classical solutions to the following master equation, which describes the so called potential mean field games of controls (see \cite{Graber21,Hofer24})
{\small\begin{align}\label{HJB-MFGC}
  \left\{\begin{aligned}
    &\quad\partial_t\mathcal V(t,x,\mu)+\frac12{\rm tr}\big(D^2_{xx}\mathcal V(t,x,\mu)(\sigma\sigma^\top+\sigma_0\sigma^\top_0)\big)+\breve{\cal H}\big(\breve\mu,x,\partial_x\mathcal V(t,x,\mu)\big)\\
    &+\frac12\int_{\mathbb R^d}{\rm tr}\big(\partial_y\partial_\mu\mathcal V(t,x,\mu,y)(\sigma\sigma^\top+\sigma_0\sigma^\top_0)\big)\mu(dy)+\frac12\int_{\mathbb R^d}\int_{\mathbb R^d}{\rm tr}\big(\partial^2_{\mu\mu}\mathcal V(t,x,\mu,y,z)\sigma_0\sigma^\top_0\big)\mu(dy)\mu(dz)\\
    &+\int_{\mathbb R^d}{\rm tr}\big(\partial_x\partial_\mu\mathcal V(t,x,\mu,y)\sigma_0\sigma^\top_0\big)\mu(dy)+\int_{\mathbb R^d}\partial_\mu\mathcal V(t,x,\mu,y)\cdot\partial_p\breve{\cal H}\big(\breve\mu,y,\partial_x\mathcal V(t,y,\mu)\big)\mu(dy)=0,\\
    &\mathcal V(T,x,\mu)=G(x,\mu),\ (t,x,\mu)\in[0,T)\times\mathbb R^d\times\mathcal P_2(\mathbb R^d).
  \end{aligned}\right.
\end{align}}
Here we have adopted the notation
\begin{align*}
 \breve\mu:=\big(Id,\partial_{\tilde\mu}\mathcal H^{(p)}(\tilde\mu,\cdot,\partial_x\mathcal V(t,\cdot,\mu)\big)\sharp\tilde\mu,\quad\tilde\mu:=\big(Id,\partial_x\mathcal V(t,\cdot,\mu)\big)\sharp\mu,
\end{align*}
as well as a structural constraint
\begin{align*}
  \partial_p\breve{\cal H}(\breve\mu,x,p)=\partial_{\tilde\mu}\mathcal H^{(p)}(\tilde\mu,x,p),\ \partial_x\breve{\cal H}(\breve\mu,x,p)=\partial_{\tilde\mu}\mathcal H^{(x)}(\tilde\mu,x,p),\ \partial_xG(x,\mu)=\partial_\mu U(\mu,x).
\end{align*}
In a controlled mean field system, $\partial_{\tilde\mu}\mathcal H^{(p)}(\tilde\mu,\cdot,\partial_x\mathcal V(t,\cdot,\mu)$ is the momentum of the controlled particle. As for mean field games of controls, given the above structural constraint, $\partial_{\tilde\mu}\mathcal H^{(p)}(\tilde\mu,\cdot,\partial_x\mathcal V(t,\cdot,\mu)$ equals the feedback strategy for the player, and $\breve\mu$ relates to the joint distribution of the player's position and strategy. Therefore, equations \eqref{HJB-MF} and \eqref{HJB-MFGC} could apply to several generalized MFC and potential MFGC that stem from different mathematical models. Similar equations with such feature also appear {in studies such as} \cite{Cosso16,Mou2022MFGC}.
We note here that \eqref{HJB-MF} and \eqref{HJB-MFGC} are considered as equations on $\mathcal P_2(\mathbb R^d)$ instead of $\mathcal P_2(\mathbb R^d)$ ($d>1$) only because we can use more elegant notations and elucidate better when dealing with the particle system corresponding to \eqref{HJB-N}. There isn't any essential difficulty to generalize our results to equations on $\mathcal P_2(\mathbb R^d)$. 

As an important mathematical tool dealing with mean field systems, HJB/master equations for MFC/MFG describe the dynamics of the population as well as its interaction with individual particles. Among other related topics, classical solutions to these equations are very desirable. A natural reason is that classical solutions could induce optimal feedback functions.  Moreover, solutions with enough regularity also help quantify the convergence, of both value functions and optimizers/equilibriums, at the optimal convergence rate $O(N^{-1})$ (see \cite{Bayraktar2024,Cardaliaguet2022,Cardaliaguet19,Car2023,Cecchin2025,Lacker19,Lacker20,Daudin2023,Lacker16,Lacker17,Lacker2018,Lacker20-1} and so on). While the well-posedness of local in time classical solutions has been established under quite general settings (see e.g. \cite{BENSOUSSAN2019ESAIM,Gangbo15,Mayorga20}), the global well-posedness of HJB/master equations in general remains a challenge. Towards that end, different kinds of conditions on HJB/master equations are proposed. A major breakthrough in that direction is the global well-posedness results under Lasry-Lions monotonicity condition, see \cite{BENSOUSSAN2019,Cardaliaguet19,Carmona2018-I,Jean14} and so on. Later on, an alternative monotonicity condition, the displacement monotonicity condition, is proposed for global well-posedness, see e.g. \cite{BanMesMou,Gangbo22,Gangbo22AOP,Mou2024}. In \cite{Bertucci2024} the regularity of solutions is investigated by considering the Lipschitz continuity of the initial condition.

In terms of the existing results on classical solutions, most of them are concerned with the distribution of the particle's position only, while \eqref{HJB-MF} and \eqref{HJB-MFGC} incorporate the particle's momentum into the distribution. Compared to the standard MFG, only a rather limited number of works, including \cite{Graber23,Mou2022MFGC}, have considered classical solutions of master equations for MFGC, which also admits the aforementioned joint distribution feature. After an adaptation, HJB/master equations with merely particle's distribution are basically the counterparts of \eqref{HJB-MF} or \eqref{HJB-MFGC} with the Hamiltonian 
\begin{align}\label{old-form}
 \mathcal H(\tilde\mu)=\int_{\mathbb R^{2d}}H(x\sharp\tilde\mu,x,p)\tilde\mu(dxdp),
\end{align}
or
\begin{align}\label{old-form-1}
 \breve{\mathcal H}(x,p,\tilde\mu)=\breve H(x,p,x\sharp\tilde\mu).
\end{align}
where $x\sharp\tilde\mu$ denotes the pushforward of $\tilde\mu$ by the projection function $(x,p)\mapsto x$. One feature of the form in \eqref{old-form}$\sim$\eqref{old-form-1}, which gives a ``linear property'' with respect to $p\sharp\tilde\mu$ in certain sense, is the following 
\begin{align*}
 \partial^2_{\tilde\mu\tilde\mu}\mathcal H^{(p)(p)}(\tilde\mu,x,p,\hat x,\hat p)=\partial_{\tilde\mu}\breve{\mathcal H}^{(p)}(x,p,\tilde\mu,\hat x,\hat p)=0,
\end{align*}
where the notations are from \eqref{explain-notation} denoting the $p$ component of Wasserstein derivatives:
\begin{align*}
 &\partial_{\tilde\mu}\mathcal H(\tilde\mu,x,p)=\big(\partial_{\tilde\mu}\mathcal H^{(x)}(\tilde\mu,x,p),\partial_{\tilde\mu}\mathcal H^{(p)}(\tilde\mu,x,p)\big)\in\mathbb R^d\times\mathbb R^d,\\
 &\partial^2_{\tilde\mu\tilde\mu}\mathcal H^{(p)}(\tilde\mu,x,p,\hat x,\hat p)=\big(\partial^2_{\tilde\mu\tilde\mu}\mathcal H^{(p)(x)}(\tilde\mu,x,p,\hat x,\hat p),\partial^2_{\tilde\mu\tilde\mu}\mathcal H^{(p)(p)}(\tilde\mu,x,p,\hat x,\hat p)\big)\in\mathbb R^{d\times d}\times\mathbb R^{d\times d},\\
 &\partial_{\tilde\mu}\breve{\mathcal H}(x,p,\tilde\mu,\hat x,\hat p)=\big(\partial_{\tilde\mu}\breve{\mathcal H}^{(x)}(x,p,\tilde\mu,\hat x,\hat p),\partial_{\tilde\mu}\breve{\mathcal H}^{(p)}(x,p,\tilde\mu,\hat x,\hat p)\big)\in\mathbb R^d\times\mathbb R^d.
\end{align*}
However, there are examples of practical mean field games or mean field control problems, emerging from different fields, that are beyond the scope of \eqref{old-form}$\sim$\eqref{old-form-1}. In practice, when formulating the mathematical model, it is often the case that not only the particle's position, but its joint distribution with the momentum are considered. To illustrate the intuition, consider the following Hamiltonian for controlled $N$-particle systems:
\begin{align*}
 H_N(x,p)=\inf_{\theta\in\mathbb R^k}\bigg\{\frac1N\sum_{i=1}^NL\bigg(x_i,\frac1N\sum_{i=1}^N\delta_{x_i},\theta\bigg)+\sum_{i=1}^Nf\bigg(x_i,\frac1N\sum_{i=1}^N\delta_{x_i},\theta\bigg)\cdot p_i+\lambda|\theta|^2\bigg\},\\ (x,p)\in(\mathbb R^d)^N\times(\mathbb R^d)^N.
\end{align*}
Such Hamiltonians stem from the models describing deep learning, see e.g. \cite{E17,E19,Ruthotto2020}. The mean field counterpart of the above is 
\begin{align*}
  \mathcal H(\tilde\mu):=\inf_{\theta\in\mathbb R^k}\bigg\{\int_{\mathbb R^{2d}}\big[L(x,x\sharp\tilde\mu,\theta)+f(x,x\sharp\tilde\mu,\theta)p\big]\tilde\mu(dxdp)+\lambda|\theta|^2\bigg\},\ \tilde\mu\in\mathcal P_2(\mathbb R^d\times\mathbb R^d).
\end{align*}
The above Hamiltonian no longer admits the form in \eqref{old-form}. Nonetheless, it can still be written in the generalized form $\mathcal H(\tilde\mu)$. Another typical example with such generalized Hamiltonian $\mathcal H(\tilde\mu)$ is the extended mean field control problems, on which we will discuss more in Section \ref{The propagation of chaos and the verification results} and \ref{global-well-posedness-0}. One can find the early references on such control problems in \cite{Acciaio2019,Wei2017} and the later ones in e.g. \cite{Bo24,Djete22,Djete22-AOP}. For more concrete examples on the generalized Hamiltonian, one might refer to the robust mean field control problems and mean field games (see e.g. \cite{Huang17,Ismail2019,Pham22}), the particle systems with centralized control (see e.g. \cite{Bo22,Yu2022,E19,Huafu24}), the mean field games with relative utility (see e.g. \cite{Huang16}), the zero-sum stochastic differential games of mean field type (see e.g. \cite{Cosso16,Li2016}) and so on. 

\subsection{Our approach and related literature}
Motivated by the situation above, we consider classical solutions to \eqref{HJB-MF} and \eqref{HJB-MFGC} with an attempt to narrow the gap between the HJB/master equation theory and models in practice. As far as we know, this is the first paper concerning the classical solutions to the HJB equation \eqref{HJB-MF} on generalized MFC.

As is mentioned, potential MFGC has the advantage of lying at the interface between mean field games and mean field control, thus can be studied within the framework of generalized MFC, which brings the underlying symmetry structure. Such symmetry structure is generally absent in MFGC (see e.g. \cite{Graber23,Kobeissi2022,Mou2022MFGC}), and allows us to adopt convexity conditions analogous to those employed in standard MFC problems such as \cite{Gangbo22,Mayorga23,Swiech24}, where the Hamiltonians admit the form in \eqref{old-form}$\sim$\eqref{old-form-1}. The Lasry-Lions monotonicity condition could also be utilized in similar researches via different approaches from ours. The related references include \cite{BENSOUSSAN2023,Cardaliaguet19,Mou2022MFGC}.

Our approach relies on two key steps, i.e., generating a local in time solution $V$ and establishing the a priori estimates on $V$. Due to the presence of $\big(Id,\partial_\mu V(t,\mu,\cdot)\big)\sharp\mu$, the local in time solution $V$ is generated in an indirect way. We first generate $\partial_\mu V(t,\mu,x)$ via the mean field FBSDE describing the extended version of the stochastic maximum principle, which is an analogy to the ones in \cite{Acciaio2019,Carmona2018-I,Carmona2015}.  Then we recover $V(t,\mu)$ via $\partial_\mu V(t,\mu,x)$ in a probabilistic way. 

Our main contribution is developing a new way to establish the a priori estimates on $N$-particle systems \eqref{HJB-N} that are uniform in $N$. In particular, the mentioned estimates yield the a priori estimates on $|D_x\partial_\mu V|_\infty$ and $|\partial^2_{\mu\mu} V|_\infty$ that are crucial to the global well-posedness of \eqref{HJB-MF}. The a priori estimates on the first order derivatives of $V_N$ in \eqref{HJB-N} is obtained by the stochastic maximum principle. For the a priori estimates on the second order derivatives, a $dN\times dN$-dimensional symmetric backward Riccati stochastic equation is introduced to decouple the Pontryagin system so that the uniform in $N$ a priori estimates are obtained by carefully analyzing the eigenvalues of solutions, which are $O(N^{-1})$. Our analysis on the backward Riccati stochastic equations exploits the symmetry feature of potential mean field games so that the Hamiltonians in \eqref{HJB-MF} are allowed to admit non-strict displacement convexity/concavity w.r.t. the marginal distribution of position/momentum.

Our approach towards $|D_x\partial_\mu V|_\infty$ and $|\partial^2_{\mu\mu} V|_\infty$ yield additional knowledge on related $N$-particle systems. It is distinc from the idea in \cite{Graber23,Mou2022MFGC} which focus on mean field systems. The analysis on controlled $N$-particle system here is carried out differently from \cite{Gangbo22,Mayorga23,Swiech24} since both individual and common noise are now involved in the MFC framework. Because Hamiltonians with non-strict displacement convexity/concavity are considered here, the bottom-up method in recent works \cite{Cirant2026,Jackson2025} can't be directly applied.

Given the a priori estimates on $|\partial_x\partial_\mu V(t,\cdot)|_\infty$ and $|\partial^2_{\mu\mu}V(t,\cdot)|_\infty$, we establish the a priori estimates on higher order derivatives by inductively taking derivatives in \eqref{HJB-MF} and referring to the Feynman-Kac representation for master equations in \cite{Buckdahn2017}. The global well-posedness of \eqref{HJB-MF} is then shown by continuation method.

The uniform a priori estimates on $N$-particle system also enables utilizing the particle system to formulate Lipschitz continuous approximations to the optimal feedback functions for generalized MFC, where the error is quantified. In a similar fashion of the particle approximation towards the value functions \cite{Bayraktar2024,Cardaliaguet2022,Car2023,Cecchin2025,Daudin2023,Lacker17,Lacker2018}, the proposed approximater is named as the particle approximation of $\partial_\mu V$.

The rest of the paper is organized as follows. In Section \ref{Notations and assumptions} we introduce notations, assumptions and our main results. In Section \ref{The propagation of chaos and the verification results} we establish the propagation of chaos as well as the verification results concerned with equation \eqref{HJB-MF}. In Section \ref{generation-of-solution} we use a probabilistic approach to generate a local in time solution to \eqref{HJB-MF}. In Section \ref{prior-on-2nd} we establish a priori estimates on $V_N$ and $V$. In Section \ref{global-well-posedness-0} we deal with the global well-posedness of HJB/master equations \eqref{HJB-MF}$\sim$\eqref{HJB-MFGC} where the degenerate case is also considered. In Section \ref{approximate-partial-mu-V} we construct an approximation to $\partial_\mu V$ as well as an alternative approximate Nash equilibrium based on the particle system, where quantitative error analysis is given.

\section{Notations, assumptions and main results}\label{Notations and assumptions}
Let $\mathcal P_2(\mathbb R^d)$ ($d=1,2$) be the set of all Borel probability measures on $\mathbb R^d$ where each $\mu\in\mathcal P_2(\mathbb R^d)$ satisfies
\begin{align*}
 \int_{\mathbb R^d}|x|^2\mu(dx)<+\infty,
\end{align*}
and $\mathcal P_2(\mathbb R^d)$ is equipped with the $\mathcal W_2$-Wasserstein metric
\begin{align*}
 \mathcal W_2(\mu,\nu):=\inf_{\gamma\in\Gamma(\mu,\nu)}\bigg(\int_{\mathbb R^d\times\mathbb R^d}|x-y|^2\gamma(dxdy)\bigg)^{\frac12}.
\end{align*}
Here $\Gamma(\mu,\nu)$ denotes the set of all Borel probability measures $\gamma$ on $\mathbb R^d\times\mathbb R^d$ with marginals $\mu,\nu$. Given a bounded measurable vector-valued function $f:\ \mathbb R^d\to\mathbb R^d$, we denote by $f\sharp\mu$ the push-forward of $\mu$ by $f$.

Next we introduce the differentiability of functions of probability measures, see e.g. \cite{Carmona2018-I}. Let $U:\mathcal P_2(\mathbb R^d)\to\mathbb R$ be a $\mathcal W_2$-continuous function. Its Wasserstein gradient takes the form
\begin{align*}
 \partial_\mu U(\mu,x):\ (\mu,x)\in\mathcal P_2(\mathbb R^d)\times\mathbb R^d\ \to\ \mathbb R^d,
\end{align*}
and can be characterized by
\begin{align*}
 U(\mathcal L_{\xi+\eta})-U(\mathcal L_\xi)=\mathbb E\big[\partial_\mu U(\mathcal L_\xi,\xi)\cdot\eta\big]+o\big(\mathbb E[|\eta|^2]^\frac12\big),
\end{align*}
for any square integrable random variables $\xi,\eta$. We also note a useful result that for empirical measures $\frac1N\sum_{i=1}^N\delta_{x_i},\ N\in\mathbb N,\ x_1,\ldots,x_N\in\mathbb R^d$, denote by $u_N(x_1,\ldots,x_N):=U\big(\frac1N\sum_{i=1}^N\delta_{x_i}\big)$, then
\begin{align*}
 \partial_{x_j}u_N(x_1,\ldots,x_N)=\frac1N\partial_\mu U\bigg(\frac1N\sum_{i=1}^N\delta_{x_i},x_j\bigg),\ 1\leq j\leq N.
\end{align*}
Let $\mathcal C^0\big(\mathcal P_2(\mathbb R^d)\big)$ denote the set of $\mathcal W_2$-continuous functions $U:\mathcal P_2(\mathbb R^d)\to\mathbb R$. By $\mathcal C^1\big(\mathcal P_2(\mathbb R^d)\big)$ we mean the space of functions $U\in\mathcal C^0\big(\mathcal P_2(\mathbb R^d)\big)$ such that $\partial_\mu U$ exists and is continuous on $\mathcal P_2(\mathbb R^d)\times\mathbb R^d$. Similarly, $\mathcal C^2\big(\mathcal P_2(\mathbb R^d)\big)$ is the space of functions $U\in\mathcal C^1\big(\mathcal P_2(\mathbb R^d)\big)$ such that the following maps exist and are all jointly continuous:
\begin{align*}
 &\mathcal P_2(\mathbb R^d)\times\mathbb R^d\ni(\mu,x)\mapsto\partial_x\partial_\mu U(\mu,x);\\
 &\mathcal P_2(\mathbb R^d)\times\mathbb R^d\times\mathbb R^d\ni(\mu,x,\tilde x)\mapsto\partial^2_{\mu\mu}U(\mu,x,\tilde x).
\end{align*}
Inductively, $\mathcal C^3\big(\mathcal P_2(\mathbb R^d)\big)$ is the space of functions $U\in\mathcal C^2\big(\mathcal P_2(\mathbb R^d)\big)$ such that the third order derivatives, i.e., the following mappings exist and are all jointly continuous:
\begin{align*}
 &\mathcal P_2(\mathbb R^d)\times\mathbb R^d\ni(\mu,x)\mapsto D^2_x\partial_\mu U(\mu,x);\\
 &\mathcal P_2(\mathbb R^d)\times\mathbb R^d\times\mathbb R^d\ni(\mu,x,\tilde x)\mapsto\big(D_x\partial^2_{\mu\mu}U(\mu,x,\tilde x),D_{\tilde x}\partial^2_{\mu\mu}U(\mu,x,\tilde x)\big);\\
 &\mathcal P_2(\mathbb R^d)\times\mathbb R^d\times\mathbb R^d\times\mathbb R^d\ni(\mu,x,\tilde x,\hat x)\mapsto\partial^3_{\mu\mu\mu}U(\mu,x,\tilde x,\hat x).
\end{align*}
We note that in the right hand side above, each entry  is in $\mathbb R^{d\times d\times d}$, the space of linear operators that map $\mathbb R^d$ to $\mathbb R^{d\times d}$. In the same way as above, we further inductively define $\mathcal C^k\big(\mathcal P_2(\mathbb R^d)\big)$, $k=4,5,6$. We also use $\mathcal C\big(\mathcal P_2(\mathbb R^d)\times\mathbb R^d\big)$ and $\mathcal C\big([0,T]\times\mathcal P_2(\mathbb R^d)\big)$ to denote the jointly continuous functions defined on $\mathcal P_2(\mathbb R^d)\times\mathbb R^d$ and $[0,T]\times\mathcal P_2(\mathbb R^d)$ respectively.

Consider the generalized Hamiltonian in \eqref{extended-hamiltonian}$\sim$\eqref{HJB-MF}. For $\tilde\mu\in\mathcal P_2(\mathbb R^d\times\mathbb R^d)$, we have $\partial_{\tilde\mu}\mathcal H(\tilde\mu,x,p)\in\mathbb R^d\times\mathbb R^d$. Denote by
\begin{align}\label{explain-notation}
 \partial_{\tilde\mu}\mathcal H(\tilde\mu,x,p)=\big(\partial_{\tilde\mu}\mathcal H^{(x)}(\tilde\mu,x,p),\partial_{\tilde\mu}\mathcal H^{(p)}(\tilde\mu,x,p)\big).
\end{align}
That is, for square integrable random variables $\xi_i$, $\eta_i\in\mathbb R^d$, $i=1,2$,
\begin{align*}
 &\quad\mathcal H\big(\mathcal L_{(\xi_1+\eta_1,\xi_2+\eta_2)}\big)-\mathcal H\big(\mathcal L_{(\xi_1,\xi_2)}\big)\notag\\
 &=\mathbb E\big[\partial_{\tilde\mu}\mathcal H^{(x)}(\mathcal L_{(\xi_1,\xi_2)},\xi_1,\xi_2)\cdot\eta_1+\partial_{\tilde\mu}\mathcal H^{(p)}(\mathcal L_{(\xi_1,\xi_2)},\xi_1,\xi_2)\cdot\eta_2\big]+o\big(\mathbb E[|\eta_1|^2+|\eta_2|^2]^\frac12\big).
\end{align*}
We may also inductively define
\begin{align*}
 \partial^2_{\tilde\mu\tilde\mu}\mathcal H^{(x)}(\tilde\mu,x_1,p_1,x_2,p_2)=\big(\partial^2_{\tilde\mu\tilde\mu}\mathcal H^{(x)(x)}(\tilde\mu,x_1,p_1,x_2,p_2),\partial^2_{\tilde\mu\tilde\mu}\mathcal H^{(x)(p)}(\tilde\mu,x_1,p_1,x_2,p_2)\big)\in\mathbb R^{d\times d}\times\mathbb R^{d\times d}.
\end{align*}
The notations $\partial^2_{\tilde\mu\tilde\mu}\mathcal H^{(x)(p)},\ \partial^2_{\tilde\mu\tilde\mu}\mathcal H^{(p)(x)},\ \partial^2_{\tilde\mu\tilde\mu}\mathcal H^{(p)(p)}$ and so on, are defined similarly.

With the previous notations, we may now propose the following assumptions on the parameters. We note here that such assumptions are the extension of the displacement convexity.
\begin{assumption}\label{assumption}
Suppose the following
\begin{enumerate}
 \item $U\in\mathcal C^6\big(\mathcal P_2(\mathbb R^d)\big)$ with Lipschitz continuous and bounded derivatives, $\mathcal H\in\mathcal C^6\big(\mathcal P_2(\mathbb R^d\times\mathbb R^d)\big)$ with Lipschitz continuous and bounded derivatives;
 \item $U$ and $\mathcal H$ are displacement convex with respect to $x$-marginal distribution in the sense that for any probability measure $\mu\in\mathcal P_2(\mathbb R^d)$, $\tilde\mu\in\mathcal P_2(\mathbb R^{2d})$ and any test mapping $\varphi\in C^\infty_c(\mathbb R^d;\mathbb R^d)$, $\phi\in C^\infty_c(\mathbb R^d;\mathbb R^d)$,
 \begin{align}
  \label{displacement-1}\int_{\mathbb R^d}\int_{\mathbb R^d}\varphi(x)^\top\partial^2_{\mu\mu}U(\mu,x,y)\varphi(y)\mu(dx)\mu(dy)+\int_{\mathbb R}\varphi(x)^\top D_x\partial_\mu U(\mu,x)\varphi(x)\mu(dx)\geq0,\\
   \label{displacement-2}\int_{\mathbb R^2}\int_{\mathbb R^2}\phi(\xi_1)^\top\partial^2_{\tilde\mu\tilde\mu}\mathcal H^{(x)(x)}(\tilde\mu,\xi_1,\xi_2)\phi(\xi_2)\tilde\mu(d\xi_1)\tilde\mu(d\xi_2)+\int_{\mathbb R^2}\phi(\xi)^\top D_x\partial_{\tilde\mu}\mathcal H^{(x)}(\tilde\mu,\xi)\phi(\xi)\tilde\mu(d\xi)\geq0.
 \end{align}
 \item $-\mathcal H$ is displacement convex with respect to the $p$-marginal distribution in the sense that for any probability measure $\tilde\mu\in\mathcal P_2(\mathbb R^2)$ test function $\phi\in C^\infty_c(\mathbb R^2)$,
 \begin{align}
  \label{displacement-3}\int_{\mathbb R^2}\int_{\mathbb R^2}\phi(\xi_1)^\top\partial^2_{\tilde\mu\tilde\mu}\mathcal H^{(p)(p)}(\tilde\mu,\xi_1,\xi_2)\phi(\xi_2)\tilde\mu(d\xi_1)\tilde\mu(d\xi_2)+\int_{\mathbb R^2}\phi(\xi)^\top D_x\partial_{\tilde\mu}\mathcal H^{(p)}(\tilde\mu,\xi)\phi(\xi)\tilde\mu(d\xi)\leq0.
 \end{align}
\end{enumerate}
\end{assumption}
For $\tilde\mu\in\mathcal P_2(\mathbb R^d\times\mathbb R^d)$ with lifting $(X,Y)\in L^2(\Omega;\mathbb R^d\times\mathbb R^d)$, an illustrative example is $$\mathcal H\big({\rm Law}(X,Y)\big)=H(\mathbb EX,\mathbb EY),$$ where $H:\mathbb R^d\times\mathbb R^d\to\mathbb R$ is smooth, admits bounded derivatives, and is convex/concave in its first/second variable. It is easy to see that $\mathcal H$ satisfies Assumption \ref{assumption} but is generally outside the framework of standard MFC.

Assumption \ref{assumption} has the following implication.
\begin{lemma}\label{assum-impli}
 Assume Assumption \ref{assumption}, then
 \begin{align*}
  D_x\partial_\mu U(\mu,\cdot),\ -D_x\partial_{\tilde\mu}\mathcal H^{(p)}(\tilde\mu,\cdot)\geq0,\quad\mu\in\mathcal P_2(\mathbb R^d),\ \tilde\mu\in\mathcal P_2(\mathbb R^d\times\mathbb R^d),
 \end{align*}
 as well as
 \begin{align}\label{sym}
  D_x\partial_{\tilde\mu}\mathcal H^{(p)}(\tilde\mu,\cdot)=D_p\partial_{\tilde\mu}\mathcal H^{(x)}(\tilde\mu,\cdot),\quad\partial^2_{\tilde\mu\tilde\mu}\mathcal H^{(x)(p)}(\tilde\mu,\cdot)=\partial^2_{\tilde\mu\tilde\mu}\mathcal H^{(p)(x)}(\tilde\mu,\cdot).
 \end{align}

\end{lemma}
\begin{proof}
 In view of the continuous differentiability of $U$ and $\mathcal H$, it suffices to focus on the empirical measures
 \begin{align*}
  \mu_N=\frac1N\sum_{i=1}^N\delta_{x_i},\quad\tilde\mu_N=\frac1N\sum_{i=1}^N\delta_{(x_i,p_i)},
 \end{align*}
 which converge to $\mu$ and $\tilde\mu$ in $\mathcal P_2(\mathbb R)$ and $\mathcal P_2(\mathbb R\times\mathbb R)$ respectively. For the same reason, we lose nothing by assuming that $x_1,\ldots,x_N,p_1,\ldots,p_N$ are mutually different. For arbitrary $\alpha\in\mathbb R^d$, take $\varphi\in C^\infty_c(\mathbb R^d;\mathbb R^d)$ such that
 \begin{align*}
  \varphi(x_1)=\alpha,\ \varphi(x)=0,\ x=x_2,\ldots,x_N.
 \end{align*}
Then according to Assumption \ref{assumption},
\begin{align*}
 &\quad\int_{\mathbb R^d}\int_{\mathbb R^d}\varphi(x)^\top\partial^2_{\mu\mu}U(\mu_N,x,y)\varphi(y)\mu_N(dx)\mu_N(dy)+\int_{\mathbb R^d}\varphi(x)^\top D_x\partial_\mu U(\mu_N,x)\varphi(x)\mu_N(dx)\notag\\
 &=\frac1{N^2}\alpha^\top\partial^2_{\mu\mu}U(\mu_N,x_1,x_1)\alpha+\frac1N\alpha^\top D_x\partial_\mu U(\mu_N,x_1)\alpha\geq0.
\end{align*}
Hence
\begin{align*}
 \frac1N\alpha^\top\partial^2_{\mu\mu}U(\mu_N,x_1,x_1)\alpha+\alpha^\top D_x\partial_\mu U(\mu_N,x_1)\alpha\geq0.
\end{align*}
Notice that $\partial^2_{\mu\mu}U$ is bounded, we may let $N$ go to infinity in the above and obtain that
\begin{align*}
 D_x\partial_\mu U(\mu,x_1)\geq0.
\end{align*}
The way of showing $-D_x\partial_{\tilde\mu}\mathcal H^{(p)}(\tilde\mu,\cdot)\geq0$ is the same as the above. Let's now turn to \eqref{sym}. Consider
\begin{align*}
 H(x_1,\ldots,x_N,p_1,\ldots,p_N):=\mathcal H(\tilde\mu_N).
\end{align*}
Then
\begin{align*}
 D_{p_i}H(x_1,\ldots,x_N,p_1,\ldots,p_N)&=\frac1N\partial_{\tilde\mu}\mathcal H^{(p)}(\tilde\mu_N,x_i,p_i),\\
 D_{x_i}H(x_1,\ldots,x_N,p_1,\ldots,p_N)&=\frac1N\partial_{\tilde\mu}\mathcal H^{(x)}(\tilde\mu_N,x_i,p_i).
\end{align*}
In view of the continuous differentiability of $\mathcal H$, we have $D_{x_i}D_{p_i}H=D_{p_i}D_{x_i}H$. Hence
\begin{align*}
 &\quad\frac1ND_x\partial_{\tilde\mu}\mathcal H^{(p)}(\tilde\mu_N,x_i,p_i)+\frac1{N^2}\partial^2_{\tilde\mu\tilde\mu}\mathcal H^{(x)(p)}(\tilde\mu_N,x_i,p_i,x_i,p_i)\notag\\
 &=\frac1ND_p\partial_{\tilde\mu}\mathcal H^{(x)}(\tilde\mu_N,x_i,p_i)+\frac1{N^2}\partial^2_{\tilde\mu\tilde\mu}\mathcal H^{(p)(x)}(\tilde\mu_N,x_i,p_i,x_i,p_i)
\end{align*}
Send $N$ to infinity and we have obtained \eqref{sym}.
\end{proof}
Given Assumption \ref{assumption}, in Section \ref{The second order estimates on $V_N$}, we will prove our first main result describing the a priori estimates on $V_N$ in \eqref{HJB-N}.
\begin{theorem}\label{eigen-0-1}
 Suppose Assumption \ref{assumption} and $\sigma>0$. There exists a unique classical solution to \eqref{HJB-N} and a constant $C$ depending only on \eqref{C-depend} such that
  \begin{align}\label{eigen1}
    0\leq\sum_{i,j=1}^N\xi^\top_iD^2_{x_jx_j}V_N(t,x)\xi_j\leq\frac CN\sum_{i=1}^N|\xi_i|^2,\quad\xi=(\xi_1,\ldots,\xi_N)\in(\mathbb R^d)^N,\ (t,x)\in[0,T]\times(\mathbb R^d)^N.
  \end{align}
\end{theorem}
The constant $C$ in \eqref{eigen1} is uniform in $N$. Together with the propagation of chaos shown in Section \ref{The propagation of chaos and the verification results}, we may obtain the convexity and sub-concavity of $V$ in \eqref{HJB-MF}, which leads to the a priori estimates on the local in time solution  generated via FBSDE \eqref{1st-step}. As a result, we will prove the global well-posedness of \eqref{HJB-MF} in Section \ref{global-well-posedness}.
\begin{theorem}\label{existence-of-decoupling-field}
  Suppose Assumption \ref{assumption} and $\sigma>0$. The mean field FBSDE \eqref{1st-step} admits a global decoupling field
  \begin{align}\label{Y_t=V(X)}
 Y_t=\partial_\mu V(t,\mathbb P_{X_t},X_t),\quad V_t=V(t,\mathbb P_{X_t}),
\end{align}
where $V(t,\cdot)\in\mathcal C^4\big(\mathcal P_2(\mathbb R^d)\big)$, $\partial_tV(t,\cdot)\in\mathcal C^2\big(\mathcal P_2(\mathbb R^d)\big)$ solves \eqref{HJB-MF} and $|\partial^{i_1}_{x_1}\cdots\partial^{i_j}_{x_j}\partial^j_\mu V|_\infty$ ($0\leq j\leq 4,\ 0\leq i_1,\ldots,i_j\leq4,\ 0\leq i_1+\cdots+i_j+j\leq 4$) depends only on \eqref{tilde c-dependence} with $K=U$. As a result, $V$ is the unique classical solution to \eqref{HJB-MF} where all derivatives are bounded.
 \end{theorem}
Before moving on to the well-posedness of \eqref{HJB-MFGC}, we introduce the following property on joint differentiability and the structural assumptions.
\begin{property}\label{joint-differentiation}
Let $\varphi:\mathbb R^d\times\mathcal P_2(\mathbb R^d)\to\mathbb R^k$ for some $k\in\mathbb N$, satisfying
  \begin{enumerate}
  \item For $(x,\mu)\in\mathbb R^d\times\mathcal P_2(\mathbb R^d)$, $\varphi(\cdot,\mu)\in C^2(\mathbb R^d),\ \partial_\mu\varphi(\cdot,\mu,\cdot)\in C^1(\mathbb R^d\times\mathbb R^d),\ \varphi(x,\cdot)\in\mathcal C^2\big(\mathcal P_2(\mathbb R^d)\big),\ \partial_x\varphi(x,\cdot)\in\mathcal C^1\big(\mathcal P_2(\mathbb R^d)\big)$;
  \item The above derivatives of $\varphi$ are bounded and jointly Lipschitz continuous.
 \end{enumerate}
\end{property}
\begin{assumption}\label{potential-game-assumption}
Suppose that
\begin{enumerate}
\item For $(\breve\mu,x,p)\in\mathcal P_2(\mathbb R^{2d})\times\mathbb R^d\times\mathbb R^d$, $\breve{\cal H}(\breve\mu,x,p)$ satisfies Property \ref{joint-differentiation}.
\item There exist $\mathcal H$ and $U$ satisfying Assumption \ref{assumption} such that
 \begin{align}\label{hamilton-coincidence}
 D_p\breve{\cal H}(\breve\mu,x,p)=\partial_{\tilde\mu}\mathcal H^{(p)}(\tilde\mu,x,p),\ D_x\breve{\cal H}(\breve\mu,x,p)=\partial_{\tilde\mu}\mathcal H^{(x)}(\tilde\mu,x,p),\ D_xG(x,\mu)=\partial_\mu U(\mu,x),
\end{align}
where $(x,p,\tilde\mu)\in\mathbb R^d\times\mathbb R^d\times\mathcal P_2(\mathbb R^{2d})$, $\mu\in\mathcal P_2(\mathbb R^d)$, and
\begin{align}\label{measure-implicit}
 \breve\mu(dxdp)=\big(x,\partial_{\tilde\mu}\mathcal H^{(p)}(\tilde\mu,x,p)\big)\sharp\tilde\mu(dxdp).
\end{align}
\end{enumerate}
\end{assumption}
{\begin{remark} In \eqref{measure-implicit}, $\tilde\mu$ and $\breve\mu$ are linked implicitly, please see more related discussion in \cite{Mou2022MFGC,Carmona2018-I}. For the case with standard potential MFG, where $\mathcal H$ and $\breve{\cal H}$ admit the forms in \eqref{old-form}$\sim$\eqref{old-form-1}, the Assumption \ref{potential-game-assumption} reduces to \eqref{hamilton-coincidence} alone:
\begin{align*}
 D_x\breve H(x,p,x\sharp\tilde\mu)&=D_x H(x\sharp\tilde\mu,x,p)+\int_{\mathbb R^d}\partial_\mu H(x\sharp\tilde\mu,x,y,q)\tilde\mu(dydq),\\
 D_p\breve H(x,p,x\sharp\tilde\mu)&=D_pH(x\sharp\tilde\mu,x,p),
\end{align*}
which is consistent with the ones in \cite{Briani18,Cecchin2022} where
\begin{align*}
 H(x\sharp\tilde\mu,x,p)=a(x,p)+F(x\sharp\tilde\mu).
\end{align*}
\end{remark}}
It will be shown in Section \ref{po-mfgc} that, as long as Assumption \ref{potential-game-assumption} is satisfied, the well-posedness of \eqref{HJB-MF} would imply the well-posedness of \eqref{HJB-MFGC} which describes potential mean field games of controls.
\begin{theorem}\label{well-posedness-mfgc}
 Suppose Assumption \ref{assumption}, Assumption \ref{potential-game-assumption} and $\sigma>0$. Then \eqref{HJB-MFGC} admits a unique classical solution such that $\partial^k_x\mathcal V(t,x,\cdot)\in\mathcal C^{3-k}\big(\mathcal P_2(\mathbb R^d)\big)$ for $(t,x,\mu)\in[0,T]\times\mathbb R^d\times\mathcal P_2(\mathbb R^d)$, $k=0,1,2,3$, where all derivatives are bounded by constants depending only on \eqref{tilde c-dependence} with $K=U$.
\end{theorem}
We note here that the a priori estimates in Theorem \ref{eigen-0-1}, Theorem \ref{existence-of-decoupling-field} and Theorem \ref{well-posedness-mfgc} are uniform in $\sigma>0$. Therefore in Section \ref{degenerate} we obtain results on the well-posedness for the degenerated case $\sigma=0$, which are the counterparts of Theorem \ref{existence-of-decoupling-field} and Theorem \ref{well-posedness-mfgc}.
\begin{theorem}\label{degenerate-HJB-MF}
 Suppose Assumption \ref{assumption} and denote the solution to \eqref{HJB-MF}  with $\sigma>0$ by $V^\sigma$. Then there exists $V(t,\cdot)\in\mathcal C^3\big(\mathcal P_2(\mathbb R^d)\big)$ such that for each $(t,\mu)\in[0,T]\times\mathcal P_2(\mathbb R^d)$,
 \begin{align*}
  \lim_{\sigma\to0}V^\sigma(t,\mu)=V(t,\mu),
 \end{align*}
and $V$ is the unique classical solution to \eqref{HJB-MF} with  $\sigma=0$. Moreover, $|D^{i_1}_{x_1}\cdots D^{i_j}_{x_j}\partial^j_\mu V|_\infty$ ($0\leq j\leq 3,\ 0\leq i_1,\ldots,i_j\leq3,\ 0\leq i_1+\cdots+i_j+j\leq 3$) depends only on \eqref{tilde c-dependence} with $K=U$.
\end{theorem}
\begin{theorem}\label{well-posedness-mfgc-degenerate}
 Suppose Assumption \ref{assumption}, Assumption \ref{potential-game-assumption} and $\sigma=0$. Then \eqref{HJB-MFGC} admits a unique classical solution such that $D^k_x\mathcal V(t,x,\cdot)\in\mathcal C^{2-k}\big(\mathcal P_2(\mathbb R^d)\big)$ for $(t,x,\mu)\in[0,T]\times\mathbb R^d\times\mathcal P_2(\mathbb R^d)$, $k=0,1,2$, where all derivatives are bounded by constants depending only on \eqref{tilde c-dependence} with $K=U$.
 \end{theorem}
Finally we mention that, in Section \ref{approximate-partial-mu-V}, the a priori estimates and well-posedness above give rise to Lipschitz approximators to optimal feedback functions in generalized MFC/potential MFGC, as well as an approximate Markovian Nash equilibrium. The corresponding convergence rates are also analysed.

\section{The propagation of chaos and the verification results}\label{The propagation of chaos and the verification results}
As is mentioned above, the HJB equation \eqref{HJB-MF} can be linked to the HJB equation \eqref{HJB-N} for controlled $N$-particle systems. In this section, we carefully study the propagation of chaos of such particle systems. We study the propagation of chaos not only due to its own interests, but also because our approach relies on it in such a way that the regularity of $V$ is indicated by the uniform estimates on $V_N$. For a better understanding on the propagation of chaos, consider the following representing controlled $N$-particle system
\begin{align}\label{N-particle}
   \left\{\begin{aligned}
    &dX^{\theta,i}_N(s) = \theta^i(s)ds +\sigma dW^i_s+\sigma_0 dW^0_s,\\
&X^{\theta,i}_N(t)=x_i\in\mathbb R^d,\ i=1,\ldots,N,
\end{aligned}\right.
 \end{align}
where $\{W^i\}_{i=0}^N$ are independent $d$-dimensional Brownian motions. The system \eqref{N-particle} evolves according the criterion of minimizing $J_N:\ [0,T]\times(\mathbb R^d)^N\times\mathcal U^{ad,N}_t\ \mapsto\ \mathbb R$ in the following:
 \begin{align}\label{optim.}
  J_N(t,x_1,\ldots,x_N,\theta):=\mathbb E\bigg[\int_t^TL_N\big(X^\theta_N(s),\theta(s)\big)ds+U(\mu_{X^\theta_N(T)})\bigg].
  \end{align}
 Here $\mathcal U^{ad}_t$ will be explained in the following and for $(x,\theta)\in(\mathbb R^d)^N\times(\mathbb R^d)^N$,
 \begin{align*}
  L_N(x,\theta):=\sup_{p\in\mathbb R^N}\big\{H_N\big(x,p\big)-\theta\cdot p\big\}.
 \end{align*}
The displacement convexity of $-\mathcal H$ in \eqref{displacement-3} and Fenchel biconjugation theorem then implies
\begin{align*}
 H_N\big(x,p\big)=\inf_{\theta\in\mathbb R^N}\big\{L_N(x,\theta)+\theta\cdot p\big\}.
\end{align*}
The optimal control problem \eqref{N-particle}$\sim$\eqref{optim.} is understood in the weak sense. To be exact, we consider the optimization problem over the admissible set $\mathcal U^{ad}_t$ consisting of the tuple $(\Omega,\mathbb P,\mathcal F,\{W^i\}_{i=0}^N,\theta)$ such that
\begin{itemize}
\item $(\Omega,\mathbb P,\mathcal F)$ is a complete probability space;
\item $\{W^i\}_{i=0}^N$ are independent $d$-dimensional Brownian motions defined on $(\Omega,\mathbb P,\mathcal F)$ with $W^i_t=0$ almost surely and $\mathcal F_s:=\sigma\big(W^i_u,u\in[t,s],i\geq0\big)$ augmented by all the $\mathbb P$-null sets in $\mathcal F$;
  \item $\theta=(\theta^1,\ldots,\theta^N)$ is an $\{\mathcal F_s\}_{t\leq s\leq T}$-adapted process on $(\Omega,\mathbb P,\mathcal F)$ with state space $(\mathbb R^d)^N$;
  \item For any $(x_1,\ldots,x_N)$, $X^\theta_N:=(X^{\theta,1}_N,\ldots,X^{\theta,N}_N)^\top$ solves \eqref{N-particle} on $(\Omega,\mathbb P,\mathcal F,\{\mathcal F_s\}_{t\leq s\leq T})$.
\end{itemize}
By convention, the above tuple in $\mathcal U^{ad,N}_t$ is denoted by $\theta$ when there is no ambiguity. With the notations above we may define the value function
\begin{align}\label{optim.-value-function}
 V_N(t,x_1,\ldots,x_N):=\inf_{\theta\in \mathcal U^{ad,N}_t}J_N(t,x_1,\ldots,x_N,\theta).
\end{align}
The formal HJB equation corresponding to \eqref{N-particle}$\sim$\eqref{optim.-value-function} is \eqref{HJB-N}. We heuristically obtain \eqref{HJB-MF} by passing $N$ to infinity in \eqref{HJB-N} and thus understand \eqref{HJB-MF} as the HJB equation with a generalized Hamiltonian. Such intuition is justified in Proposition \ref{prop-propagation-1} below.
\begin{proposition}\label{prop-propagation-1}
  Suppose 
  \begin{enumerate}
  \item Assumption \ref{assumption};
  \item $V_N\in C^{1,2}\big([0,T)\times(\mathbb R^d)^N\big)\cap C\big([0,T]\times(\mathbb R^d)^N\big)$ is the classical solution to \eqref{HJB-N} with bounded derivatives;
  \item $V$ is the classical solution to \eqref{HJB-MF}, $\partial_tV(\cdot)\in\mathcal C\big([0,T]\times\mathcal P_2(\mathbb R^d)\big)$, $V(t,\cdot)\in\mathcal C^2\big(\mathcal P_2(\mathbb R^d)\big)$ with jointly continuous and bounded derivatives.
  \end{enumerate}
 Then there exists a constant $C$ depending only on $\sigma$ and $|\partial^2_{\mu\mu}V|_\infty$ such that
  \begin{align}\label{propagation-1}
    \big|V(t,\mu_x)-V_N(t,x_1,\ldots,x_N)\big|\leq\frac CN,\quad\mu_x:=\frac1N\sum_{i=1}^N\delta_{x_i},\ x:=(x_1,\ldots,x_N)\in(\mathbb R^d)^N.
  \end{align}
\end{proposition}
\begin{proof}
  Denote by
  \begin{align*}
    v_N(t,x_1,\ldots,x_N):=V(t,\mu_x).
  \end{align*}
  According to \eqref{HJB-MF},
  {\small\begin{align*}
  \left\{\begin{aligned}
    &\partial_tv_N(t,x)+\frac12\sum_{i=1}^N{\rm tr}\big(D^2_{x_ix_i}v_N(t,x)\sigma\sigma^\top\big)+\frac12\sum_{i,j=1}^N{\rm tr}\big(D^2_{x_ix_j}v_N(t,x)\sigma_0\sigma^\top_0\big)+H_N\big(x,\nabla_x v_N(t,x)\big)\\
    &\qquad\qquad-\frac1{2N^2}\sum_{i=1}^N{\rm tr}\big(\partial^2_{\mu\mu}V(t,\mu_x,x_i,x_i)\sigma\sigma^\top\big)=0,\\
    &v_N(T,x)=U(\mu_x),\quad(t,x)\in[0,T)\times(\mathbb R^d)^N.
  \end{aligned}\right.
\end{align*}}
Define
\begin{align*}
  \hat v_N(t,x):=v_N(t,x)-V_N(t,x),\quad g_N(t,x):=\frac1{2N^2}\sum_{i=1}^N{\rm tr}\big(\partial^2_{\mu\mu}V(t,\mu_x,x_i,x_i)\sigma\sigma^\top\big).
\end{align*}
Since $\partial^2_{\mu\mu}V$ is bounded, we get
\begin{align*}
  |g_N(t,x)|\leq\frac CN,
\end{align*}
where $C$ depends only on $|\partial^2_{\mu\mu}V|_\infty$. Moreover, the boundedness of $\partial_{\tilde\mu}\mathcal H^{(p)}$ implies that $H_N(x,p)$ is Lipshitz in $p$. Therefore we have the existence of bounded $h(t,x):\ [0,T]\times(\mathbb R^d)^N\ \to\ (\mathbb R^d)^N$ such that
\begin{align}\label{propagation-0}
  \left\{\begin{aligned}
    &\partial_t\hat v_N(t,x)+\frac12\sum_{i=1}^N{\rm tr}\big(D^2_{x_ix_i}\hat v_N(t,x)\sigma\sigma^\top\big)+\frac12\sum_{i,j=1}^N{\rm tr}\big(D^2_{x_ix_j}\hat v_N(t,x)\sigma_0\sigma^\top_0\big)+h(t,x)\cdot\nabla_x\hat v_N(t,x)=g_N(t,x),\\
    &\hat v_N(T,x)=0,\quad(t,x)\in[0,T)\times(\mathbb R^d)^N.
  \end{aligned}\right.
\end{align}
An application of the maximum principle yields the estimates on $|\hat v_N|_\infty$ which implies \eqref{propagation-1}.
\end{proof}
\begin{remark}
 According to the definition, $g_N=0$ whenever $\sigma=0$. Then \eqref{propagation-0} gives $\hat v_N(t,x)=0$, i.e., $V(t,\mu_x)=V_N(t,x)$. In other words, $V_N$ is exactly the finite projection of $V$ (see \cite{Cecchin2025,Huafu24,Swiech24}) when there is no individual noise.
\end{remark}
Next to the propagation of chaos in Proposition \ref{prop-propagation-1}, we show that $V$ in \eqref{HJB-MF} could actually serve as the value function of the corresponding extended mean field control problems, which is the continuation of \eqref{N-particle}$\sim$\eqref{optim.-value-function}. Consider the representative particle 
\begin{align}\label{representative-particle}
   \left\{\begin{aligned}
    &dX_t= \theta_tdt +\sigma dW_t+\sigma_0 dW^0_t,\\
&X_0=\xi\in\mathcal F_0,
\end{aligned}\right.
 \end{align}
 where the control $\theta\in\mathcal U^{ad,N}_0$ with $N=1$ and \eqref{N-particle} replaced with \eqref{representative-particle}. We are interested in minimizing the following
\begin{align}\label{mean-field-optim.}
   J(\theta,\mathcal L_\xi):=\mathbb E\bigg[\int_0^TL\big(\mu_{(X_t,\theta_t)}\big)dt+U(\mu_{X_T})\bigg],
 \end{align}
 where $\mu_{(X_t,\theta_t)}$ is the distribution conditional on $\mathcal F^{W_0}_t$ and
 \begin{align*}
  L(\tilde\mu):=\sup_{\hat\mu\in\mathcal P_2(\mathbb R^{3d}),\atop (\pi_1,\pi_3)\sharp\hat\mu=\tilde\mu}\bigg\{\mathcal H\big((\pi_1,\pi_2)\sharp\hat\mu\big)
-\int_{\mathbb R^d\times\mathbb R^d} p\cdot\theta\hat\mu(dxdpd\theta)\bigg\}.
 \end{align*}
Given the displacement concavity \eqref{displacement-3} and regularity in Assumption \ref{assumption}, an application of Fenchel biconjugation theorem yields
\begin{align}\label{example-H}
 \mathcal H(\tilde\mu)=\inf_{\hat\mu\in\mathcal P_2(\mathbb R^{3d}),\atop (\pi_1,\pi_2)\sharp\hat\mu=\tilde\mu}\bigg\{L\big((\pi_1,\pi_3)\sharp\hat\mu\big)
+\int_{\mathbb R^d\times\mathbb R^d} p\cdot\theta\hat\mu(dxdpd\theta)\bigg\}.
 \end{align}
We note here that for $\tilde\mu\in\mathcal P_2(\mathbb R^{3d})$ with liftings $(X,Y,\Theta)$, $\pi_1\sharp\tilde\mu,\ \pi_2\sharp\tilde\mu,\ \pi_3\sharp\tilde\mu$ denote the marginal distribution of $X,\ Y,\ \theta$ respectively.

\begin{proposition}\label{mean-field-verification}
  Suppose 
  \begin{enumerate}
  \item Assumption \ref{assumption};
  \item $V$ is the classical solution to \eqref{HJB-MF}, $\partial_tV(\cdot)\in\mathcal C\big([0,T]\times\mathcal P_2(\mathbb R^d)\big)$, $V(t,\cdot)\in\mathcal C^2\big(\mathcal P_2(\mathbb R^d)\big)$ with jointly continuous and bounded derivatives.
  \end{enumerate}
 Then
  \begin{align}\label{mean-field-verification-1}
    V(0,\mathcal L_\xi)=\inf_{\theta\in\mathcal U^{ad}_0}\mathbb E\bigg[\int_0^TL\big(\mu_{(X_t,\theta_t)}\big)dt+U(\mu_{X_T})\bigg].
  \end{align}
\end{proposition}
\begin{proof}
 Consider the $X_t$ in \eqref{representative-particle} with its associated conditional distribution. According to \eqref{HJB-MF}, an application of It\^{o}'s formula (see \cite{Carmona2018-I}) yields
 \begin{align*}
  &dV(t,\mu_{X_t})=\bigg[\partial_tV(t,\mu_{X_t})+\frac12\int_{\mathbb R^d}{\rm tr}\big(D_x\partial_\mu V(t,\mu_{X_t},x)(\sigma\sigma^\top+\sigma_0\sigma^\top_0)\big)\mu(dx)\notag\\
  &\quad\quad+\frac12\int_{\mathbb R^d}\int_{\mathbb R^d}{\rm tr}\big(\partial^2_{\mu\mu}V(t,\mu_{X_t},x,y)\sigma_0\sigma^\top_0\big)\mu(dx)\mu(dy)+\tilde{\mathbb E}\big[\partial_\mu V\big(t,\mu_{X_t},\tilde X_t\big)\cdot\tilde\theta_t\big]\bigg]dt+d\mathcal M_t\notag\\
  &=-\bigg[\mathcal H\big(\tilde\mu_t\big)-\tilde{\mathbb E}\big[\partial_\mu V\big(t,\mu_{X_t},\tilde X_t\big)\cdot\tilde\theta_t\big]\bigg]dt+d\mathcal M_t,
 \end{align*}
where $\mathcal M_t$ is a martingale,
\begin{align*}
 \mu_{X_t}=\text{Law}\big(X_t|\mathcal F^{W_0}_t\big),\ \tilde\mu_t=\big(Id,\partial_\mu V(t,\mu_{X_t},\cdot)\big)\sharp\mu_{X_t},
\end{align*}
and $(\tilde X_t,\tilde\theta_t)$ is an independent copy of $(X_t,\theta_t)$ conditional on $\mathcal F^{W_0}_t$ while $\tilde{\mathbb E}$ is the expectation conditional on $\mathcal F^{W_0}_t$ taken with respect to $(\tilde X_t,\tilde\theta_t)$.

In view of the definition \eqref{example-H}, 
\begin{align*}
 \mathcal H\big(\tilde\mu_t\big)-\tilde{\mathbb E}\big[\partial_\mu V\big(t,\mu_{X_t},\tilde X_t\big)\cdot\tilde\theta_t\big]\leq L(\mu_{\tilde X_t,\tilde\theta_t}).
\end{align*}
Hence
\begin{align*}
 V(0,\mu_{X_0})\leq\mathbb E\bigg[V(T,\mu_{X_T})+\int_0^TL(\mu_{X_t,\tilde\theta_t})dt\bigg]=\mathbb E\bigg[U(\mu_{X_T})+\int_0^TL(\mu_{X_t,\theta_t})dt\bigg].
\end{align*}
On the other hand, consider the solution to
\begin{align*}
 dX^*_t=\partial_{\tilde\mu}\mathcal H^{(p)}\big(\tilde\mu_t,X^*_t,\partial_\mu V(t,\mu_{X^*_t},X^*_t)\big)dt +\sigma dW_t+\sigma_0 dW^0_t,\ X^*_0=\xi\in\mathcal F_0,
\end{align*}
and let
\begin{align*}
 \theta^*_t=\partial_{\tilde\mu}\mathcal H^{(p)}\big(\tilde\mu,X^*_t,\partial_\mu V(t,\mu_{X^*_t},X^*_t)\big).
\end{align*}
Then all the ``$\leq$'' above become ``$=$''. Hence $V$ is the value function and
\begin{align*}
 \hat\theta^*(t,x,\mu):=\partial_{\tilde\mu}\mathcal H^{(p)}\big((Id,\partial_\mu V(t,\mu,\cdot))\sharp\mu,x,\partial_\mu V(t,\mu,x)\big),\ (t,x,\mu)\in[0,T]\times\mathbb R^d\times\mathcal P_2(\mathbb R^d),
\end{align*}
is the optimal feedback function.
\end{proof}
\begin{remark}
 The variational representation \eqref{mean-field-verification-1} actually implies the uniqueness of solutions satisfying the assumptions in Proposition \ref{mean-field-verification}.
\end{remark}

\section{The generation of a local in time solution}\label{generation-of-solution}
This section is devoted to the generation of a local in time classical solution $V$ to \eqref{HJB-MF} via a probabilistic method. In the case  where only the distribution of the particle's position is involved, the value function $V(t,\mu)$ in \eqref{HJB-MF} can be generated directly via the flow of mean field FBSDE, see e.g. \cite{Carmona2018-I,Jean14,Mou2022}. However, such a way of generation is difficult to apply to \eqref{HJB-MF} when faced with the joint distribution of the particle's position and momentum in the HJB equation. As a remedy, we first generate the partial derivative $\partial_\mu V(t,\mu,x)$, then recover $V(t,\mu)$ with $\partial_\mu V(t,\mu,x)$. Given an initial distribution $\mu=\text{Law}(\xi)\in\mathcal P_2(\mathbb R)$ and an initial position $x\in\mathbb R$, the idea to generate $\partial_\mu V(t,\mu,x)$ is as follows. First, we formulate \eqref{1st-step} below using the mean field version of stochastic maximum principle and obtain the distribution flow $\tilde\mu_t$ of the optimal path-momentum pair $(X_t,Y_t)$; then we use $\tilde\mu_t$ to generate the decoupling field of \eqref{2nd-step} which is actually $\partial_\mu V(t,\mu,x)$. As for $V(t,\mu)$, we introduce its counterpart $V^\mu_t$ in \eqref{1st-step} then establish its connection with $\partial_\mu V(t,\mu,x)$ in Lemma \ref{deri-represent}. After the generation of $V(t,\mu)$, a refined result on local in time smooth solutions is given in Lemma \ref{short-time-well-posedness}.

Let us begin with the description on the probability basis. Let  $\Omega=\Omega_0\times\Omega_1\times\Omega_2$ and  $\Omega\ni\omega=(\omega_0,\omega_1,\omega_2)$ be the sample space and the sample. Define the following independent random variables and processes on $(\Omega,\mathcal F,\mathbb F,\mathbb P)$:
\begin{align*}
  \xi=\xi(\omega_0),\quad W_t=W(\omega_1,t),\quad W^0_t=W^0(\omega_2,t),\quad t\in[0,T],
\end{align*}
where $\xi$ has the distribution $\mu$, $\mathbb F=\big(\mathcal F_t\big)_{t\in[0,T]}$, $\mathcal F_t=\mathcal F^W_t\vee\mathcal F^{W^0}_t\vee\mathcal F^\xi$, and $W,$ $W^0$ are independent Brownian motions.

\noindent\textit{Fist step}. For $s\in[0,T]$, consider the mean field FBSDE
\begin{align}\label{1st-step}
  \left\{\begin{aligned}
    dX^\mu_s&=\partial_{\tilde\mu}\mathcal H^{(p)}\big(\mathbb P_{(X^\mu_s,Y^\mu_s)},X^\mu_s,Y^\mu_s\big)ds+\sigma dW_s+\sigma_0dW^0_s,\ X^\mu_t=\xi,\\
    dY^\mu_s&=-\partial_{\tilde\mu}\mathcal H^{(x)}\big(\mathbb P_{(X^\mu_s,Y^\mu_s)},X^\mu_s,Y^\mu_s\big)ds+Z^\mu_sdW_s+Z^{\mu,0}_sdW^0_s,\ Y^\mu_T=\partial_\mu U(\mathbb P_{X^\mu_T},X^\mu_T),\\
    V^\mu_s&=\mathbb E^{W_0}_s\bigg[U(\mathbb P_{X^\mu_T})+\int_s^T\bigg(\mathcal H(\mathbb P_{(X^\mu_u,Y^\mu_u)})-\partial_{\tilde\mu}\mathcal H^{(p)}\big(\mathbb P_{(X^\mu_u,Y^\mu_u)},X^\mu_u,Y^\mu_u\big)\cdot Y^\mu_u\bigg)du\bigg],
\end{aligned}\right.
\end{align}
where
\begin{align*}
 \mathbb P_{(X^\mu_u,Y^\mu_u)}:=\text{Law}\big((X^\mu_u,Y^\mu_u)\big|\mathcal F^{W_0}_u\big),
\end{align*}
$\mathbb E_s$ is the conditional expectation taken with respect to $\mathcal F_s$, and $\mathbb E^{W_0}_s$ is the conditional expectation taken with respect to $\mathcal F^{W_0}_s$. We note here that similar FBSDEs have been studied in \cite{Carmona2013}. Although $V^\mu$ in \eqref{1st-step} is decoupled from $(X^\mu,Y^\mu,Z^\mu,Z^{\mu,0})$, we understand it as part of the equation.

\noindent\textit{Second step}. After solving \eqref{1st-step}, we obtain $\tilde\mu_s:=\mathbb P_{(X^\mu_s,Y^\mu_s)}$. Consider the FBSDE on $\Omega_1\times\Omega_2$
\begin{align}\label{2nd-step}
  \left\{\begin{aligned}
    &dX^{\mu,x}_s=\partial_{\tilde\mu}\mathcal H^{(p)}\big(\tilde\mu_s,X^{\mu,x}_s,Y^{\mu,x}_s\big)ds+\sigma dW_s+\sigma_0dW^0_s,\ X^{\mu,x}_t=x,\\
    &dY^{\mu,x}_s=-\partial_{\tilde\mu}\mathcal H^{(x)}\big(\tilde\mu_s,X^{\mu,x}_s,Y^{\mu,x}_s\big)ds+Z^{\mu,x}_sdW_s+Z^{\mu,x,0}_sdW^0_s,\\
    &Y^{\mu,x}_T=\partial_\mu U(x\sharp\tilde\mu_T,X^{\mu,x}_T).
  \end{aligned}\right.
\end{align}
Note here that the FBSDE \eqref{2nd-step} is independent of $\mathcal F^\xi$.

Now we take $Y^{\mu,x}_t$ as the output and we may use the contraction method to show that the bounded decoupling field of \eqref{1st-step}$\sim$\eqref{2nd-step} uniquely exists for sufficiently small time horizon $T$. Denote such decoupling field by
\begin{align}\label{def-master-field}
 \tilde V(t,\mu,x):=Y^{\mu,x}_t,\quad V(t,\mu)=V(t,\mathbb P_{X^\mu_t}):=V^\mu_t.
\end{align} 
Next we analysis the relation between the solution to \eqref{1st-step} and \eqref{2nd-step} given their strong well-posedness.
\begin{lemma}
  Suppose that for any $\xi\in L^2(\Omega)$, \eqref{1st-step} admits a unique strong solution with $\mathbb P_{(X^\mu_u,Y^\mu_u)}$ replaced by fixed $\tilde\mu_t$ and suppose that for any $x_0\in\mathbb R^d$, \eqref{2nd-step} admits a unique strong solution where $Y^{\mu,x}_t$ is a measurable function of $x\in\mathbb R^d$. Then
  \begin{align}\label{decoupling-field}
    Y^\mu_s(\omega)=Y^\mu_s(\omega_0,\omega_1,\omega_2)=Y^{\mu,\xi(\omega_0)}_s(\omega_1,\omega_2).
  \end{align}
\end{lemma}
\begin{proof}
  For each $\omega_0\in\Omega_0$, according to \eqref{2nd-step},we may take $x_0:=\xi(\omega_0)$, then
  \begin{align*}
  \left\{\begin{aligned}
    &dX^{\mu,\xi(\omega_0)}_s=\partial_{\tilde\mu}\mathcal H^{(p)}\big(\tilde\mu_s,X^{\mu,\xi(\omega_0)}_s,Y^{\mu,\xi(\omega_0)}_s\big)ds+\sigma dW_s+\sigma_0dW^0_s,\ X^{\mu,\xi(\omega_0)}_t=\xi(\omega_0),\\
    &dY^{\mu,\xi(\omega_0)}_s=-\partial_{\tilde\mu}\mathcal H^{(x)}\big(\tilde\mu_s,X^{\mu,\xi(\omega_0)}_s,Y^{\mu,\xi(\omega_0)}_s\big)ds+Z^{\mu,\xi(\omega_0)}_sdW_s+Z^{\mu,\xi(\omega_0),0}_sdW^0_s,\\
    &Y^{\mu,\xi(\omega_0)}_T=\partial_\mu U\big(x\sharp\tilde\mu_T,X^{\mu,\xi(\omega_0)}_T\big).
  \end{aligned}\right.
\end{align*}
It's easy to see that $(X^{\mu,\xi(\omega_0)}_s,Y^{\mu,\xi(\omega_0)}_s)$ solves \eqref{1st-step} with fixed $\tilde\mu_s=\mathbb P_{(X^\mu_s,Y^\mu_s)}$. Hence we have \eqref{decoupling-field} by the assumption on the uniqueness.
\end{proof}
Let's further study the decoupling field in \eqref{def-master-field}. The next lemma is crucial in our approach to generating a local in time solution. This lemma depicts the relation between $\tilde V(t,\mu,x)$ and $V(t,\mu)$ in \eqref{def-master-field} so that we may proceed and show that $V(t,\mu)$ is indeed a solution of \eqref{HJB-MF}.
\begin{lemma}\label{deri-represent}
Suppose that the $\tilde V(t,\mu,x_1),\ V(t,\mu)$ in \eqref{def-master-field} are well-defined and that $\tilde V(t,\cdot)\in\mathcal C\big(\mathcal P_2(\mathbb R^d)\times\mathbb R^d\big)$, $V(t,\cdot)\in\mathcal C^1\big(\mathcal P_2(\mathbb R^d)\big)$ with bounded derivatives. Then
\begin{align}\label{deri-represent-1}
 \partial_\mu V(t,\mu,x_1)=\tilde V(t,\mu,x_1),\quad(t,\mu,x_1)\in[0,T]\times\mathcal P_2(\mathbb R^d)\times\mathbb R^d.
\end{align}
\end{lemma}
\begin{proof}
 Consider 
 \begin{align}\label{deri-FBSDE-0}
  \left\{\begin{aligned}
    &dX^{\varepsilon,\mu}_s=\partial_{\tilde\mu}\mathcal H^{(p)}\big(\mathbb P_{(X^{\varepsilon,\mu}_s,Y^{\varepsilon,\mu}_s)},X^{\varepsilon,\mu}_s,Y^{\varepsilon,\mu}_s\big)ds+\sigma dW_s+\sigma dW^0_s,\ X^{\varepsilon,\mu}_t=\xi+\varepsilon\varphi(\xi),\\
    &dY^{\varepsilon,\mu}_s=-\partial_{\tilde\mu}\mathcal H^{(x)}\big(\mathbb P_{(X^{\varepsilon,\mu}n_s,Y^{\varepsilon,\mu}_s)},X^{\varepsilon,\mu}_s,Y^{\varepsilon,\mu}_s\big)ds+Z^{\varepsilon,\mu}_sdW_s+Z^{\varepsilon,\mu,0}_sdW^0_s,\\
    &Y^{\varepsilon,\mu}_T=\partial_\mu U(\mathbb P_{X^{\varepsilon,\mu}_T},X^{\varepsilon,\mu}_T),\\
    &V^{\varepsilon,\mu}_s=\mathbb E^{W_0}_s\bigg[\int_s^T\bigg(\mathcal H(\mathbb P_{(X^{\varepsilon,\mu}_u,Y^{\varepsilon,\mu}_u)})-\partial_{\tilde\mu}\mathcal H^{(p)}\big(\mathbb P_{(X^{\varepsilon,\mu}_s,Y^{\varepsilon,\mu}_s)},X^{\varepsilon,\mu}_s,Y^{\varepsilon,\mu}_s\big)\cdot Y^{\varepsilon,\mu}_u\bigg)du\\
    &\qquad\qquad+U(\mathbb P_{X^{\varepsilon,\mu}_T})\bigg],
  \end{aligned}\right.
\end{align}
as well as
\begin{align}\label{deri-FBSDE-0-1}
  \left\{\begin{aligned}
    &dX^{\varepsilon,\mu,x}_s=\partial_{\tilde\mu}\mathcal H^{(p)}\big(\mathbb P_{(X^\varepsilon_s,Y^\varepsilon_s)},X^{\varepsilon,\mu,x}_s,Y^{\varepsilon,\mu,x}_s\big)ds+\sigma dW_s+\sigma_0dW^0_s,\ X^{\varepsilon,\mu,x}_t=x,\\
    &dY^{\varepsilon,\mu,x}_s=-\partial_{\tilde\mu}\mathcal H^{(x)}\big(\mathbb P_{(X^\varepsilon_s,Y^\varepsilon_s)},X^{\varepsilon,\mu,x}_s,Y^{\varepsilon,x}_s\big)ds+Z^{\varepsilon,\mu,x}_sdW_s+Z^{\varepsilon,\mu,x,0}_sdW^0_s,\\
     &Y^{\varepsilon,\mu,x}_T=\partial_\mu U(\mathbb P_{X^\varepsilon_T},X^{\varepsilon,\mu,x}_T).
  \end{aligned}\right.
\end{align}
In \eqref{deri-FBSDE-0}, $\varphi\in C^\infty_c(\mathbb R^d;\mathbb R^d)$ is a test function and we would like to formally take the $L^2$-derivative w.r.t. $\varepsilon$ at $\varepsilon=0$ in \eqref{deri-FBSDE-0}. The resulting limit in $L^2(\Omega)$, i.e., 
\begin{align*}
 (\hat X_s,\hat Y_s,\hat V_s):=L^2-\lim_{\varepsilon\to0}\frac1\varepsilon(X^{\varepsilon,\mu}_s-X^\mu_s,Y^{\varepsilon,\mu}_s-Y^\mu_s,V^\varepsilon_s-V^\mu_s),
\end{align*}
satisfies the linearized system
{\small\begin{align}\label{deri-FBSDE}
  \left\{\begin{aligned}
    d\hat X_s&=\tilde{\mathbb E}\big[\partial^2_{\tilde\mu\tilde\mu}\mathcal H^{(p)(x)}\big(\mathbb P_{(X^\mu_s,Y^\mu_s)},X^\mu_s,Y^\mu_s,\tilde X^\mu_s,\tilde Y^\mu_s\big)\tilde{\hat X}_s\big]ds+\tilde{\mathbb E}\big[\partial^2_{\tilde\mu\tilde\mu}\mathcal H^{(p)(p)}\big(\mathbb P_{(X^\mu_s,Y^\mu_s)},X^\mu_s,Y^\mu_s,\tilde  X^\mu_s,\tilde  Y^\mu_s\big)\tilde{\hat Y}_s\big]ds\\
    &\quad+\partial_x\partial_{\tilde\mu}\mathcal H^{(p)}\big(\mathbb P_{(X^\mu_s,Y^\mu_s)},X^\mu_s,Y^\mu_s\big) \hat X_sds+\partial_p\partial_{\tilde\mu}\mathcal H^{(p)}\big(\mathbb P_{(X^\mu_s,Y^\mu_s)},X^\mu_s,Y^\mu_s\big) \hat Y_sds,\quad\hat X_t=\varphi(\xi),\\
    d\hat Y_s&=-\tilde{\mathbb E}\big[\partial^2_{\tilde\mu\tilde\mu}\mathcal H^{(x)(x)}\big(\mathbb P_{(X^\mu_s,Y^\mu_s)},X^\mu_s,Y^\mu_s,\tilde  X^\mu_s,\tilde  Y^\mu_s\big)\tilde{\hat X}_s\big]ds-\tilde{\mathbb E}\big[\partial^2_{\tilde\mu\tilde\mu}\mathcal H^{(x)(p)}\big(\mathbb P_{(X^\mu_s,Y^\mu_s)},X^\mu_s,Y^\mu_s,\tilde  X^\mu_s,\tilde  Y^\mu_s\big)\tilde{\hat Y}_s\big]ds\\
    &\quad-\partial_x\partial_{\tilde\mu}\mathcal H^{(x)}\big(\mathbb P_{(X^\mu_s,Y^\mu_s)},X^\mu_s,Y^\mu_s\big) \hat X_sds-\partial_p\partial_{\tilde\mu}\mathcal H^{(x)}\big(\mathbb P_{(X^\mu_s,Y^\mu_s)},X^\mu_s,Y^\mu_s\big) \hat Y_sds+\hat Z_sdW_s+\hat Z^0_sdW^0_s,\\
    \hat Y_T&=\tilde{\mathbb E}\big[\partial^2_{\mu\mu}U(\mathbb P_{X^\mu_T},X^\mu_T,\tilde X^\mu_T)\tilde{\hat X}_T\big]+\partial_x\partial_\mu U(\mathbb P_{X^\mu_T},X^\mu_T)\hat X_T,
     \end{aligned}\right.
\end{align}}
as well as
\begin{align*}
 \hat V_s&=\mathbb E^{W_0}_s\bigg[\tilde{\mathbb E}\big[\partial_\mu U(\mathbb P_{X^\mu_T},{\tilde X^\mu}_T)\cdot\tilde{\hat X}_T\big]-\int_s^T\partial_{\tilde\mu}\mathcal H^{(p)}(\mathbb P_{(X^\mu_u,Y^\mu_u)},X^\mu_u,Y^\mu_u)\cdot\hat Y_udu\notag\\
 &\quad-\int_s^T\bigg(\tilde{\mathbb E}\big[\partial^2_{\tilde\mu\tilde\mu}\mathcal H^{(p)(x)}\big(\mathbb P_{(X^\mu_u,Y^\mu_u)},X^\mu_u,Y^\mu_u,\tilde  X^\mu_u,\tilde  Y^\mu_u\big)\cdot\tilde{\hat X}_u\big]\notag\\
 &\quad+\tilde{\mathbb E}\big[\partial^2_{\tilde\mu\tilde\mu}\mathcal H^{(p)(p)}\big(\mathbb P_{(X^\mu_u,Y^\mu_u)},X^\mu_u,Y^\mu_u,\tilde X^\mu_u,\tilde Y^\mu_u\big)\tilde{\hat Y}_u\big]\bigg)\cdot Y^\mu_udu\notag\\
 &+\int_s^T\bigg(\tilde{\mathbb E}\big[\partial_{\tilde\mu}\mathcal H^{(x)}\big(\mathbb P_{(X^\mu_u,Y^\mu_u)},\tilde X^\mu_u,\tilde Y^\mu_u\big)\cdot\tilde{\hat X}_u\big]+\tilde{\mathbb E}\big[\partial_{\tilde\mu}\mathcal H^{(p)}\big(\mathbb P_{(X^\mu_u,Y^\mu_u)},\tilde  X^\mu_u,\tilde  Y^\mu_u\big)\cdot\tilde{\hat Y}_u\big]\bigg)du\notag\\
 &-\int_s^T\bigg(\partial_x\partial_{\tilde\mu}\mathcal H^{(p)}(\mathbb P_{(X^\mu_u,Y^\mu_u)},X^\mu_u,Y^\mu_u)\hat X_u+\partial_p\partial_{\tilde\mu}\mathcal H^{(p)}(\mathbb P_{(X^\mu_u,Y^\mu_u)},X^\mu_u,Y^\mu_u)\hat Y_u\bigg)\cdot\hat Y_udu\bigg],
\end{align*}
where we use $(\tilde X^\mu_s,\tilde Y^\mu_s,\tilde{\hat X}_s,\tilde{\hat Y}_s)$ to denote the independent copy of $(X^\mu_s,Y^\mu_s,\hat X_s,\hat Y_s)$ conditional on $\mathcal F^{W_0}_t$ and use $\tilde{\mathbb E}$ to denote the expectation taken with respect to $(\tilde X^\mu_s,\tilde Y^\mu_s,\tilde{\hat X}_s,\tilde{\hat Y}_s)$ conditional on $\mathcal F^{W_0}_t$.

In view of the equations above, some algebra yields
\begin{align*}
 \hat V_s=\mathbb E^{W_0}_s\bigg[\hat X_T\cdot Y^\mu_T-\int_s^T\hat X_u\cdot dY^\mu_u-\int_s^TY^\mu_u\cdot d\hat X_u\bigg]=\mathbb E^{W_0}_s\big[\hat X_s\cdot Y^\mu_s\big].
 \end{align*}
Let $s=t$. The above equality and \eqref{decoupling-field} imply
\begin{align*}
 \lim_{\varepsilon\to0}\frac{V^\varepsilon_t-V_t}\varepsilon=\hat V_t=\mathbb E\big[\hat X_t\cdot Y^\mu_t\big]=\int_{\mathbb R}\varphi(x)\cdot Y^{\mu,x}_t\mu(dx),\ \forall\varphi\in C^\infty_c(\mathbb R^d;\mathbb R^d).
\end{align*}
We thus have \eqref{deri-represent-1} according to the characterization of Wasserstein derivatives.
\end{proof}
\begin{remark}
 In view of \eqref{def-master-field}$\sim$\eqref{deri-represent-1}, for $Y^\mu_t$ in \eqref{1st-step} we obtain
 \begin{align}\label{deri-represent-rmk}
  Y^\mu_t=\partial_\mu V(t,\mathcal L_\xi,\xi).
 \end{align}
\end{remark}

In a similar fashion to \eqref{deri-FBSDE}, taking the formal derivative w.r.t. $\varepsilon$ at $\varepsilon=0$ in \eqref{deri-FBSDE-0-1} and considering the corresponding linearized system, we obtain
{\small\begin{align}\label{deri-FBSDE-1}
  \left\{\begin{aligned}
    d\hat X^{x}_s&=\tilde{\mathbb E}\big[\partial^2_{\tilde\mu\tilde\mu}\mathcal H^{(p)(x)}\big(\mathbb P_{(X_s,Y_s)},X^{x}_s,Y^{x}_s,\tilde X_s,\tilde Y_s\big)\tilde{\hat X}_s\big]ds+\tilde{\mathbb E}\big[\partial^2_{\tilde\mu\tilde\mu}\mathcal H^{(p)(p)}\big(\mathbb P_{(X_s,Y_s)},X^{x}_s,Y^{x}_s,\tilde  X_s,\tilde  Y_s\big)\tilde{\hat Y}_s\big]ds\\
    &\quad+\partial_x\partial_{\tilde\mu}\mathcal H^{(p)}\big(\mathbb P_{(X_s,Y_s)},X^{x}_s,Y^{x}_s\big) \hat X^{x}_sds+\partial_p\partial_{\tilde\mu}\mathcal H^{(p)}\big(\mathbb P_{(X_s,Y_s)},X^{x}_s,Y^{x}_s\big) \hat Y^{x}_sds,\quad\hat X^{x}_t=0,\\
    d\hat Y^{x}_s&=-\tilde{\mathbb E}\big[\partial^2_{\tilde\mu\tilde\mu}\mathcal H^{(x)(x)}\big(\mathbb P_{(X_s,Y_s)},X^{x}_s,Y^{x}_s,\tilde  X_s,\tilde  Y_s\big)\hat X_s\big]ds-\tilde{\mathbb E}\big[\partial^2_{\tilde\mu\tilde\mu}\mathcal H^{(x)(p)}\big(\mathbb P_{(X_s,Y_s)},X^{x}_s,Y^{x}_s,\tilde  X_s,\tilde  Y_s\big) \hat Y_s\big]ds\\
    &\quad-\partial_x\partial_{\tilde\mu}\mathcal H^{(x)}\big(\mathbb P_{(X_s,Y_s)},X^{x}_s,Y^{x}_s\big) \hat X^{x}_sds-\partial_p\partial_{\tilde\mu}\mathcal H^{(x)}\big(\mathbb P_{(X_s,Y_s)},X^{x}_s,Y^{x}_s\big) \hat Y^{x}_sds+\hat Z^{x}_sdW_s+\hat Z^0_sdW^0_s,\\
    \hat Y^{x}_T&=\tilde{\mathbb E}\big[\partial^2_{\mu\mu}U(\mathbb P_{X_T},X^{x}_T,\tilde X_T)\tilde{\hat X}_T\big]+\partial_x\partial_\mu U(\mathbb P_{X_T},X^{x}_T)\hat X^{x}_T.
\end{aligned}\right.
\end{align}}
Similar to \eqref{deri-represent-1}, we have the following representation, which will be used in Section \ref{prior-on-2nd}.
\begin{lemma}\label{deri-represent-2-0}
Suppose that $V(t,\mu)$ in \eqref{def-master-field} is well-defined and $V(t,\cdot)\in\mathcal C^2\big(\mathcal P_2(\mathbb R^d)\big)$ with continuous and bounded derivatives, then $\hat Y^{x}_t$ in \eqref{deri-FBSDE-1} admits the following representation
\begin{align}\label{deri-represent-2}
 \hat Y^{x}_t=\int_{\mathbb R}\partial^2_{\mu\mu}V(t,\mu,x,y)\varphi(y)\mu(dy),\quad\mu=\mathcal L_\xi.
\end{align}
\end{lemma}
\begin{proof}
 In view of Lemma \ref{deri-represent}, for a test function $\varphi\in C^\infty_c(\mathbb R^d;\mathbb R^d)$,
 \begin{align*}
  \hat Y^{x}_t&=\lim_{\varepsilon\to0+}\frac{Y^{\varepsilon,\mu,x}_t-Y^{\mu,x}_t}\varepsilon=\lim_{\varepsilon\to0+}\varepsilon^{-1}\bigg[\partial_\mu V\big(t,(Id+\varepsilon\varphi)\sharp\mu,x\big)-\partial_\mu V(t,\mu,x)\bigg]\notag\\
  &=\int_{\mathbb R}\partial^2_{\mu\mu}V(t,\mu,x,y)\varphi(y)\mu(dy).
 \end{align*}
\end{proof}
The next lemma shows that $V$ solves the master equation \eqref{HJB-MF} whenever it is sufficiently smooth.
\begin{lemma}\label{veri-V}
 Let $V(t,\mu)$ be well-defined as in \eqref{def-master-field}. Suppose that $\partial_tV(\cdot)\in\mathcal C\big([T-\tilde c,T]\times\mathcal P_2(\mathbb R^d)\big)$, $V(t,\cdot)\in\mathcal C^2\big(\mathcal P_2(\mathbb R^d)\big)$ on $t\in[T-\tilde c,T]$ with jointly continuous and bounded derivatives. Then $V(t,\mu)$ satisfies \eqref{HJB-MF} on $t\in[T-\tilde c,T]$.
\end{lemma}
\begin{proof}
 In view of \eqref{1st-step}, \eqref{def-master-field} and Lemma \ref{deri-represent},
 \begin{align*}
 \left\{\begin{aligned}
    X_s&=\xi+\int_t^s\partial_{\tilde\mu}\mathcal H^{(p)}\big(\mathbb P_{(X_u,Y_u)},X_u,Y_u\big)du+\sigma W_s-\sigma W_t+\sigma_0W^0_s-\sigma_0W^0_t\\
    &=\xi+\int_t^s\partial_{\tilde\mu}\mathcal H^{(p)}\big(\mathbb P_{(X_u,\partial_\mu V(u,\mathbb P_{X_u},X_u)))},X_u,\partial_\mu V(s,\mathbb P_{X_u},X_u))\big)du\\
    &\quad+\sigma W_s-\sigma W_t+\sigma_0W^0_s-\sigma_0W^0_t\\
    V_s&=\mathbb E^{W_0}_s\bigg[U(\mathbb P_{X_T})+\int_s^T\bigg(\mathcal H(\mathbb P_{(X_u,Y_u)})-\partial_{\tilde\mu}\mathcal H^{(p)}\big(\mathbb P_{(X_u,Y_u)},X_u,Y_u\big)\cdot Y_u\bigg)du\bigg]\\
  =&\mathbb E^{W_0}_s\bigg[U(\mathbb P_{X_T})+\int_s^T\bigg(\mathcal H(\mathbb P_{(X_u,\partial_\mu V(u,\mathbb P_{X_u},X_u))})-\partial_{\tilde\mu}\mathcal H^{(p)}(u)\cdot\partial_\mu V(u,\mathbb P_{X_u},X_u)\bigg)du\bigg],\end{aligned}\right.
 \end{align*}
and
{\small\begin{align*}
 &\quad V(t,\mathbb P_{X_t})=\mathbb E_t[V_t]\\
 &=\mathbb E_t\bigg[U(\mathbb P_{X_T})+\int_t^T\bigg(\mathcal H(\mathbb P_{(X_u,\partial_\mu V(u,\mathbb P_{X_u},X_u))})-\partial_{\tilde\mu}\mathcal H^{(p)}(u)\cdot\partial_\mu V(u,\mathbb P_{X_u},X_u)\bigg)du\bigg]\notag\\
 &=\mathbb E_t\bigg[V(t+\Delta t,\mathbb P_{X_{t+\Delta t}})+\int_t^{t+\Delta t}\bigg(\mathcal H(\mathbb P_{(X_u,\partial_\mu V(u,\mathbb P_{X_u},X_u))})-\partial_{\tilde\mu}\mathcal H^{(p)}(u)\cdot\partial_\mu V(u,\mathbb P_{X_u},X_u)\bigg)du\bigg],
\end{align*}}
where
\begin{align*}
 \partial_{\tilde\mu}\mathcal H^{(p)}(u):=\partial_{\tilde\mu}\mathcal H^{(p)}\big(\mathbb P_{(X_u,\partial_\mu V(u,\mathbb P_{X_u},X_u)))},X_u,\partial_\mu V(s,\mathbb P_{X_u},X_u))\big).
\end{align*}
Now we may apply It\^{o}'s formula to the above and find that $V$ satisfies \eqref{HJB-MF}.
\end{proof}
In view of Lemma \ref{veri-V}, the global well-posedness of \eqref{HJB-MF} boils down to the global well-posedness of \eqref{1st-step}$\sim$\eqref{2nd-step} as well as the differentiability of the corresponding decoupling field in \eqref{def-master-field}. Towards that end, we will establish a priori estimates on the decoupling field in the next two sections.

The previous results in this section relies on the assumption that there exists a sufficiently smooth decoupling field in \eqref{def-master-field}. At the end of this section, we explain how \eqref{1st-step}$\sim$\eqref{2nd-step} generate a smooth local in time solution. In fact, thanks to Lemma \ref{deri-represent}, the existence of a smooth local in time solution can now be established by the contraction method used in \cite{Carmona2018-I,Jean14,Mou2022} and so on. In view of Lemma \ref{deri-represent}, we may take the formal derivation w.r.t. the $x$ and $\mu$ in \eqref{1st-step}$\sim$\eqref{2nd-step} to  obtain the flow that generates $\partial_x\partial_\mu V(t,\mu,x)=\breve Y^{\mu,x}_t$ as follows,
\begin{align}\label{2nd-order-generation}
  \left\{\begin{aligned}
    &d\breve X^{\mu,x}_s=\big[\partial_x\partial_{\tilde\mu}\mathcal H^{(p)}\big(\tilde\mu_s,X^{\mu,x}_s,Y^{\mu,x}_s\big)\breve X^{\mu,x}_s+\partial_p\partial_{\tilde\mu}\mathcal H^{(p)}\big(\tilde\mu_s,X^{\mu,x}_s,Y^{\mu,x}_s\big)\breve Y^{\mu,x}_s\big]ds,\\
     &\breve X^{\mu,x}_t=1,\\
    &d\breve Y^{\mu,x}_s=-\big[\partial_x\partial_{\tilde\mu}\mathcal H^{(x)}\big(\tilde\mu_s,X^{\mu,x}_s,Y^{\mu,x}_s\big)\breve X^{\mu,x}_s+\partial_p\partial_{\tilde\mu}\mathcal H^{(x)}\big(\tilde\mu_s,X^{\mu,x}_s,Y^{\mu,x}_s\big)\breve Y^{\mu,x}_s\big]ds\\
    &\qquad\qquad+\breve Z^{\mu,x}_sdW_s+\breve Z^{\mu,x,0}_sdW^0_s,\\
    &\breve Y^{\mu,x}_T=\partial_x\partial_\mu U(x\sharp\tilde\mu_T,X^{\mu,x}_T)\breve X^{\mu,x}_T.
  \end{aligned}\right.
\end{align}
Similarly, we have the flow that generates $\partial^2_{\mu\mu}V(t,\mu,x_1,x_2)=\check Y^{\mu,x_1,x_2}_t$ with the following
\begin{align}\label{2nd-order-generation-2}\left\{\begin{aligned}
 d\check X^{\mu,\xi,x_2}_s&=\tilde{\mathbb E}\big[\partial^2_{\tilde\mu\tilde\mu}\mathcal H^{(p)(x)}\big(\tilde\mu_s,X^\mu_s,Y^\mu_s,\tilde X^\mu_s,\tilde Y^\mu_s\big)(\tilde{\check X}^{\mu,\tilde\xi,x_2}_s+\tilde{\breve X}^{\mu,x_2}_s)\big]ds\\
 &\quad+\tilde{\mathbb E}\big[\partial^2_{\tilde\mu\tilde\mu}\mathcal H^{(p)(p)}\big(\tilde\mu_s,X^\mu_s,Y^\mu_s,\tilde X^\mu_s,\tilde Y^\mu_s\big)(\tilde{\check Y}^{\mu,\tilde\xi,x_2}_s+\tilde{\breve Y}^{\mu,x_2}_s)\big]ds\\
 &\quad+\big[\partial_x\partial_{\tilde\mu}\mathcal H^{(p)}\big(\tilde\mu_s,X^\mu_s,Y^\mu_s\big)\check X^{\mu,\xi,x_2}_s+\partial_p\partial_{\tilde\mu}\mathcal H^{(p)}\big(\tilde\mu_s,X^\mu_s,Y^\mu_s\big)\check Y^{\mu,\xi,x_2}_s\big]ds,\\
 \check X^{\mu,\xi,x_2}_t&=0,\\
d\check Y^{\mu,\xi,x_2}_s&=-\tilde{\mathbb E}\big[\partial^2_{\tilde\mu\tilde\mu}\mathcal H^{(x)(x)}\big(\tilde\mu_s,X^\mu_s,Y^\mu_s,\tilde X^\mu_s,\tilde Y^\mu_s\big)(\tilde{\check X}^{\mu,\tilde\xi,x_2}_s+\tilde{\breve X}^{\mu,x_2}_s)\big]ds\\
&\quad-\tilde{\mathbb E}\big[\partial^2_{\tilde\mu\tilde\mu}\mathcal H^{(x)(p)}\big(\tilde\mu_s,X^\mu_s,Y^\mu_s,\tilde X^\mu_s,\tilde Y^\mu_s\big)(\tilde{\check Y}^{\mu,\tilde\xi,x_2}_s+\tilde{\breve Y}^{\mu,x_2}_s)\big]ds\\
&\quad-\big[\partial_x\partial_{\tilde\mu}\mathcal H^{(x)}\big(\tilde\mu_s,X^\mu_s,Y^\mu_s\big)\check X^{\mu,\xi,x_2}_s+\partial_p\partial_{\tilde\mu}\mathcal H^{(x)}\big(\tilde\mu_s,X^\mu_s,Y^\mu_s\big)\check Y^{\mu,\xi,x_2}_s\big]ds\\
&\quad+\check Z^{\mu,\xi,x_2}_sdW_s+\check Z^{\mu,\xi,x_2,0}_sdW^0_s\\
\check Y^{\mu,\xi,x_2}_T&=\tilde{\mathbb E}\big[\partial^2_{\mu\mu}U(x\sharp\tilde\mu_T,X^\mu_T,\tilde X^\mu_T)(\tilde{\check X}^{\mu,\tilde\xi,x_2}_T+\tilde{\breve X}^{\mu,x_2}_T)\big]+\partial_x\partial_\mu U(x\sharp\tilde\mu_T,X^\mu_T)\check X^{\mu,\xi,x_2}_T,
\end{aligned}\right.
\end{align}
as well as 
\begin{align}\label{2nd-order-generation-1}\left\{\begin{aligned}
&d\check X^{\mu,x_1,x_2}_s=\tilde{\mathbb E}\big[\partial^2_{\tilde\mu\tilde\mu}\mathcal H^{(p)(x)}\big(\tilde\mu_s,X^{\mu,x_1}_s,Y^{\mu,x_1}_s,\tilde X^\mu_s,\tilde Y^\mu_s\big)(\tilde{\check X}^{\mu,\tilde\xi,x_2}_s+\tilde{\breve X}^{\mu,x_2}_s)\big]ds\\
&\quad+\big[\partial_x\partial_{\tilde\mu}\mathcal H^{(p)}\big(\tilde\mu_s,X^{\mu,x_1}_s,Y^{\mu,x_1}_s\big)\check X^{\mu,x_1,x_2}_s+\partial_p\partial_{\tilde\mu}\mathcal H^{(p)}\big(\tilde\mu_s,X^{\mu,x_1}_s,Y^{\mu,x_1}_s\big)\check Y^{\mu,x_1,x_2}_s\big]ds\\
&\quad+\tilde{\mathbb E}\big[\partial^2_{\tilde\mu\tilde\mu}\mathcal H^{(p)(p)}\big(\tilde\mu_s,X^{\mu,x_1}_s,Y^{\mu,x_1}_s,\tilde X^\mu_s,\tilde Y^\mu_s\big)(\tilde{\check Y}^{\mu,\tilde\xi,x_2}_s+\tilde{\breve Y}^{\mu,x_2}_s)\big]ds,\ \check X^{\mu,x_1,x_2}_t=0,\\
&d\check Y^{\mu,x_1,x_2}_s=-\tilde{\mathbb E}\big[\partial^2_{\tilde\mu\tilde\mu}\mathcal H^{(x)(x)}\big(\tilde\mu_s,X^{\mu,x_1}_s,Y^{\mu,x_1}_s,\tilde X^\mu_s,\tilde Y^\mu_s\big)(\tilde{\check X}^{\mu,\tilde\xi,x_2}_s+\tilde{\breve X}^{\mu,x_2}_s)\big]ds\\
&\quad-\tilde{\mathbb E}\big[\partial^2_{\tilde\mu\tilde\mu}\mathcal H^{(x)(p)}\big(\tilde\mu_s,X^{\mu,x_1}_s,Y^{\mu,x_1}_s,\tilde X^\mu_s,\tilde Y^\mu_s\big)(\tilde{\check Y}^{\mu,\tilde\xi,x_2}_s+\tilde{\breve Y}^{\mu,x_2}_s)\big]ds\\
&\quad-\big[\partial_x\partial_{\tilde\mu}\mathcal H^{(x)}\big(\tilde\mu_s,X^{\mu,x_1}_s,Y^{\mu,x_1}_s\big)\check X^{\mu,x_1,x_2}_s+\partial_p\partial_{\tilde\mu}\mathcal H^{(x)}\big(\tilde\mu_s,X^{\mu,x_1}_s,Y^{\mu,x_1}_s\big)\check Y^{\mu,x_1,x_2}_s\big]ds\\
&\quad+\check Z^{\mu,x_1,x_2}_sdW_s+\check Z^{\mu,x_1,x_2,0}_sdW^0_s,\\
&\check Y^{\mu,x_1,x_2}_T=\tilde{\mathbb E}\big[\partial^2_{\mu\mu}U(x\sharp\tilde\mu_T,X^{\mu,x_1}_T,\tilde X^\mu_T)(\tilde{\check X}^{\mu,\tilde\xi,x_2}_T+\tilde{\breve X}^{\mu,x_2}_T)\big]\\
&\quad+\partial_x\partial_\mu U(x\sharp\tilde\mu_T,X^{\mu,x_1}_T)\check X^{\mu,x_1,x_2}_T,
 \end{aligned}\right.
\end{align}
where $(\tilde X^\mu_s,\tilde Y^\mu_s,\tilde{\check X}^{\mu,\tilde\xi,x_2}_s,\tilde{\check Y}^{\mu,\tilde\xi,x_2}_s)$ is an independent copy of $(X^\mu_s,Y^\mu_s,{\check X}^{\mu,\xi,x_2}_s,{\check Y}^{\mu,\xi,x_2}_s)$ conditional on $\mathcal F^{W_0}_t$ and $\tilde{\mathbb E}$ is the expectation taken with respect to $(\tilde X^\mu_s,\tilde Y^\mu_s,\tilde{\check X}^{\mu,\tilde\xi,x_2}_s,\tilde{\check Y}^{\mu,\tilde\xi,x_2}_s)$ conditional on $\mathcal F^{W_0}_t$.

Using the aforementioned contraction method, we may show that $\partial_x\partial_\mu V(t,\mu,x)$ and $\partial^2_{\mu\mu}V(t,\mu,x_1,x_2)$, which are part of the decoupling field of \eqref{1st-step}$\sim$\eqref{2nd-step} and \eqref{2nd-order-generation}$\sim$\eqref{2nd-order-generation-1}, both exist and are continuous for a short time horizon. Inductively, we may further take formal derivatives in \eqref{2nd-order-generation}$\sim$\eqref{2nd-order-generation-1} and show the existence of the continuous higher order derivatives of $V$ for a sufficiently small time horizon. As for $\partial_t V$, we have by the flow property of \eqref{1st-step} that
\begin{align}
 V(t+\delta,\mu)=\mathbb E\bigg[U(\mathbb P_{X_{T-\delta}})+\int_s^{T-\delta}\bigg(\mathcal H(\mathbb P_{(X_u,Y_u)})-\partial_{\tilde\mu}\mathcal H^{(p)}\big(\mathbb P_{(X_u,Y_u)},X_u,Y_u\big)\cdot Y_u\bigg)du\bigg],
\end{align}
where $(X_u,Y_u)$ is from \eqref{1st-step}. Hence
\begin{align*}
 &\quad\partial_tV(t,\mu)=\lim_{\delta\to0+}\frac{V(t+\delta,\mu)-V(t,\mu)}\delta\notag\\
 &=\mathbb E\bigg[\tilde{\mathbb E}\big[\frac{\sigma^2}2\partial_x\partial_\mu U(\mathbb P_{X_T},\tilde X_T)+\partial_\mu U(\mathbb P_{X_T},\tilde X_T)\partial_{\mu}\mathcal H^{(p)}\big(\mathbb P_{(X_T,Y_T)},\tilde X_T,\tilde Y_T\big)\big]\\
 &\quad+\mathcal H(\mathbb P_{(X_T,Y_T)})-\partial_{\tilde\mu}\mathcal H^{(p)}\big(\mathbb P_{(X_T,Y_T)},X_T,Y_T\big)\cdot Y_T\bigg]\notag\\
 &\quad+\mathbb E\tilde{\mathbb E}\hat{\mathbb E}\big[\frac{\sigma^2_0}2\partial^2_{\mu\mu}U(\mathbb P_{X_T},\tilde X_T,\hat X_T)\big].
\end{align*}
In view of the above representation, we may also derive the regularity of $\partial_t V(t,\cdot)$
by inductively taking the derivatives with respect to the initial data in the flow \eqref{1st-step}$\sim$\eqref{2nd-step} and \eqref{2nd-order-generation}$\sim$\eqref{2nd-order-generation-1}.

We conclude the previous discussion in the following lemma. Since the proof is basically the same as those in \cite{Carmona2018-I,Jean14,Mou2022}, we omit it here.
\begin{lemma}\label{short-time-well-posedness}
Suppose Assumption \ref{assumption}. Consider the mean field FBSDE system
\begin{align}\label{1st-step-1}
  \left\{\begin{aligned}
    &dX_s=\partial_{\tilde\mu}\mathcal H^{(p)}\big(\mathbb P_{(X_s,Y_s)},X_s,Y_s\big)ds+\sigma dW_s+\sigma_0 dW^0_s,\ X_t=\xi,\\
    &dY_s=-\partial_{\tilde\mu}\mathcal H^{(x)}\big(\mathbb P_{(X_s,Y_s)},X_s,Y_s\big)ds+Z_sdW_s+Z^0_sdW^0_s,\ Y_T=\partial_\mu K(\mathbb P_{X_T},X_T),\\
    V_s&=\mathbb E^{W_0}_s\bigg[K(\mathbb P_{X_T})+\int_s^T\bigg(\mathcal H(\mathbb P_{(X_u,Y_u)})-\partial_{\tilde\mu}\mathcal H^{(p)}\big(\mathbb P_{(X_u,Y_u)},X_u,Y_u\big)\cdot Y_u\bigg)du\bigg],
 \end{aligned}\right.
\end{align}
Then there exists a  sufficiently small constant $\tilde c$ such that for any $T-t<\tilde c$, the mean field FBSDE system \eqref{1st-step-1} admits a unique decoupling field $V(t,\mu)$ satisfying $V(t,\cdot,)\in\mathcal C^4\big(\mathcal P_2(\mathbb R^d)\big)$, $\partial_tV(t,\cdot)\in\mathcal C^2\big(\mathcal P_2(\mathbb R^d)\big)$, $\partial_tV(\cdot)\in\mathcal C\big([0,T]\times\mathcal P_2(\mathbb R^d)\big)$ where all derivatives are jointly continuous and bounded, as well as
\begin{align*}
 Y_s=\partial_\mu V(s,\mathbb P_{X_s},X_s),\quad V_s=V(s,\mathbb P_{X_s}).
\end{align*}
Moreover, the constant $\tilde c$ as well as $|\partial^{i_1}_{x_1}\cdots\partial^{i_j}_{x_j}\partial^{j_1}_\mu\partial^{j_2}_t V|_\infty$ ($0\leq j_1\leq 4,\ 0\leq j_2\leq1,\ 0\leq i_1,\ldots,i_j\leq4,\ 0\leq i_1+\cdots+i_j+j_1+2j_2\leq 4$) depends only on
 \begin{align}\label{tilde c-dependence}
  \sum_{\substack{0\leq k\leq 4,0\leq l_1,\ldots,l_k\leq4\\0\leq\tilde l_1,\ldots,\tilde l_k\leq4\\0\leq l_1+\tilde l_1\cdots+l_k+\tilde l_k+k\leq 4}}|\partial^{l_1}_{x_1}\partial^{\tilde l_1}_{p_1}\cdots\partial^{l_k}_{x_k}\partial^{\tilde l_k}_{p_k}\partial^k_{\tilde\mu}\mathcal H|_\infty+\sum_{\substack{0\leq k\leq 4,0\leq l_1,\ldots,l_k\leq4\\0\leq l_1+\cdots+l_k+k\leq 4}}|\partial^{l_1}_{x_1}\cdots\partial^{l_k}_{x_k}\partial^k_\mu K|_\infty.
 \end{align}
\end{lemma}
Here in the above lemma , for $0\leq k\leq 4$, we have denoted the $k$-th intrinsic derivatives by $\partial^k_\mu V(t,\mu,x_1,\ldots,x_k)$, $\partial^k_{\tilde\mu}\mathcal H(\tilde\mu,x_1,p_1,\ldots,x_k,p_k)$ and $\partial^k_\mu K(\mu,x_1,\ldots,x_k)$.

\section{A priori estimates on $V_N$ and $V$}\label{prior-on-2nd}
According to Lemma \ref{short-time-well-posedness}, given sufficiently smooth parameters, we are able to establish the existence of the decoupling field to \eqref{1st-step} for a sufficiently small time horizon $t\in[T-\tilde c,T]$ in which $V(t,\cdot)\in\mathcal C^4\big(\mathcal P_2(\mathbb R^d)\big)$. Here the constant $\tilde c$ depends on \eqref{tilde c-dependence} with $K=U$. Our goal in this section is to establish the global in time a priori estimates based on such smoothness assumptions. To be exact, the estimates in this section depend only on $\mathcal H$ and $U$ rather than $\tilde c$.
\subsection{A priori estimates on the first order derivatives of $V_N$}
Consider the HJB equation \eqref{HJB-N} for the $N$-particle system. Given Assumption \ref{assumption}, for each $N\geq1$, one may show that \eqref{HJB-N} admits a unique classical solution $V_N$ (see e.g. \cite{Fleming2006}). Thanks to Assumption \ref{assumption}, we may show higher regularity of $V_N$.
\begin{lemma}\label{differentiability}
 Suppose Assumption \ref{assumption} and $\sigma>0$. Then \eqref{HJB-N} admits a unique classical solution satisfying $V_N$, $\partial_{x_i}V_N$, $\partial^2_{x_ix_j}V_N\in C^{1,2+\gamma}\big([0,T)\times(\mathbb R^d)^N\big)\cap C\big([0,T]\times(\mathbb R^d)^N\big)$ for some $\gamma>0$ and $1\leq i,j\leq N$.
\end{lemma}
\begin{proof}
In view of the variational representation in \eqref{optim.-value-function}, there exists a constant $C>0$ such that
\begin{align*}
 |V_N(t,x)-V_N(t,\tilde x)|\leq C|x-\tilde x|,\quad t\in[0,T],\ x,\tilde x\in(\mathbb R^d)^N.
\end{align*}
In other words, the weak derivatives of $V_N$ is bounded. Consider a larger constant $C$ that exceeds $|\partial_{\tilde\mu}\mathcal H|_\infty$ as well as the bound of the weak derivatives of $V_N$. For such constant $C$, it can be shown that for $x\in(\mathbb R^d)^N$ and $|p|\leq C$,
\begin{align}\label{confine-hamiltonian}
 H_N(x,p)=\inf_{|\theta|\leq C}\big\{L^C_N\big(\mu^N_{(x,\theta)}\big)+\theta p\big\}:=\inf_{|\theta|\leq C}\sup_{|\tilde p|\leq C}\big\{H_N(x,\tilde p)-\theta\cdot\tilde p+\theta\cdot p\big\},\\
 \mu^N_{(x,\theta)}=\frac1N\sum_{i=1}^N\delta_{(x_i,\theta_i)}.\notag
\end{align}
Now that $\theta$ and $\tilde p$ in \eqref{confine-hamiltonian} are confined to a compact set and $\partial_{x_i}V(t,\mathbf x)$ is bounded by $C$, we may replace the $L_N$ in \eqref{optim.} with the $L^C_N$ in \eqref{confine-hamiltonian}, then show that the resulting value function $V_N$ in \eqref{optim.-value-function} is the weak solution (in the distributional sense) to \eqref{HJB-N} by referring to Theorem 4.4.3, Theorem 4.7.2 and Theorem 4.7.4 in \cite{Krylov1980}. Moreover, $V_N$ and its weak derivatives $\partial_tV_N$, $\partial_{x_i}V_N$, $\partial^2_{x_ix_j}V_N$, $1\leq i,j\leq N$ are locally bounded with polynomial growth in $x$ and uniformly in $t\in[0,T)$. Given the above facts, we may rewrite \eqref{HJB-N} as
 \begin{align}\label{regularity}
    \partial_tV_N(t,x)+\frac{\sigma^2}2\sum_{i,j=1}^N\partial^2_{x_ix_j}V_N(t,x)+\frac{\sigma^2_0}2\sum_{i=1}^N\partial^2_{x_ix_i}V_N(t,x)=g(t,x),
  \end{align}
  where
   \begin{align*}
  g(t,x):=-\mathcal H\bigg(\frac1N\sum_{i=1}^N\delta_{(x_i,N\partial_{x_i}V_N(t,x))}\bigg),
\end{align*}
is locally Lipschitz in $x$. Moreover, we have from Corollary 4.7.8 of \cite{Krylov1980} that for $0<\gamma<1$ and $x$ from a compact set, $\partial_{x_i}V_N(t,x)$ is $\frac\gamma2$-H\"older with respect to $t$. Hence for $x$ from a compact subset of $\mathbb R^N$, $g(t,x)$ is Lipschitz in $x$ and H\"older in $t$. Viewing $V_N$ as the solution to PDE \eqref{regularity} with constant coefficients as well as the same terminal condition as \eqref{HJB-N}, we have from standard results that $\partial_tV_N$, $\partial_{x_i}V_N$, $\partial^2_{x_ix_j}V_N\in C^{\frac\gamma2,\gamma}_{loc}\left([0,T)\times(\mathbb R^d)^N\right)$, $1\leq i,j\leq N$.

To show higher regularity, we may take $\partial_{x_i}$ in \eqref{regularity} as well as the corresponding terminal condition. Then similar analysis to the above yields $\partial_{x_i}V_N\in C^{1,2+\gamma}\big([0,T)\times(\mathbb R^d)^N\big)$. Further taking $\partial_{x_j}$ in the equation on $\partial_{x_i}V_N$ yields $\partial^2_{x_ix_j}V_N\in C^{1,2+\gamma}\big([0,T)\times(\mathbb R^d)^N\big)$ and the proof is complete.
\end{proof}
Next we establish the first uniform estimate on $N$-particle systems in \eqref{HJB-N}.
\begin{proposition}\label{uniform-1-1}
  Suppose Assumption \ref{assumption} and $\sigma>0$. Then there exists a constant $C$ depending only on $|\partial_{\tilde\mu}\mathcal H|_\infty+|\partial_\mu U|_\infty$ (independent of $N$) such that
  \begin{align}\label{uniform-1}
    \big|N\partial_{x_i}V_N(t,x)\big|\leq C,\quad i=1,\ldots,N,\ (t,x)\in[0,T]\times(\mathbb R^d)^N.
  \end{align}
  \end{proposition}
  \begin{proof}
    We lose nothing by only considering $t=0$. According to Lemma \ref{differentiability}, equation \eqref{HJB-N} has a classical solution. Moreover, we may use a standard verification argument to show the existence of an optimizer (see e.g. \cite{Fleming2006}). In view of the stochastic maximum principle, the optimal path and the adjoint process satisfy
    \begin{align}\label{uniform-1-1-rep}
      \left\{\begin{aligned}
        &dX_i(t)=\partial_{\tilde\mu}\mathcal H^{(p)}\bigg(\frac1N\sum_{j=1}^N\delta_{(X_j(t),NY_j(t))},X_i(t),NY_i(t)\bigg)dt+\sigma dW_i(t)+\sigma_0 dW_0(t),\\
        &dNY_i(t)=-\partial_{\tilde\mu}\mathcal H^{(x)}\bigg(\frac1N\sum_{j=1}^N\delta_{(X_j(t),NY_j(t))},X_i(t),NY_i(t)\bigg)dt+\sum_{j=0}^NZ_{ij}(t)dW_i(t),\\
        &X_i(0)=x_i,\quad NY_i(T)=\partial_\mu U\bigg(\frac1N\sum_{j=1}^N\delta_{X_j(T)},X_i(T)\bigg),\quad i=1,2,\ldots,N,
      \end{aligned}\right.
    \end{align}
    where $Y_i(t)=\partial_{x_i}V_N\big(t,X(t)\big)$ ($1\leq i\leq N$). By assumption $\partial_{\tilde\mu}\mathcal H^{(x)}$, $\partial_\mu U$ are bounded, hence we have \eqref{uniform-1} by taking expectation in the above.
  \end{proof}
Next we infer the Lipschitz continuity of $V$ from the uniform estimates above.
\begin{proposition}\label{1st-regularity}
  Suppose
\begin{enumerate}
  \item Assumption \ref{assumption} and $\sigma>0$;
  \item $V$ is the classical solution to \eqref{HJB-MF}. For some $\tilde\delta>0$, $\partial_tV(\cdot)\in\mathcal C\big([T-\tilde\delta,T]\times\mathcal P_2(\mathbb R^d)\big)$, $V(t,\cdot)\in\mathcal C^2\big(\mathcal P_2(\mathbb R^d)\big)$ with jointly continuous and bounded derivatives on $t\in[T-\tilde\delta,T]$.
  \end{enumerate}
  Then it holds for $\mu,\ \nu\in\mathcal P_2(\mathbb R^d)$ that
  \begin{align}\label{1st-regularity-1}
    |V(t,\mu)-V(t,\nu)|\leq C\mathcal W_1(\mu,\nu),\quad t\in[T-\tilde\delta,T],
  \end{align}
  where the constant $C$ depends only on $|\partial_{\tilde\mu}\mathcal H|_\infty+|\partial_\mu U|_\infty$.
\end{proposition}

\begin{proof}
  According to the law of large number, we may choose empirical measures
  \begin{align*}
    \mu_N=\frac1N\sum_{i=1}^N\delta_{x_{N,i}},\quad\nu_N=\frac1N\sum_{i=1}^N\delta_{y_{N,i}},
  \end{align*}
  such that
  \begin{align*}
    \mathcal W_1(\mu_N,\mu)+\mathcal W_1(\nu_N,\nu)\to0\quad\text{as}\quad N\to+\infty.
  \end{align*}
  Up to a permutation, we have
  \begin{align*}
    \mathcal W_1(\mu_N,\nu_N)=\frac1N\sum_{i=1}^N|x_{N,i}-y_{N,i}|.
  \end{align*}
  In view of Proposition \ref{uniform-1-1}, we have
  \begin{align*}
     |V_N(t,x_1,\ldots,x_N)-V_N(t,y_1,\ldots,y_N)|\leq\frac CN\sum_{i=1}^N|x_i-y_i|=C\mathcal W_1(\mu_N,\nu_N).
  \end{align*}
  Moreover, by Proposition \ref{prop-propagation-1},
  \begin{align*}
    |V(t,\mu_N)-V(t,\nu_N)|\leq C\mathcal W_1(\mu_N,\nu_N)+\frac{\tilde C}N.
  \end{align*}
  Here the constant $C$ is from Proposition \ref{uniform-1-1} thus depends only on $\mathcal H$ and $U$, while the constant $\tilde C$ is from \eqref{propagation-1}. Now sending $N$ to infinity in the above and noticing that we have assumed the continuity of $V$, we have shown \eqref{1st-regularity-1}.
\end{proof}

\subsection{A priori estimates on the second order derivatives of $V_N$}\label{The second order estimates on $V_N$}
Now we proceed to a priori estimates on the second order derivatives of $V_N$. In view of Lemma \ref{differentiability}, $V_N$ is sufficiently smooth. We may take $\partial^2_{x_kx_l}$ in \eqref{HJB-N} and obtain a PDE system on $\partial^2_{x_kx_l}V_N,\ 1\leq k,l\leq N$. For the ease of notation we make the following abbreviation, for $1\leq k,l\leq N$,
\begin{align}\label{abbreviation}
  V^k_N=V^k_N(t,x):=D_{x_k}V_N(t,x_1,\ldots,x_N),\quad V^{kl}_N=V^{kl}_N(t,x):=D^2_{x_kx_l}V_N(t,x_1,\ldots,x_N).
\end{align}
After some basic algebra, for $(t,x)\in[0,T)\times\mathbb R^N$,\ \\
{\small\begin{align}\label{Fey-Kac-0-0}
  \left\{\begin{aligned}
    &\partial_tV^k_N+\frac{\sigma^2}2\sum_{i=1}^N\partial^2_{x_ix_i}V^k_N+\frac{\sigma^2_0}2\sum_{i,j=1}^N\partial^2_{x_ix_j}V^k_N+\sum_{i=1}^N\partial_{\tilde\mu}\mathcal H^{(p)}(\tilde\mu_t,x_i,NV^i_N)\partial_{x_i}V^k_N+\frac1N\partial_{\tilde\mu}\mathcal H^{(x)}(\tilde\mu_t,x_k,NV^k_N)=0,\\
    &\quad V^k_N(T,x)=\frac1N\partial_\mu U\bigg(\frac1N\sum_{i=1}^N\delta_{x_i},x_k\bigg),
  \end{aligned}\right.
\end{align}}
where $\tilde\mu_t=\frac1N\sum_{i=1}^N\delta_{(x_i,NV^i_N(t,x))}$. For $k\neq l$:
{\small\begin{align}\label{Fey-Kac}
  \left\{\begin{aligned}
    &\quad\partial_t V^{kl}_N+\frac{\sigma^2}2\sum_{i=1}^N\partial^2_{x_ix_i}V^{kl}_N+\frac{\sigma^2_0}2\sum_{i,j=1}^N\partial^2_{x_ix_j}V^{kl}_N+\sum_{i=1}^N\partial_{\tilde\mu}\mathcal H^{(p)}(\tilde\mu_t,x_i,NV^i_N)\partial_{x_i}V^{kl}_N\\
    &\quad+\frac1N\sum_{i=1}^N\partial^2_{\tilde\mu\tilde\mu}\mathcal H^{(p)(x)}(\tilde\mu_t,x_i,NV^i_N,x_l,NV^l_N)V^{ik}_N+\frac1N\sum_{i=1}^N\partial^2_{\tilde\mu\tilde\mu}\mathcal H^{(x)(p)}(\tilde\mu_t,x_k,NV^k_N,x_i,NV^i_N)V^{il}_N\\
    &\quad+\partial_x\partial_{\tilde\mu}\mathcal H^{(p)}(\tilde\mu_t,x_l,NV^l_N)V^{lk}_N+\partial_p\partial_{\tilde\mu}\mathcal H^{(x)}(\tilde\mu_t,x_k,NV^k_N)V^{kl}_N+N\sum_{i=1}^N\partial_p\partial_{\tilde\mu}\mathcal H^{(p)}(\tilde\mu_t,x_i,NV^i_N)V^{ik}_NV^{il}_N\\
    &\quad+\sum_{i,j=1}^N\partial^2_{\tilde\mu\tilde\mu}\mathcal H^{(p)(p)}(\tilde\mu_t,x_i,NV^i_N,x_j,NV^j_N)V^{ik}_NV^{jl}_N+\frac1{N^2}\partial^2_{\tilde\mu\tilde\mu}\mathcal H^{(x)(x)}(\tilde\mu_t,x_k,NV^k_N,x_l,NV^l_N)=0,\\
    &\quad V^{kl}_N(T,x)=\frac1{N^2}\partial^2_{\mu\mu}U\bigg(\frac1N\sum_{i=1}^N\delta_{x_i},x_k,x_l\bigg).
  \end{aligned}\right.
\end{align}}
For $k=l$:
{\scriptsize\begin{align}\label{Fey-Kac-1}
  \left\{\begin{aligned}
    &\quad\partial_tV^{kk}_N+\frac{\sigma^2}2\sum_{i=1}^N\partial^2_{x_ix_i}V^{kk}_N+\frac{\sigma^2_0}2\sum_{i,j=1}^N\partial^2_{x_ix_j}V^{kk}_N+\sum_{i=1}^N\partial_{\tilde\mu}\mathcal H^{(p)}(\tilde\mu_t,x_i,NV^i_N)\partial_{x_i}V^{kk}_N+\partial_p\partial_{\tilde\mu}\mathcal H^{(x)}(\tilde\mu_t,x_k,NV^k_N)V^{kk}_N\\
    &\quad+\partial_x\partial_{\tilde\mu}\mathcal H^{(p)}(\tilde\mu_t,x_k,NV^k_N)V^{kk}_N+\frac1N\sum_{i=1}^N\partial^2_{\tilde\mu\tilde\mu}\mathcal H^{(p)(x)}(\tilde\mu_t,x_i,NV^i_N,x_k,NV^k_N)V^{ik}_N\\
    &\quad+\frac1N\sum_{i=1}^N\partial^2_{\tilde\mu\tilde\mu}\mathcal H^{(x)(p)}(\tilde\mu_t,x_k,NV^k_N,x_j,NV^j_N)V^{jk}_N+N\sum_{i=1}^N\partial_p\partial_{\tilde\mu}\mathcal H^{(p)}(\tilde\mu_t,x_i,NV^i_N)V^{ik}_NV^{ik}_N\\
    &\quad+\sum_{i,j=1}^N\partial^2_{\tilde\mu\tilde\mu}\mathcal H^{(p)(p)}(\tilde\mu_t,x_i,NV^i_N,x_j,NV^j_N)V^{ik}_NV^{jk}_N+\frac1{N^2}\partial^2_{\tilde\mu\tilde\mu}\mathcal H^{(x)(x)}(\tilde\mu_t,x_k,NV^k_N,x_k,NV^k_N)+\frac1N\partial_x\partial_{\tilde\mu}\mathcal H^{(x)}(\tilde\mu_t,x_k,NV^k_N)=0,\\
    &\quad V^{kk}_N(T,x)=\frac1{N^2}\partial^2_{\mu\mu}U\bigg(\frac1N\sum_{i=1}^N\delta_{x_i},x_k,x_k\bigg)+\frac1N\partial_x\partial_{\mu}U\bigg(\frac1N\sum_{i=1}^N\delta_{x_i},x_k\bigg).
  \end{aligned}\right.
\end{align}}
Here, with a little abuse of notation, the differential operator is applied to each entry of the matrix. For an $\mathbb R^{d^2\times d^2}$ matrix $A=(a_{ij})_{1\leq i,j\leq d}$ with $a_{ij}\in\mathbb R^{d\times d}$, we define ${\rm Tr}A:=\sum_{i=1}^da_{ii}\in\mathbb R^{d\times d}$. The following is the first key a priori estimate on the second order derivatives.
\begin{lemma}\label{prop-eigen}
  Suppose
  \begin{enumerate}
   \item Assumption \ref{assumption} and $\sigma>0$;
   \item For some $\tilde\delta>0$, the system \eqref{Fey-Kac}$\sim$\eqref{Fey-Kac-1} admits bounded classical solutions $V^{ij}_N\in C^{1,2}([T-\tilde\delta,T]\times\big(\mathbb R^d)^N\big)$, $1\leq i,j\leq N$.
  \end{enumerate}
  Then there exists a constant $C$ depending only on
 \begin{align}\label{C-depend}
  |\partial^2_{\tilde\mu\tilde\mu}\mathcal H|_\infty+|\partial_x\partial_{\tilde\mu}\mathcal H^{(x)}|_\infty+|\partial_p\partial_{\tilde\mu}\mathcal H^{(p)}|_\infty+|\partial^2_{\mu\mu}U|_\infty+|\partial_x\partial_\mu U|_\infty,
 \end{align}
independent of $\tilde\delta$ such that
  \begin{align}\label{eigen}
    0\leq\sum_{i,j=1}^N\xi^\top_iV^{ij}_N(t,x)\xi_j\leq\frac CN\sum_{i=1}^N|\xi_i|^2,\ \xi=(\xi_1,\ldots,\xi_N)\in(\mathbb R^d)^N,\ (t,x)\in[T-\tilde\delta,T]\times(\mathbb R^d)^N.
  \end{align}
\end{lemma}
\begin{proof}
Without loss of generality, we show \eqref{eigen} with $t=T-\tilde\delta$. Towards that end, we apply the nonlinear Feynman-Kac representation formula to \eqref{Fey-Kac}$\sim$\eqref{Fey-Kac-1}. Consider the underlying forward process
\begin{align}\label{Fey-Kac-0}
 \left\{\begin{aligned}&dX^i_N(t)=\partial_{\tilde\mu}\mathcal H^{(p)}\big(\tilde\mu(t),X^i_N(t),NV^i_N(t,X_N(t))\big)dt+\sigma dW^i_t+\sigma_0 dW^0_t,\\
 &X^i_N(T-\tilde\delta)=x_i,\quad1\leq i\leq N,\end{aligned}\right.
\end{align}
where
\begin{align*}
 \tilde\mu(s):=\frac1N\sum_{i=1}^N\delta_{\big(X^i_N(s),NV^i_N(s,X_N(s))\big)},\ s\in[T-\tilde\delta,T],
\end{align*}
and for $t\in[T-\tilde\delta,T],\ 1\leq k,l\leq N$:
\begin{align}\label{Fey-Kac-0-1}
  Y^{kl}_N(t):=V^{kl}_N(t,X_N(t)).
\end{align}
According to the assumptions in the current proposition, $Y_N$ is bounded. Therefore we have from \eqref{Fey-Kac}$\sim$\eqref{Fey-Kac-1} that
 \begin{align}\label{Fey-Kac-3}
    Y_N(t)=\mathbb E_t\bigg[U_N(T)+\int_t^T\bigg(H^{xx}_N(s)+A_N(s)Y_N(s)+Y_N(s)A_N(s)^\top+Y_N(s)H^{pp}_N(s)Y_N(s)\bigg)ds\bigg].
  \end{align}
  Here for $1\leq k,l\leq N$, $T-\tilde\delta\leq s\leq T$,
  \begin{align}\label{parameters}
  \left\{\begin{aligned}
    &\big(U_N(T)\big)_{kl}:=\frac1{N^2}\partial^2_{\mu\mu}U\big(x\sharp\tilde\mu(T),X^k_N(T),X^l_N(T)\big)+\frac{\delta_{kl}}N\partial_x\partial_\mu U\big(x\sharp\tilde\mu(T),X^k_N(T)\big),\\
    &\big(A_N(s)\big)_{kl}:=\frac1N\partial^2_{\tilde\mu\tilde\mu}\mathcal H^{(p)(x)}\big(\tilde\mu(s),X^k_N(s),NV^k_N(s),X^l_N(s),NV^l_N(s)\big)\\
    &\qquad\qquad\qquad+\delta_{kl}\partial_x\partial_{\tilde\mu}\mathcal H^{(p)}\big(\tilde\mu(s),X^k_N(s),NV^k_N(s)\big),\\
    &\big(H^{pp}_N(s)\big)_{kl}:=\partial^2_{\tilde\mu\tilde\mu}\mathcal H^{(p)(p)}\big(\tilde\mu(s),X^k_N(s),NV^k_N(s),X^l_N(s),NV^l_N(s)\big)\\
    &\qquad\qquad\qquad+\delta_{kl}N\partial_p\partial_{\tilde\mu}\mathcal H^{(p)}\big(\tilde\mu(s),X^k_N(s),NV^k_N(s)\big),\\
    &\big(H^{xx}_N(s)\big)_{kl}:=\frac1{N^2}\partial^2_{\tilde\mu\tilde\mu}\mathcal H^{(x)(x)}\big(\tilde\mu(s),X^k_N(s),NV^k_N(s),X^l_N(s),NV^l_N(s)\big)\\
    &\qquad\qquad\qquad+\frac1N\delta_{kl}\partial_x\partial_{\tilde\mu}\mathcal H^{(x)}\big(\tilde\mu(s),X^k_N(s),NV^k_N(s)\big)\end{aligned}\right.
  \end{align}
  and $V^k_N(s):=V^k_N\big(s,X_N(s)\big)$. According to Assumption \ref{assumption},
  \begin{align}\label{eigen-1-0}
   U_N(T), H^{xx}_N(s)\geq0,\quad H^{pp}_N\leq 0.
  \end{align}
Consider
\begin{align*}
 d\Phi(t)=\Phi(t)\big[A_N(t)+\frac12Y_N(t)H^{pp}_N(t)\big]dt,\quad t\in[T-\tilde\delta,T],\ \Phi(T-\tilde\delta)=I_{dN}\in\mathbb R^{dN\times dN}.
\end{align*}
Since $Y_N(t)$ is bounded, $\Phi(t)$ is also bounded. Hence
\begin{align*}
    \Phi(t)Y_N(t)\Phi(t)^\top=\mathbb E_t\bigg[\Phi(T)U_N(T)\Phi(T)^\top+\int_t^T\Phi(s)H^{xx}_N(s)\Phi(s)^\top ds\bigg],
  \end{align*}
  which implies that $Y_N(t)\geq0$.
  
  On the other hand, consider any $\alpha\in(\mathbb R^d)^N$ satisfying $|\alpha|=1$. According to \eqref{Fey-Kac-3} and \eqref{eigen-1-0},
    \begin{align}\label{eigen-1}
    0&\leq\alpha^\top Y_N(t)\alpha\\
    &\leq\mathbb E_t\bigg[\alpha^\top U_N(T)\alpha+\int_t^T\bigg(\alpha^\top H^{xx}_N(s)\alpha+\alpha^\top A_N(s)Y_N(s)\alpha+\alpha^\top Y_N(s)A_N(s)\alpha\bigg)ds\bigg].\notag
  \end{align}
According to \eqref{parameters}, there exists a constant $C$ depending only on \eqref{C-depend} such that
\begin{align*}
 \alpha^\top U_N(T)\alpha+\int_t^T\alpha^\top H^{xx}_N(s)\alpha ds\leq\frac CN.
\end{align*}
Moreover, \eqref{parameters} also yields 
\begin{align*}
 \alpha^\top A_N(s)Y_N(s)\alpha\leq|A_N(s)^\top\alpha|\cdot|Y_N(s)\alpha|\leq C|\alpha|\cdot|Y_N(s)\alpha|\leq C\sup_{|\beta|=1}\beta^\top Y_N(s)\beta.
\end{align*}
In view of \eqref{eigen-1} and the estimates above,
\begin{align*}
 \sup_{|\beta|=1}\beta^\top Y_N(t)\beta\leq\frac CN+C\mathbb E_t\bigg[\int_t^T\bigg(\sup_{|\beta|=1}\beta^\top Y_N(s)\beta\bigg)ds\bigg].
\end{align*}
Then we have by Gr{\"o}nwall's inequality that
\begin{align*}
 \sup_{T-\tilde\delta\leq t\leq T}\mathbb E\bigg[\sup_{|\beta|=1}\beta^\top Y_N(t)\beta\bigg]\leq\frac CN,
\end{align*}
which implies \eqref{eigen}.
\end{proof}
In Lemma \ref{prop-eigen} we have assumed that $V^{ij}_N$ is bounded for some time interval. Next we show that this is indeed the case for at least a small time horizon.
\begin{lemma}\label{continuation-1}
 Suppose
 \begin{enumerate}
  \item Assumption \ref{assumption} and $\sigma>0$;
  \item The system \eqref{Fey-Kac}$\sim$\eqref{Fey-Kac-1} admits bounded solutions on $t\in[T_0,T]$, in particular,
 \begin{align}\label{continuation-1-assumption}
  |V^{ij}_N(T_0,x)|\leq\frac CN,\quad x\in(\mathbb R^d)^N,
 \end{align}
 where $C$ is from \eqref{eigen}. 
 \end{enumerate}
Then there exists a positive constant $\tilde\delta_0$ depending only on $C$ in \eqref{eigen} such that
 \begin{align}\label{continuation-1-result}
  |V^{ij}_N(t,x)|\leq\frac CN,\quad (t,x)\in[T_0-\tilde\delta_0,T_0]\times(\mathbb R^d)^N.
 \end{align}
\end{lemma}
\begin{proof}
For any $\varepsilon>0$, we may find a smooth cutoff function satisfying
\begin{align*}
 C_c^\infty\big((\mathbb R^d)^N\big)\ni\psi_\varepsilon:\ (\mathbb R^d)^N\to[0,1],\ \psi_\varepsilon=1\ \text{on}\ B(0,\varepsilon^{-1})\subset(\mathbb R^d)^N,
\end{align*}
which slowly decays to $0$ in such a way that
\begin{align*}
 |\nabla^2_{xx}\psi_\varepsilon|\leq\varepsilon|x|^{-2},\quad|\nabla_x\psi_\varepsilon|\leq\varepsilon|x|^{-1}.
\end{align*}
Recall the representation in \eqref{confine-hamiltonian} and define
\begin{align*}
 H^\varepsilon_N(x,p)=\inf_{|\theta|\leq C}\big\{\psi_\varepsilon(x)L^C_N\big(\mu^N_{(x,\theta)}\big)+\theta p\big\}.
\end{align*}
Next we replace the $H_N(x,p)$ and $\mathcal H\big(\frac1N\sum_{i=1}^N\delta_{(x_i,p_i)}\big)$ in \eqref{HJB-N} and \eqref{Fey-Kac}$\sim$\eqref{Fey-Kac-1} with $H^\varepsilon_N(x,p)$ and denote by $V^\varepsilon_N$ and $V^{\varepsilon,kj}_N$ ($1\leq k,j\leq N$) the corresponding solutions. In view of Theorem 4.7.2 and Theorem 4.7.4 in \cite{Krylov1980}, $D^2_{x_ix_j}V^\varepsilon_N$ $(1\leq i,j\leq N)$ are bounded. According to the same mentioned references, $D^2_{x_ix_j}V^\varepsilon_N$ are locally bounded uniformly in $\varepsilon$ because of the growth rate of $L^C_N$. Note that $V^\varepsilon_N$ and $V_N$ are both value functions of certain optimal control problems, it is then easy to see that $(V^\varepsilon_N$, $\nabla V^\varepsilon_N$, $\nabla^2_x V^\varepsilon_N)$ converges to $(V_N$, $\nabla V_N$, $\nabla^2_x V_N)$ as $\varepsilon$ goes to $0+$. We may also show that $D^2_{x_ix_j}V^\varepsilon_N\in C^{1,2+\gamma}\big([0,T)\times(\mathbb R^d)^N\big)$ via the same method as in Proposition \ref{differentiability}.  Since $D^2_{x_ix_j}V^\varepsilon_N$ is bounded, we may get an analogy to \eqref{Fey-Kac-3}: 
\begin{align}\label{Fey-Kac-3-var}
    Y^\varepsilon_N(t)&=\mathbb E_t\bigg[\int_t^{T_0}\bigg(H^{\varepsilon,xx}_N(s)+A^\varepsilon_N(s)Y^\varepsilon_N(s)+Y^\varepsilon_N(s)A^\varepsilon_N(s)^\top+Y^\varepsilon_N(s)H^{\varepsilon,pp}_N(s)Y^\varepsilon_N(s)\bigg)ds\notag\\
   &\qquad\quad+V^{\varepsilon,xx}_N(T_0)\bigg],
  \end{align}
  where
 \begin{align*}
  &V^{\varepsilon,xx}_N(T_0)=\nabla^2_{xx}V^\varepsilon_N\big(T_0,X^\varepsilon_N(T_0)\big),\\
  &H^{\varepsilon,xx}_N(s)=\nabla^2_{xx}H^\varepsilon_N\big(X^\varepsilon_N(s),\nabla_x V^\varepsilon_N(s,X^\varepsilon_N(s))\big),\\
  &A^\varepsilon_N(s)=\nabla^2_{xp}H^\varepsilon_N\big(X^\varepsilon_N(s),\nabla_x V^\varepsilon_N(s,X^\varepsilon_N(s))\big),\\
  &H^{\varepsilon,pp}_N(s)=\nabla^2_{pp}H^\varepsilon_N\big(X^\varepsilon_N(s),\nabla_x V^\varepsilon_N(s,X^\varepsilon_N(s))\big),
 \end{align*}
and for $1\leq i\leq N$,
\begin{align*}
 \left\{\begin{aligned}&dX^{\varepsilon,i}_N(t)=\partial_{p_i}H^\varepsilon_N\big(X^\varepsilon_N(s),\nabla_x V^\varepsilon_N(s,X^\varepsilon_N(s))\big)dt+\sigma dW_i(t)+\sigma dW_0(t),\\
 &X^i_N(T-\tilde\delta_0)=x_i,\end{aligned}\right.
\end{align*}
as well as
\begin{align*}
 Y^\varepsilon_N(t)=\nabla^2_{xx}V^\varepsilon_N\big(t,X_N(t)\big),\quad t\in[T-\tilde\delta_0,T].
\end{align*}
Here in the above $\tilde\delta_0$ is to be determined. In view of \eqref{continuation-1-assumption}$\sim$\eqref{Fey-Kac-3-var}, using a contraction method similar to that in Proposition 3.10 of \cite{Huafu24}, we can show the existence of a small time duration $\tilde\delta_0$ such that
\begin{align*}
 \big|D^2_{x_ix_j}V^\varepsilon_N\big|\leq\tilde C,\quad1\leq i,j\leq N,\ t\in[T_0-\tilde\delta_0,T_0],
\end{align*}
where both $\tilde\delta_0$, $\tilde C$ depend only on $C$ from \eqref{eigen} and are independent of $\varepsilon$. Therefore, we may send $\varepsilon$ to $0+$ and get
\begin{align*}
 \big|D^2_{x_ix_j}V_N\big|\leq\tilde C,\quad1\leq i,j\leq N,\ t\in[T_0-\tilde\delta_0,T_0].
\end{align*}
In other words, the system \eqref{Fey-Kac}$\sim$\eqref{Fey-Kac-1} admit bounded solutions on $t\in[T_0-\tilde\delta_0,T]$ with $\tilde\delta_0>0$. Then Lemma \ref{prop-eigen} further yields \eqref{continuation-1-result}.
\end{proof}
Combining Lemma \ref{prop-eigen} and Lemma \ref{continuation-1}, we can use an induction argument and prove the main result Theorem \ref{eigen-0-1}.
\begin{proof}[Proof of Theorem \ref{eigen-0-1}]
It suffices to show \eqref{eigen1}. In view of Lemma \ref{continuation-1} as well as the terminal condition, we could set $T_0=T$ in \eqref{continuation-1} and get 
 \begin{align}\label{continuation-2}
  |V^{ij}_N(t,x)|\leq\frac CN,\quad1\leq i,j\leq N,
 \end{align}
for $(t,x)\in[T-\tilde\delta_0,T]\times\mathbb R^N$. Then we may set $T_0=T-\tilde\delta_0$ in \eqref{continuation-1} and get \eqref{continuation-2} for $(t,x)\in[T-2\tilde\delta_0,T]\times(\mathbb R^d)^N$. Repeat the process above we can obtain \eqref{continuation-2} for $(t,x)\in[T-n\tilde\delta_0,T]\times(\mathbb R^d)^N$, $n=1,2,\ldots$. Since $\tilde\delta_0$ is positive, we have \eqref{eigen1} by Lemma \ref{prop-eigen} when $n$ is sufficiently large.
\end{proof}

\subsection{A priori estimates on the second order derivatives of $V$}
In this section we first show that $V(t,\mu)$ is both displacement convex and sub-concave in $\mu$ whenever $V(t,\cdot)\in\mathcal C^2\big(\mathcal P_2(\mathbb R^d)\big)$ for $t\in[T-\tilde\delta,T]$. Then we further show a priori estimates on $D_x\partial_\mu V$ and $\partial^2_{\mu\mu}V$ which are independent of $\tilde\delta$ .
\begin{lemma}\label{finite-convex}
  Suppose Assumption \ref{assumption} and $\sigma>0$. Let $C$ be the constant from \eqref{eigen1}, then
 \begin{align*}
 \frac CN\sum_{i=1}^N|x_i|^2-V_N(t,x),\ V_N(t,x)
 \end{align*}
 are both convex in $x=(x_1,\ldots,x_N)\in(\mathbb R^d)^N$.
\end{lemma}
\begin{proof}
 This follows directly from \eqref{eigen1} and the definition.
\end{proof}

\begin{lemma}\label{limit-convex}
Suppose
\begin{enumerate}
  \item Assumption \ref{assumption} and $\sigma>0$;
  \item $V$ is the classical solution to \eqref{HJB-MF}. For some $\tilde\delta>0$, $\partial_tV(\cdot)\in\mathcal C\big([T-\tilde\delta,T]\times\mathcal P_2(\mathbb R^d)\big)$, $V(t,\cdot)\in\mathcal C^2\big(\mathcal P_2(\mathbb R^d)\big)$ with jointly continuous and bounded derivatives on $t\in[T-\tilde\delta,T]$.
 \end{enumerate}
 Let $C$ be the constant from \eqref{eigen1}, then for $t\in[T-\tilde\delta,T]$
 \begin{align*}
 C\int_{\mathbb R}x^2\mu(dx)-V(t,\mu),\ V(t,\mu)
 \end{align*}
 are both displacement convex with respect to $\mu\in\mathcal P_2(\mathbb R^d)$.
\end{lemma}
\begin{proof}
 For $\mu_i\in\mathcal P_2(\mathbb R^d)$, $i=1,2$, consider the random variable $(\xi_1,\xi_2)$ taking values in $\mathbb R^{2d}$, whose marginal distribution satisfies $\text{Law}(\xi_i)=\mu_i$. Then there exists
 \begin{align*}
  \tilde\mu_N=\frac1N\sum_{i=1}^N\delta_{(x_i,\tilde x_i)}\ \longrightarrow\ \text{Law}(\xi_1,\xi_2)\ \text{in}\ \mathcal P_2(\mathbb R^{2d})\ \text{as}\ N\to+\infty.
 \end{align*}
 In particular, we have
 \begin{align*}
  \frac1N\sum_{i=1}^N\delta_{x_i}\ \longrightarrow\ \text{Law}(\xi_1),\quad\frac1N\sum_{i=1}^N\delta_{\tilde x_i}\ \longrightarrow\ \text{Law}(\xi_2)\ \text{in}\ \mathcal P_2(\mathbb R^d)\ \text{as}\ N\to+\infty,
 \end{align*}
as well as
\begin{align*}
 \frac1N\sum_{i=1}^N\delta_{\lambda x_i+(1-\lambda)\tilde x_i}\ \longrightarrow\ \text{Law}(\lambda\xi_1+(1-\lambda)\xi_2)=:\hat\mu\ \text{in}\ \mathcal P_2(\mathbb R^d)\ \text{as}\ N\to+\infty.
\end{align*}
According to Lemma \ref{finite-convex},
\begin{align*}
 &\quad\lambda\bigg(C\int_{\mathbb R}x^2\tilde\mu_N(dxd\tilde x)-V_N(t,\mathbf x)\bigg)+(1-\lambda)\bigg(C\int_{\mathbb R}\tilde x^2\tilde\mu_N(dxd\tilde x)-V_N(t,\tilde{\mathbf x})\bigg)\notag\\
 &\geq C\int_{\mathbb R}(\lambda x+(1-\lambda)\tilde x)^2\tilde\mu_N(dxd\tilde x)-V_N\big(t,\lambda\mathbf x+(1-\lambda)\tilde{\mathbf x}\big),
\end{align*}
where $\mathbf x=(x_1,\ldots,x_N)$ and $\tilde{\mathbf x}=(\tilde x_1,\ldots,\tilde x_N)$. In view of Proposition \ref{prop-propagation-1} and Proposition \ref{1st-regularity}, we may send the above $N$ to infinity and use the convergence of $\tilde\mu_N$ and $V_N$ to obtain
\begin{align*}
 &\quad\lambda\bigg(C\int_{\mathbb R}x^2\mu_1(dx)-V(t,\mu_1)\bigg)+(1-\lambda)\bigg(C\int_{\mathbb R}x^2\mu_2(dx)-V(t,\mu_2)\bigg)\notag\\
 &\geq C\int_{\mathbb R}x^2\hat\mu(dx)-V(t,\hat\mu).
\end{align*}
The displacement convexity of $V(t,\mu)$ can be shown in the same way.
\end{proof}

\begin{lemma}\label{key-second-order}
 Suppose
\begin{enumerate}
  \item Assumption \ref{assumption} and $\sigma>0$;
  \item $V$ is the classical solution to \eqref{HJB-MF}. For some $\tilde\delta>0$, $\partial_tV(\cdot)\in\mathcal C\big([T-\tilde\delta,T]\times\mathcal P_2(\mathbb R^d)\big)$, $V(t,\cdot)\in\mathcal C^2\big(\mathcal P_2(\mathbb R^d)\big)$ with jointly continuous and bounded derivatives on $t\in[T-\tilde\delta,T]$.
 \end{enumerate}
 Let $C$ be the constant from \eqref{eigen1}. Then for any $\alpha\in(\mathbb R^d)^N$, $|\alpha|=1$, and 
 \begin{align*}
  \mu_N=\frac1N\sum_{i=1}^N\delta_{x_i},
 \end{align*}
it holds that
 \begin{align}\label{secondorder-descrete-0}
  -C\leq\frac1N\sum_{i,j=1}^N\alpha^\top_i\partial^2_{\mu\mu}V(t,\mu_N,x_i,x_j)\alpha_j\leq C,\quad0\leq\partial_x\partial_\mu V(t,\mu_N,x_i)\leq CI_d,\quad t\in[T-\tilde\delta,T].
 \end{align}
\end{lemma}
\begin{proof}
Let us first show that 
\begin{align}\label{secondorder-descrete-1}
 0\leq\partial_x\partial_\mu V(t,\mu_N,x_1)\leq CI_d.
\end{align}
In view of the law of large number, it is easy to construct a series of empirical measures 
\begin{align*}
 \gamma_m=\frac1m\sum_{i=1}^m\delta_{y_{m,i}},\quad m=1,2,\ldots
\end{align*}
satisfying
\begin{enumerate}
\item $y_{m,1}=x_1$;
\item $y_{m,1},y_{m,2},\ldots,y_{m,m}$ are mutually different from each other;
\item their limit satisfying
\begin{align*}
 \lim_{m\to+\infty}\gamma_m=\mu_N.
\end{align*}
\end{enumerate}
In fact, in item 2, if some of the elements coincide, then one can perturb them and get them different without impacting item 3. According to Lemma \ref{limit-convex}, in view of the definition of displacement convex, define 
 \begin{align*}
 g(a):=V\bigg(t,\frac1m\sum_{i=1}^m\delta_{y_i+az_i}\bigg),\quad\tilde g(a):=\frac1m\sum_{i=1}^m\big(y_i+az_i\big)^2-V\bigg(t,\frac1m\sum_{i=1}^m\delta_{y_i+az_i}\bigg),
\end{align*}
 then 
\begin{align*}
 g''(0),\ \tilde g''(0)\geq0.
\end{align*}
Therefore for $z_1,\ldots,z_N\in\mathbb R^d$,
\begin{align}\label{secondorder-descrete-2}
 \frac Cm\sum_{i=1}^m|z_i|^2\geq\frac1{m^2}\sum_{i,j=1}^mz^\top_i\partial^2_{\mu\mu}V(t,\gamma_m,y_i,y_j)z_j+\frac1m\sum_{i=1}^mz^\top_i\partial_x\partial_\mu V(t,\gamma_m,y_i)z_i\geq0.
\end{align}
Take $z_1=\beta,\ z_i=0\ (1<i\leq m)$, then the above implies that
\begin{align*}
 0\leq\beta^\top\partial_x\partial_\mu V(t,\gamma_m,x_1)\beta+\frac1m\beta^\top\partial^2_{\mu\mu} V(t,\gamma_m,x_1,x_1)\beta\leq C|\beta|^2.
\end{align*}
Notice that $\partial^2_{\mu\mu} V$ is continuous, hence sending $m$ to infinity gives \eqref{secondorder-descrete-1}. We continue to show the remaining part of \eqref{secondorder-descrete-0}. In fact, we may replace the $\gamma_m$ in \eqref{secondorder-descrete-2} with $\mu_N$ and obtain
\begin{align*}
  \frac CN\sum_{i=1}^N|z_i|^2\geq\frac1{N^2}\sum_{i,j=1}^Nz^\top_i\partial^2_{\mu\mu}V(t,\mu_N,x_i,x_j)z_j+\frac1N\sum_{i=1}^Nz^\top_i\partial_x\partial_\mu V(t,\mu_N,x_i)z_i\geq0.
\end{align*}
Combining the above with \eqref{secondorder-descrete-1}, we get \eqref{secondorder-descrete-0}.
\end{proof}

\begin{lemma}
 Suppose
\begin{enumerate}
  \item Assumption \ref{assumption} and $\sigma>0$;
  \item $V$ is the classical solution to \eqref{HJB-MF}. For some $\tilde\delta>0$, $\partial_tV(\cdot)\in\mathcal C\big([T-\tilde\delta,T]\times\mathcal P_2(\mathbb R^d)\big)$, $V(t,\cdot)\in\mathcal C^2\big(\mathcal P_2(\mathbb R^d)\big)$ with jointly continuous and bounded derivatives on $t\in[T-\tilde\delta,T]$.
 \end{enumerate}
Then for any pair of random variables $(\xi,\eta)\in\mathbb R^{2d}$ with finite second moment, we have
 \begin{align}\label{discrete-differ-1}
  \mathbb E\big[|\partial_\mu V(t,\mathcal L_\xi,\xi)-\partial_\mu V(t,\mathcal L_\eta,\eta)|^2\big]\leq C\mathbb E\big[|\xi-\eta|^2\big],\quad t\in[T-\tilde\delta,T].
 \end{align}
\end{lemma}
\begin{proof}
 Let us consider the following sequence of empirical measures
 \begin{align*}
  \frac1N\sum_{i=1}^N\delta_{(x_i,y_i)},
 \end{align*}
which is the coupling of 
 \begin{align*}
  \frac1N\sum_{i=1}^N\delta_{x_i}\quad\text{and}\quad\frac1N\sum_{i=1}^N\delta_{y_i}.
 \end{align*}
In view of the assumption on the smoothness of $V$, it suffices to show that
 \begin{align}\label{discrete-differ}
   \frac1N\sum_{i=1}^N\bigg|\partial_\mu V\bigg(t,\frac1N\sum_{i=1}^N\delta_{x_i},x_i\bigg)-\partial_\mu V\bigg(t,\frac1N\sum_{i=1}^N\delta_{y_i},y_i\bigg)\bigg|^2\leq\frac CN\sum_{i=1}^N|x_i-y_i|^2.
 \end{align}
Define
\begin{align*}
  A^\lambda=(a^\lambda_{ij})_{n\times n},\quad a^\lambda_{ij}=\partial^2_{\mu\mu}V\bigg(t,\frac1N\sum_{k=1}^N\delta_{x_k+\lambda(y_k-x_k)},x_i+\lambda(y_i-x_i),x_j+\lambda(y_j-x_j)\bigg)\in\mathbb R^{d\times d},
\end{align*}
as well as
\begin{align*}
 \Delta\partial_\mu V_i:=\partial_\mu V\bigg(t,\frac1N\sum_{k=1}^N\delta_{x_k},x_i\bigg)-\partial_\mu V\bigg(t,\frac1N\sum_{k=1}^N\delta_{y_k},y_i\bigg).
\end{align*}
We first show that for any $\alpha$, $\beta\in(\mathbb R^d)^N$,
\begin{align}\label{secondorder-descrete-0-1}
 \langle\alpha,A^\lambda\beta\rangle\leq CN|\alpha|\cdot|\beta|.
\end{align}
In fact, according to the first inequality in \eqref{secondorder-descrete-0}, for any $\alpha$, $\beta\in(\mathbb R^d)^N$,
\begin{align*}
 -CN|\alpha|^2\leq\langle\alpha,A^\lambda\alpha\rangle\leq CN|\alpha|^2.
\end{align*}
That is to say, the eigenvalues of $A^\lambda$ are all bounded by $CN$. Hence \eqref{secondorder-descrete-0-1}.

Next we show \eqref{discrete-differ}. Present $\Delta\partial_\mu V_i$ in such a way that
\begin{align*}
&\quad\Delta\partial_\mu V_i=\partial_\mu V\bigg(t,\frac1N\sum_{k=1}^N\delta_{y_k},y_i\bigg)-\partial_\mu V\bigg(t,\frac1N\sum_{k=1}^N\delta_{x_k},x_i\bigg)\notag\\
 &=\int_0^1D_x\partial_\mu V\bigg(t,\frac1N\sum_{k=1}^N\delta_{x_k+\lambda(y_k-x_k)},x_i+\lambda(y_i-x_i)\bigg)(y_i-x_i)d\lambda\notag\\
 &+\frac1N\sum_{j=1}^N\int_0^1\partial^2_{\mu\mu}V\bigg(t,\frac1N\sum_{k=1}^N\delta_{x_k+\lambda(y_k-x_k)},x_i+\lambda(y_i-x_i),x_j+\lambda(y_j-x_j)\bigg)(x_j-y_j)d\lambda.
\end{align*}
Then by \eqref{secondorder-descrete-0} and \eqref{secondorder-descrete-0-1}, for any $\alpha\in\mathbb R^N$ satisfying $|\alpha|=1$,
\begin{align*}
 \langle\alpha,\Delta\partial_\mu V\rangle&=\int_0^1\langle\alpha,diag\big(D_x\partial_\mu V(\lambda)\big)(x-y)\rangle d\lambda+\frac1N\int_0^1\langle\alpha,A^\lambda(x-y)\rangle d\lambda\notag\\
 &\leq C|\alpha|\cdot|x-y|+\frac1N\int_0^1CN|\alpha|\cdot|x-y| d\lambda\notag\\
 &\leq C|\alpha|\cdot|x-y|.
\end{align*}
Since $\alpha$ is arbitrary, we have
\begin{align*}
 |\Delta\partial_\mu V|\leq C|x-y|,
\end{align*}
which is exactly \eqref{discrete-differ}.
\end{proof}
With the notations in \eqref{1st-step} and \eqref{deri-represent-1}, we may use the result in \eqref{discrete-differ} and obtain the key estimate on the FBSDE \eqref{deri-FBSDE}. 
 \begin{lemma}\label{lip-estimate-1}
  Let $(\hat X_s,\hat Y_s)$ solve \eqref{deri-FBSDE} and suppose the following:
  \begin{enumerate}
  \item Assumption \ref{assumption} and $\sigma>0$;
  \item $T-t$ is sufficiently small so that Lemma \ref{short-time-well-posedness} is valid with $K=U$.
  \end{enumerate}
 Then there exists a constant $C$ depending only on \eqref{C-depend} such that
  \begin{align}\label{Lip-uniq}
   \sup_{t\leq s\leq T}\mathbb E[|\hat X_s|^2]\leq C\mathbb E[|\hat X_t|^2],\quad\sup_{t\leq s\leq T}\mathbb E[|\hat Y_s|]\leq C\mathbb E[|\hat X_t|^2]^\frac12.
  \end{align}
 \end{lemma}
\begin{proof}
 Let's adopt the notations in \eqref{deri-FBSDE-0} and \eqref{deri-FBSDE}. According to \eqref{deri-FBSDE-0},
\begin{align}\label{deri-FBSDE-2}
X^\varepsilon_0-X_0&=\varepsilon\varphi(\xi),\notag\\
    d\big(X^\varepsilon_s-X_s\big)&=\int_0^1\tilde{\mathbb E}\big[\partial^2_{\tilde\mu\tilde\mu}\mathcal H^{(p)}\big(\mathbb P_{\big(X^{\varepsilon,\lambda}_s,Y^{\varepsilon,\lambda}_s\big)},X^\varepsilon_s,Y^\varepsilon_s,\tilde X^{\varepsilon,\lambda}_s,\tilde Y^{\varepsilon,\lambda}_s\big)\cdot\big(\Delta\tilde X^\varepsilon_s,\Delta\tilde Y^\varepsilon_s\big)\big]d\lambda ds\notag\\
    &\quad+\big[\partial_{\tilde\mu}\mathcal H^{(p)}\big(\mathbb P_{(X_s,Y_s)},X^\varepsilon_s,Y^\varepsilon_s\big)-\partial_{\tilde\mu}\mathcal H^{(p)}\big(\mathbb P_{(X_s,Y_s)},X_s,Y_s\big)\big]ds,
 \end{align}
 where
 \begin{align*}
  &\big(\tilde X^{\varepsilon,\lambda}_s,\tilde Y^{\varepsilon,\lambda}_s\big):=(1-\lambda)(\tilde X_s,\tilde Y_s)+\lambda(\tilde X^\varepsilon_s,\tilde Y^\varepsilon_s),\\
  &(\Delta\tilde X^\varepsilon,\Delta\tilde Y^\varepsilon):=(\tilde X^\varepsilon_s-\tilde X_s,\tilde Y^\varepsilon_s-\tilde Y_s),
 \end{align*}
 $(\tilde X_s,\tilde Y_s,\tilde X^\varepsilon_s,\tilde Y^\varepsilon_s)$ is the independent copy of $(X^\mu_s,Y^\mu_s,X^\varepsilon_s,Y^\varepsilon_s)$ conditional on $\mathcal F^{W_0}_s$ and $\tilde{\mathbb E}$ is the expectation taken with respect to $(\tilde X_s,\tilde Y_s,\tilde X^\varepsilon_s,\tilde Y^\varepsilon_s)$ conditional on $\mathcal F^{W_0}_s$. Therefore
\begin{align*}
 &\quad\max_{t\leq u\leq s}\mathbb E[X^\varepsilon_u-X_u|^2]\\
 &\leq C\mathbb E[|X^\varepsilon_t-X_t|^2]+C\int_t^s\max_{t\leq v\leq u}\mathbb E[X^\varepsilon_v-X_v|^2]du+C\int_t^s\max_{t\leq v\leq u}\mathbb E[Y^\varepsilon_v-Y_v|^2]du.
\end{align*}
Plug \eqref{deri-represent-rmk} and \eqref{discrete-differ-1} into the above,
\begin{align*}
 \max_{t\leq u\leq s}\mathbb E[X^\varepsilon_u-X_u|^2]\leq C\mathbb E[|X^\varepsilon_t-X_t|^2]+C\int_t^s\max_{t\leq v\leq u}\mathbb E[X^\varepsilon_v-X_v|^2]du.
\end{align*}
Gronwall's inequality yields
\begin{align*}
 \max_{t\leq s\leq T}\mathbb E[X^\varepsilon_s-X_s|^2]\leq C\mathbb E[|X^\varepsilon_t-X_t|^2].
\end{align*}
Notice that \eqref{deri-FBSDE} is the linearized system obtained by formally taking derivative w.r.t. $\varepsilon$ at $\varepsilon=0$ in \eqref{deri-FBSDE-0-1}. Hence the above implies
\begin{align}\label{lip-estimate}
 \max_{t\leq s\leq T}\mathbb E[|\hat X_s|^2]\leq C\mathbb E[|\hat X_t|^2].
\end{align}
According to \eqref{deri-FBSDE}, 
\begin{align*}
 &d\hat Y_t=-\partial_x\partial_{\tilde\mu}\mathcal H^{(x)}\big(\mathbb P_{(X_t,Y_t)},X_t,Y_t\big) \hat X_tdt-\partial_p\partial_{\tilde\mu}\mathcal H^{(x)}\big(\mathbb P_{(X_t,Y_t)},X_t,Y_t\big) \hat Y_tdt\\
 &-\tilde{\mathbb E}\big[\partial^2_{\tilde\mu\tilde\mu}\mathcal H^{(x)(x)}\big(\mathbb P_{(X_t,Y_t)},X_t,Y_t,\tilde  X_t,\tilde  Y_t\big)\hat X_t\big]dt-\tilde{\mathbb E}\big[\partial^2_{\tilde\mu\tilde\mu}\mathcal H^{(x)(p)}\big(\mathbb P_{(X_t,Y_t)},X_t,Y_t,\tilde  X_t,\tilde  Y_t\big) \hat Y_t\big]dt\\
 &+\hat Z_tdW_t+\hat Z^0_tdW^0_t.
\end{align*}
After taking expectation on both sides in the above and noticing our assumptions on the parameters, it is easy to show, by \eqref{lip-estimate} and Gr{\"o}nwall's inequality, that
\begin{align*}
 \sup_{t\leq s\leq T}\mathbb E[|\hat Y_s|]\leq C\mathbb E[|\hat X_t|^2]^\frac12.
\end{align*}
\end{proof}
With the preparation of Lemma \ref{lip-estimate-1} at hand, we finally come to the estimate on $\hat Y^{x_1}_t$ \eqref{deri-represent-2}.
\begin{lemma}\label{lip-estimate-2-1}
  Under the same assumptions as Lemma \ref{lip-estimate-1}, there exists a constant $C$ depending only on \eqref{C-depend} such that for $(\hat X^x_s,\hat Y^x_s)$ in \eqref{deri-FBSDE-1},
  \begin{align}\label{Lip-uniq}
   \sup_{t\leq s\leq T}\mathbb E[|\hat X^{x}_s|^2]\leq C,\quad\sup_{t\leq s\leq T}\mathbb E[|\hat Y^{x}_s|]\leq C\mathbb E[|\hat X_t|^2]^\frac12.
  \end{align}
 \end{lemma}
 \begin{proof}
  For $s\in[t,T]$, denote by
 {\small \begin{align}\label{def-alpha-beta}
   \left\{\begin{aligned}\alpha_s&:=\tilde{\mathbb E}\big[\partial^2_{\tilde\mu\tilde\mu}\mathcal H^{(p)(x)}\big(\mathbb P_{(X_s,Y_s)},X^{x_0}_s,Y^{x_0}_s,\tilde X_s,\tilde Y_s\big)\tilde{\hat X}_s\big]+\tilde{\mathbb E}\big[\partial^2_{\tilde\mu\tilde\mu}\mathcal H^{(p)(p)}\big(\mathbb P_{(X_s,Y_s)},X^{x}_s,Y^{x}_s,\tilde  X_s,\tilde  Y_s\big)\tilde{\hat Y}_s\big],\\
   \beta_s&:=-\tilde{\mathbb E}\big[\partial^2_{\tilde\mu\tilde\mu}\mathcal H^{(x)(x)}\big(\mathbb P_{(X_s,Y_s)},X^{x}_s,Y^{x}_s,\tilde  X_s,\tilde  Y_s\big)\hat X_s\big]-\tilde{\mathbb E}\big[\partial^2_{\tilde\mu\tilde\mu}\mathcal H^{(x)(p)}\big(\mathbb P_{(X_s,Y_s)},X^{x}_s,Y^{x}_s,\tilde  X_s,\tilde  Y_s\big) \hat Y_s\big].\end{aligned}\right.
  \end{align}}
 Then FBSDE \eqref{deri-FBSDE-1} is reduced to a linear one
 \begin{align*}
  \left\{\begin{aligned}
    d\hat X^{x}_s&=\alpha_sds+\partial_x\partial_{\tilde\mu}\mathcal H^{(p)}\big(\mathbb P_{(X_s,Y_s)},X^{x}_s,Y^{x}_s\big) \hat X^{x}_sds+\partial_p\partial_{\tilde\mu}\mathcal H^{(p)}\big(\mathbb P_{(X_s,Y_s)},X^{x}_s,Y^{x}_s\big) \hat Y^{x}_sds,\\
    d\hat Y^{x}_s&=\beta_sds-\partial_x\partial_{\tilde\mu}\mathcal H^{(x)}\big(\mathbb P_{(X_s,Y_s)},X^{x}_s,Y^{x}_s\big) \hat X^{x}_sds-\partial_p\partial_{\tilde\mu}\mathcal H^{(x)}\big(\mathbb P_{(X_s,Y_s)},X^{x}_s,Y^{x}_s\big) \hat Y^{x}_sds\notag\\
    &\quad+\hat Z^{x}_sdW_s+\hat Z^0_sdW^0_s,\\
    &\hat X^{x}_t=0,\quad\hat Y^{x}_T=\tilde{\mathbb E}\big[\partial^2_{\mu\mu}U(\mathbb P_{X_T},X^{x}_T,\tilde X_T)\tilde{\hat X}_T\big]+\partial_x\partial_\mu U(\mathbb P_{X_T},X^{x}_T)\hat X^{x}_T.
\end{aligned}\right.
 \end{align*}
 Here, according to Lemma \ref{lip-estimate-1}, $\alpha_t$ and $\beta_t$ are both bounded.
 
In view of Lemma \ref{assum-impli},
\begin{align*}
 \partial_p\partial_{\tilde\mu}\mathcal H^{(p)}\big(\mathbb P_{(X_s,Y_s)},X^{x}_s,Y^{x}_s\big),\ -\partial_x\partial_{\tilde\mu}\mathcal H^{(x)}\big(\mathbb P_{(X_s,Y_s)},X^{x}_s,Y^{x}_s\big)\leq0.
\end{align*}
Therefore we make the ansatz
\begin{align}\label{lip-estimate-2-1-0}
 \hat Y^{x}_s=\hat P^{x}_s\hat X^{x}_s+\gamma_s,
\end{align}
where $\hat P^{x}_s$ and $\gamma_s$ backwardly satisfy
{\small\begin{align*}
 \left\{\begin{aligned}
 &d\hat P^{x}_s=-\bigg[\partial_x\partial_{\tilde\mu}\mathcal H^{(x)}(s)+\hat P^{x}_s\partial_x\partial_{\tilde\mu}\mathcal H^{(p)}(s)+\partial_p\partial_{\tilde\mu}\mathcal H^{(x)}(s)\hat P^{x}_s+\hat P^{x}_s\partial_p\partial_{\tilde\mu}\mathcal H^{(p)}(s)\hat P^{x}_s\bigg]ds+d\mathcal M_s,\notag\\
 &d\gamma_s=\bigg[\beta_s-\hat P^{x}_s\alpha_s-\hat P^{x}_s\partial_p\partial_{\tilde\mu}\mathcal H^{(p)}(s)\gamma_s-\partial_p\partial_{\tilde\mu}\mathcal H^{(x)}(s)\gamma_s\bigg]ds+d\mathcal M^\gamma_s,\ s\in[t,T],\notag\\
 &\hat P^{x}_T=\partial_x\partial_\mu U(\mathbb P_{X_T},X^{x}_T),\quad\gamma_T=\tilde{\mathbb E}\big[\partial^2_{\mu\mu}U(\mathbb P_{X_T},X^{x}_T,\tilde X_T)\tilde{\hat X}_T\big],
 \end{aligned}\right.
\end{align*}}
Here $\mathcal M_t$ and $\mathcal M^\gamma_t$ are matrix-valued  and vector-valued martingales, and for
\begin{align*}
 \eta=\partial_x\partial_{\tilde\mu}\mathcal H^{(x)},\ \partial_x\partial_{\tilde\mu}\mathcal H^{(p)},\ \partial_p\partial_{\tilde\mu}\mathcal H^{(p)},\ \partial_p\partial_{\tilde\mu}\mathcal H^{(x)},
\end{align*}
we have used the abbreviation
\begin{align*}
 \eta(s)=\eta\big(\mathbb P_{(X_s,Y_s)},X^{x}_s,Y^{x}_s\big).
\end{align*}
Given Assumption \ref{assumption}, according to the results on symmetric backward stochastic Riccati equations (see \cite{Tang2003}), $\hat P^{x}_s$ exists and is bounded thus so does $\gamma_s$. Then we may use the similar argument to the proof of Lemma \ref{prop-eigen} and show that $\hat P^{x}_s$ is bounded by a constant depending only on \eqref{C-depend}. Plug $\hat X^{x}_t=0$ into \eqref{lip-estimate-2-1-0}, in view of Lemma \ref{lip-estimate-1}, \eqref{def-alpha-beta} and a priori estimates on $\hat P^{x}_s$,
\begin{align*}
 |\hat Y^{x}_t|=|\gamma_t|\leq C\mathbb E\bigg[\int_t^T\big(|\alpha_s|^2+|\beta_s|^2\big)ds\bigg]^\frac12\leq C\mathbb E[|\hat X_t|^2]^\frac12.
\end{align*}
where $C$ depends only on \eqref{C-depend}. Hence \eqref{Lip-uniq}.
 \end{proof}
A direct consequence of Lemma \ref{key-second-order} and Lemma \ref{lip-estimate-2-1} is the following.
\begin{theorem}\label{lip-estimate-2-1-1}
   Under the same assumptions as Lemma \ref{lip-estimate-1}, there exists a constant $C$ depending only on \eqref{C-depend} such that
   \begin{align*}
    0\leq|\partial_x\partial_\mu V(t,\cdot)|_\infty+|\partial^2_{\mu\mu}V(t,\cdot)|_\infty\leq C.
   \end{align*}
  \end{theorem}
\begin{proof}
 In view of \eqref{deri-FBSDE}, \eqref{deri-represent-2} and \eqref{Lip-uniq}, since $\hat X_t=\varphi(\xi)$ in \eqref{deri-FBSDE} and $\varphi$ is arbitrary, we have
 \begin{align*}
  0\leq|\partial^2_{\mu\mu}V(t,\cdot)|_\infty\leq C.
 \end{align*}
Notice also Lemma \ref{key-second-order}, and the proof is completed.
\end{proof}

\subsection{A priori estimates on higher order derivatives of $V$}\label{higher-order-estmiates}
In this section we further exploit the results in Section \ref{prior-on-2nd} and study the higher regularity concerning with the following
\begin{align}\label{higher-order-notation}
 D^{i_1}_{x_1}\cdots D^{i_j}_{x_j}\partial^j_\mu V(t,\mu,x_1,\ldots,x_j),\quad0\leq i_1,\ldots,i_j\leq4,\ 0\leq i_1+\cdots+i_j+j\leq 4.
\end{align}
According to Lemma \ref{veri-V} and Lemma \ref{short-time-well-posedness}, there exists a constant $\tilde c$ such that, on $t\in[T-\tilde c,T]$, the solution to the HJB equation \eqref{HJB-MF}  admits a solution $V\in\mathcal C^4\big(\mathcal P_2(\mathbb R^d)\big)$ with bounded derivatives. In order to show the global well-posedness of solution $V\in\mathcal C^4\big(\mathcal P_2(\mathbb R^d)\big)$, we hope to  repeatedly used Lemma \ref{veri-V} and Lemma \ref{short-time-well-posedness} on the time interval $[T-(n+1)\delta,T-n\delta]$, $n=0,1,2,\ldots$, which is in the spirit of \cite{Carmona2018-I,Jean14}. However, every time we apply the contraction method, the corresponding time duration $\delta$ might change according to the following sum:
\begin{align*}
 \sum_{\substack{0\leq k\leq 4,0\leq l_1,\ldots,l_k\leq4\\0\leq l_1+\cdots+l_k+k\leq 4}}\big|D^{l_1}_{x_1}\cdots D^{l_k}_{x_k}\partial^k_\mu V(T-n\tilde\delta,\cdot)\big|_\infty\notag\\
 +\sum_{\substack{0\leq k\leq 4,0\leq l_1,\ldots,l_k\leq4\\0\leq\tilde l_1,\ldots,\tilde l_k\leq4,0\leq l_1+\cdots+l_k\\+\tilde l_1+\cdots+\tilde l_k+k\leq 4}}\big|D^{l_1}_{x_1}D^{\tilde l_1}_{p_1}\cdots D^{l_k}_{x_k}\partial^{\tilde l_k}_{p_k}\partial^k_{\tilde\mu}\mathcal H\big|_\infty.
\end{align*}
In view of the above dependence, it suffices to show that $V$ is sufficiently smooth and
\begin{align}\label{higher-order-goal}
 \sum_{\substack{0\leq k\leq 4,0\leq l_1,\ldots,l_k\leq4\\0\leq l_1+\cdots+l_k+k\leq 4}}\big|D^{l_1}_{x_1}\cdots D^{l_k}_{x_k}\partial^k_\mu V(t,\cdot)\big|_\infty<C_2,
\end{align}
where the constant $C_2$ depends only on \eqref{tilde c-dependence} with $K=U$. In particular, $C_2$ is independent of $\tilde\delta$.

As is shown in Proposition \ref{1st-regularity} and Theorem \ref{lip-estimate-2-1-1}, $\partial_\mu V$, $\partial_x\partial_\mu V$, $\partial^2_{\mu\mu}V$ are all uniformly bounded whenever $V$ is sufficiently smooth. In other words, we have already shown that $\partial_\mu V$, $\partial_x\partial_\mu V$, $\partial^2_{\mu\mu}V$ are all bounded on $t\in[T-\tilde\delta,T]$ by $C_2$ in \eqref{higher-order-goal}. Given such facts, in this section we shall inductively show that other terms in the left hand side of \eqref{higher-order-goal} exist and are all bounded by the constant $C_2$ in \eqref{higher-order-goal}. Our induction relies on a modification of Theorem 7.2 in \cite{Buckdahn2017} which concerns with the well-posedness and the Feynman-Kac representation of classical solutions to linear master equations. 

\subsubsection{Well-posedness and Feynman-Kac representation of linear master equations}
Next we describe the aforementioned modification of Theorem 7.2 in \cite{Buckdahn2017}, which could be proved following a similar idea to that in \cite{Buckdahn2017}. Consider the following McKean-Vlasov stochastic differential equation system which we call the population-particle pair for brevity:
\begin{align}\label{pp-pair}\left\{\begin{aligned}
 &d\tilde X^{t,\xi}_s=\tilde\sigma(s,\tilde X^{t,\xi}_s,\mathbb P_{\tilde X^{t,\xi}_s|\mathcal F^0_s})dB_s+\tilde\sigma_0(s,\tilde X^{t,\xi}_s,\mathbb P_{\tilde X^{t,\xi}_s|\mathcal F^0_s})dB^0_s+\tilde b(s,\tilde X^{t,\xi}_s,\mathbb P_{\tilde X^{t,\xi}_s|\mathcal F^0_s})ds,\\
 &dX^{t,x,\xi}_s=\sigma(s,X^{t,x,\xi}_s,\mathbb P_{\tilde X^{t,\xi}_s|\mathcal F^0_s})dB_s+\sigma_0(s,X^{t,x,\xi}_s,\mathbb P_{\tilde X^{t,\xi}_s|\mathcal F^0_s})dB^0_s+b(s,X^{t,x,\xi}_s,\mathbb P_{\tilde X^{t,\xi}_s|\mathcal F^0_s})ds,\\
 &\tilde X^{t,\xi}_t=\xi,\ X^{t,x,\xi}_t=x.\end{aligned}\right.
\end{align}
To continue, we recall Property \ref{joint-differentiation} on joint differentiability and make the following assumption upon the parameters in \eqref{pp-pair}:
\begin{assumption}\label{assumption-linear}
Let $\sigma,\tilde\sigma:\mathbb R^d\times\mathcal P_2(\mathbb R^d)\to\mathbb R^{d\times d}$; $b,\tilde b:\mathbb R^d\times\mathcal P_2(\mathbb R^d)\to\mathbb R^d$; $\Phi_1,\Phi_2,\Phi_3:\mathbb R^d\times\mathcal P_2(\mathbb R^d)\to\mathbb R$. For $\varphi=\sigma,\tilde\sigma,b,\tilde b,\Phi_1,\Phi_2,\Phi_3$, it satisfies Property \ref{joint-differentiation}.
\end{assumption}
Given Assumption \ref{assumption-linear}, it is easy to show the well-posedness of \eqref{pp-pair}. As is mentioned, one has the following modification of Theorem 7.2 in \cite{Buckdahn2017}.
\begin{theorem}\label{li-result}
 Suppose  Assumption  \ref{assumption-linear}. Then, the function
 \begin{align}\label{master-Feynman-Kac}
  &V(t,x,\mathbb P_\xi)=\mathbb E\bigg[\exp\bigg\{-\int_t^T\Phi_2\big(s,X^{t,x,\mathbb P_\xi}_s,\mathbb P_{\tilde X^{t,\xi}_s|\mathcal F^0_s}\big)ds\bigg\}\Phi_1\big(X^{t,x,\mathbb P_\xi}_T,\mathbb P_{\tilde X^{t,\xi}_T|\mathcal F^0_T}\big)\notag\\
  &\qquad\qquad-\int_t^T\exp\bigg\{-\int_t^s\Phi_2\big(r,X^{t,x,\mathbb P_\xi}_r,\mathbb P_{\tilde X^{t,\xi}_r|\mathcal F^0_r}\big)dr\bigg\}\Phi_3\big(s,X^{t,x,\mathbb P_\xi}_s,\mathbb P_{\tilde X^{t,\xi}_s|\mathcal F^0_s}\big)ds\bigg],\notag\\
  &(t,x,\xi)\in[0,T]\times\mathbb R^d\times L^2(\mathcal F_t;\mathbb R^d),
 \end{align}
satisfies Property \ref{joint-differentiation} and is the unique classical solution, with all derivatives being bounded, to the following linear master equation
\begin{align}\label{master-Feynman-Kac-1}
  0&=\partial_tV(t,x,\mu)+\sum_{i=1}^d\partial_{x_i}V(t,x,\mu)b_i(x,\mu)+\frac12\sum_{i,j,k=1}^d\partial^2_{x_ix_j}V(t,x,\mu)(\sigma_{i,k}\sigma_{j,k})(x,\mu)\notag\\
  &\quad+\frac12\sum_{i,j,k=1}^d\partial^2_{x_ix_j}V(t,x,\mu)(\sigma_{0,i,k}\sigma_{0,j,k})(x,\mu)+\int_{\mathbb R^d}\partial_\mu V(x,\mu,y)^\top\tilde b(y,\mu)\mu(dy)\notag\\
  &\quad+\sum_{i,j,k=1}^d\int_{\mathbb R^d}\partial_x\partial_\mu V(t,x,\mu,y)\sigma_{0,i,k}(x,\mu)\tilde\sigma_{0,j,k}(y,\mu)\mu(dy)\notag\\
  &\quad+\frac12\sum_{i,j,k=1}^d\int_{\mathbb R^d}\partial_y\partial_\mu V(t,x,\mu,y)(\tilde\sigma_{i,k}\tilde\sigma_{j,k})(y,\mu)\mu(dy)\notag\\
  &\quad+\frac12\sum_{i,j,k=1}^d\int_{\mathbb R^d}\partial_y\partial_\mu V(t,x,\mu,y)(\tilde\sigma_{0,i,k}\tilde\sigma_{0,j,k})(y,\mu)\mu(dy)\notag\\
  &\quad+\sum_{i,j,k=1}^d\int_{\mathbb R^d}\partial^2_{\mu\mu} V(t,x,\mu,y_1,y_2)\tilde\sigma_{0,i,k}(y_1,\mu)\tilde\sigma_{0,j,k}(y_2,\mu)\mu(dy_1)\mu(dy_2)\notag\\
  &\quad+\Phi_2(t,x,\mu)V(t,x,\mu)+\Phi_3(t,x,\mu),\quad(t,x,\mu)\in[0,T]\times\mathbb R^d\times\mathcal P(\mathbb R^d),\notag\\
  &V(T,x,\mu)=\Phi_1(x,\mu).
 \end{align}
\end{theorem}
\begin{proof} 
Given $(\tilde X^{t,\xi}_s,X^{t,x,\xi}_s)$ in \eqref{pp-pair}, we may consider inductively taking their $L^2$-derivatives w.r.t. $(x,\mu)$ and obtain the linearized system describing these $L^2$-derivatives. Then we may use the same argument as in Lemma 4.1 and Lemma 4.2 in \cite{Buckdahn2017} to show that the aforementioned linearized system is well-posed. The existence of $L^2$-derivatives of $(\tilde X^{t,\xi}_s,X^{t,x,\xi}_s)$ further yields that the $V(t,x,\mu)$ in \eqref{master-Feynman-Kac} is continuously differentiable in $(t,x,\mu)$. At the meant time, we have from the flow property of $V(t,x,\mu)$ that for $\Delta t>0$,
\begin{align*}
 &0=\frac1{\Delta t}\mathbb E\bigg[\exp\bigg\{-\int_t^{t+\Delta t}\Phi_2\big(s,X^{t,x,\mathbb P_\xi}_s,\mathbb P_{\tilde X^{t,\xi}_s|\mathcal F^0_s}\big)ds\bigg\}V\big(X^{t,x,\mathbb P_\xi}_{t+\Delta t},\mathbb P_{\tilde X^{t,\xi}_{t+\Delta t}|\mathcal F^0_{t+\Delta t}}\big)\notag\\
  &-\int_t^{t+\Delta t}\exp\bigg\{-\int_t^s\Phi_2\big(r,X^{t,x,\mathbb P_\xi}_r,\mathbb P_{\tilde X^{t,\xi}_r|\mathcal F^0_r}\big)dr\bigg\}\Phi_3\big(s,X^{t,x,\mathbb P_\xi}_s,\mathbb P_{\tilde X^{t,\xi}_s|\mathcal F^0_s}\big)ds-V(t,x,\mathbb P_\xi)\bigg],\notag\\
  &(t,x,\xi)\in[0,T]\times\mathbb R^d\times L^2(\mathcal F_t;\mathbb R^d).
\end{align*}
Since $V$ is sufficiently smooth, we may pass $\Delta t$ to $0+$ and apply the generalized It\^o's formula to the above and obtain \eqref{master-Feynman-Kac-1}. On the other hand, given any classical solution to \eqref{master-Feynman-Kac-1}, we may show the uniqueness via the Feynman-Kac representation in \eqref{master-Feynman-Kac}.
\end{proof}

\subsubsection{A priori estimates on higher order derivatives}
In this subsection we apply Theorem \ref{li-result} and inductively show the aforementioned a priori estimates. Since the a priori estimates for all the derivatives $D^{i_1}_{x_1}\cdots D^{i_j}_{x_j}\partial^j_\mu V$ in \eqref{higher-order-notation} are of the same nature, we only describe how we estimate $D^k_x\partial_\mu V(t,\mu,x)$ ($k=1,2,3$). 

\begin{proposition}\label{further-regularity}
Suppose the following:
 \begin{enumerate}
  \item Assumption \ref{assumption} and $\sigma>0$;
  \item There exists a positive constant $\tilde\delta$ such that $V(t,\cdot)\in\mathcal C^4(\mathcal P_2\big(\mathbb R^d)\big)$ with continuous and bounded derivatives on $t\in[T-\tilde\delta,T]$.
 \end{enumerate}
Then $D^k_z\partial_\mu V(t,\mu,z)$ ($k=1,2,3$) are continuous and bounded by a constant depending only on \eqref{tilde c-dependence} with $K=U$ (independent of $\tilde\delta$) on $t\in[T-\tilde\delta,T]$.
\end{proposition}
\begin{proof}
 We prove this proposition by inductively taking $D_x$ and $\partial_\mu$ in \eqref{HJB-MF}. Let us first take $\partial_\mu$ in \eqref{HJB-MF} and denote $\partial_\mu V$ by $V^{(0,1)}$. Then
\begin{align}\label{HJB-MF-1}
  \left\{\begin{aligned}
    &\partial_t V^{(0,1)}(t,\mu,z)+\frac12{\rm tr}\big((\sigma^\top\sigma+\sigma^\top_0\sigma_0)D^2_z V^{(0,1)}(t,\mu,z)\big)+\frac12\int_{\mathbb R^d}{\rm tr}\big((\sigma^\top\sigma+\sigma^\top_0\sigma_0)\partial_x\partial_\mu V^{(0,1)}(t,\mu,x,z)\big)\mu(dx)\\
    &+\int_{\mathbb R^d}{\rm tr}\big(\sigma^\top_0\sigma_0\partial_y\partial_\mu V^{(0,1)}(t,\mu,x,z)\big)\mu(dx)+\frac12\int_{\mathbb R^d}\int_{\mathbb R^d}{\rm tr}\big(\sigma^\top_0\sigma_0\partial^2_{\mu\mu}V^{(0,1)}(t,\mu,x,y,z)\big)\mu(dx)\mu(dy)+\mathcal G(\mu,z)=0,\\
    &V^{(0,1)}(T,\mu,z)=\partial_\mu U(\mu,z),\quad(t,\mu,z)\in[0,T)\times\mathcal P_2(\mathbb R^d)\times\mathbb R^d.
  \end{aligned}\right.
\end{align}
where the differential and trace operators are taken componentwise. In the above calculation we have introduced the notation
\begin{align*}
 &\quad\mathcal G(\mu,x):=\partial_{\tilde\mu}\big(\mathcal H(\tilde\mu)\big)(x):=\partial_{\tilde\mu}\bigg(\mathcal H\big((Id,V^{(0,1)}(t,\mu,\cdot))\sharp\mu\big)\bigg)(x)\notag\\
 &=\partial_{\tilde\mu}\mathcal H^{(x)}(\tilde\mu,x,V^{(0,1)}(t,\mu,x))+D_xV^{(0,1)}(t,\mu,x)\partial_{\tilde\mu}\mathcal H^{(p)}\big(\tilde\mu,x,V^{(0,1)}(t,\mu,x)\big)\notag\\
 &\quad+\int_{\mathbb R^d}\partial_\mu V^{(0,1)}(t,\mu,y,x)\partial_{\tilde\mu}\mathcal H^{(p)}(\tilde\mu,y,V^{(0,1)}(t,\mu,y))\mu(dy).
\end{align*}
Rewrite \eqref{HJB-MF-1} as the following
\begin{align}\label{HJB-MF-2}
  \left\{\begin{aligned}
    &\mathcal L^{(0,1)}_{t,\mu}V^{(0,1)}(t,\mu,z)+\partial_{\tilde\mu}\mathcal H^{(x)}(\tilde\mu,z,V^{(0,1)}(t,\mu,z))=:\mathcal L^{(0,1)}_{t,\mu}V^{(0,1)}(t,\mu,z)+\mathcal G^{(0,1)}(t,\mu,z)=0,\\
    &V^{(0,1)}(T,\mu,z)=\partial_\mu U(\mu,z),\quad(t,\mu,z)\in[0,T)\times\mathcal P_2(\mathbb R^d)\times\mathbb R^d.
  \end{aligned}\right.
\end{align}
Here for sufficiently smooth $V$, i.e., $V(\cdot,\mu,\cdot)\in C^{1,2}\big([0,T)\times\mathbb R^d;\mathbb R^d\big)\cap C\big([0,T]\times\mathbb R^d;\mathbb R^d\big)$, $V(t,\cdot,z)\in\mathcal C^2\big(\mathcal P_2(\mathbb R^d);\mathbb R^d\big)$, $\partial_zV(t,\cdot,z)\in\mathcal C^1\big(\mathcal P_2(\mathbb R^d);\mathbb R^d\big)$ with jointly continuous derivatives, $(t,\mu,z)\in[0,T)\times\mathcal P_2(\mathbb R^d)\times\mathbb R^d$, the linear operator is defined as follows
\begin{align*}
& \mathcal L^{(0,1)}_{t,\mu}V(t,\mu,z):=\partial_t V(t,\mu,z)+\frac12{\rm tr}\big((\sigma^\top\sigma+\sigma^\top_0\sigma_0)D^2_z V(t,\mu,z)\big)\\
    &\quad+\frac12\int_{\mathbb R^d}{\rm tr}\big((\sigma^\top\sigma+\sigma^\top_0\sigma_0)\partial_x\partial_\mu V(t,\mu,x,z)\big)\mu(dx)+\int_{\mathbb R^d}{\rm tr}\big(\sigma^\top_0\sigma_0\partial_y\partial_\mu V(t,\mu,x,z)\big)\mu(dx)\\
    &\quad+\frac12\int_{\mathbb R^d}\int_{\mathbb R^d}{\rm tr}\big(\sigma^\top_0\sigma_0\partial^2_{\mu\mu}V(t,\mu,x,y,z)\big)\mu(dx)\mu(dy)+D_zV(t,\mu,z)\partial_{\tilde\mu}\mathcal H^{(p)}(\tilde\mu,z,V^{(0,1)}(t,\mu,z))\\
    &\quad+\int_{\mathbb R^d}\partial_\mu V(t,\mu,x,z)\partial_{\tilde\mu}\mathcal H^{(p)}(\tilde\mu,x,V^{(0,1)}(t,\mu,x))\mu(dx),
\end{align*}
and
\begin{align*}
 \mathcal G^{(0,1)}(t,\mu,x):=\partial_{\tilde\mu}\mathcal H^{(x)}(\tilde\mu,x,V^{(0,1)}(t,\mu,x)).
\end{align*}
The ${\cal G}^{(0,1)}$ above depends on $V^{(0,1)}$. But according to the assumption of this proposition, $V(t,\cdot)\in\mathcal C^4\big(\mathcal P_2(\mathbb R^d)\big)$ on $t\in[T-\tilde\delta,T]$, with all derivatives being bounded. Hence we may choose
\begin{align*}
 \Phi_1(t,\mu,x)=\partial_\mu U(\mu,x),\quad\Phi_2(t,\mu,x)=0,\quad\Phi_3(t,\mu,x)=\mathcal G^{(0,1)}(t,\mu,x),
\end{align*}
and apply Theorem \ref{li-result} to \eqref{HJB-MF-2} and obtain that $V^{(0,1)}(t,\mu,z)$ satisfies Property \ref{joint-differentiation}. Moreover, the Feynman-Kac representation in \eqref{master-Feynman-Kac} yields that $V^{(0,1)}(t,\mu,z)$ is uniformly bounded.

Next we take $D_z$ in \eqref{HJB-MF-2} and obtain
\begin{align}\label{HJB-MF-3-0}
  \left\{\begin{aligned}
    &\mathcal L^{(0,1)}_{t,\mu}V^{(1,1)}(t,\mu,z)+D_z\big[\partial_{\tilde\mu}\mathcal H^{(p)}(\tilde\mu,z,V^{(0,1)}(t,\mu,z))\big]+D_z\mathcal G^{(0,1)}(t,\mu,z)\\
    &=:\mathcal L^{(1,1)}_{t,\mu}V^{(1,1)}(t,\mu,z)+{\cal G}^{(1,1)}(t,\mu,z)=0,\quad(t,\mu,z)\in[0,T)\times\mathcal P_2(\mathbb R^d)\times\mathbb R^d,\\
    &V^{(1,1)}(T,\mu,z)=D_z\partial_\mu U(\mu,z).
  \end{aligned}\right.
\end{align}
Here we have introduced the notations $V^{(1,1)}(t,\mu,z):=D_zV^{(0,1)}(t,\mu,z)\in\mathbb R^{d\times d}$ as well as
\begin{align*}
 &{\cal G}^{(1,1)}(t,\mu,z):=D_z\mathcal G^{(0,1)}(t,\mu,z)-D_p\partial_{\tilde\mu}\mathcal H^{(x)}(\tilde\mu,z,V^{(0,1)}(t,\mu,z))D_x V^{(0,1)}(t,\mu,z)\\
 &=D_y\bigg[\partial_{\tilde\mu}\mathcal H^{(x)}(\tilde\mu,y,V^{(0,1)}(t,\mu,y))\bigg]-V^{(1,1)}(t,\mu,z)D_p\partial_{\tilde\mu}\mathcal H^{(x)}(\tilde\mu,z,V^{(0,1)}(t,\mu,z))\notag\\
 &=D_z\partial_{\tilde\mu}\mathcal H^{(x)}(\tilde\mu,z,V^{(0,1)}(t,\mu,z)),
 \end{align*}
 And for $V(\cdot,\mu,\cdot)\in C^{1,2}\big([0,T)\times\mathbb R\big)\cap C\big([0,T]\times\mathbb R\big)$, $V(t,\cdot,z)\in\mathcal C^2\big(\mathcal P_2(\mathbb R)\big)$, $\partial_zV(t,\cdot,z)\in\mathcal C^1\big(\mathcal P_2(\mathbb R)\big)$ with jointly continuous derivatives, $(t,\mu,z)\in[0,T)\times\mathcal P_2(\mathbb R)\times\mathbb R$,
 \begin{align*}
 \mathcal L^{(1,1)}_{t,\mu}V(t,\mu,z)&:=\mathcal L^{(0,1)}_{t,\mu}V(t,\mu,z)+D_z\partial_{\tilde\mu}\mathcal H^{(p)}\big(\tilde\mu,z,V^{(0,1)}(t,\mu,z)\big)V(t,\mu,z)\\
 &\quad+V^{(1,1)}(t,\mu,z)D_p\partial_{\tilde\mu}\mathcal H^{(p)}\big(\tilde\mu,z,V^{(0,1)}(t,\mu,z)\big)V(t,\mu,z)\\
 &\quad+D_p\partial_{\tilde\mu}\mathcal H^{(x)}\big(\tilde\mu,z,V^{(0,1)}(t,\mu,z)\big)V(t,\mu,z).\notag
\end{align*}
At this stage we find that the parameters of $\mathcal L^{(1,1)}_{t,\mu}$ and ${\cal G}^{(1,1)}$ depend on both $V^{(0,1)}$ and $V^{(1,1)}$. Since we have assumed that $V(t,\cdot)\in\mathcal C^4(\mathcal P(\mathbb R^d))$ with bounded derivatives, $\mathcal L^{(1,1)}_{t,\mu}$ and ${\cal G}^{(1,1)}$ satisfy the requirements in Theorem \ref{li-result} on $t\in[T-\tilde\delta,T]$. Hence we may choose
\begin{align*}
 &\Phi_1(t,x,\mu)=D_y\partial_\mu U(\mu,x),\\
 &\Phi_2(t,x,\mu)=D_p\partial_{\tilde\mu}\mathcal H^{(p)}(\tilde\mu,x,V^{(0,1)}(t,\mu,x))V^{(1,1)}(t,\mu,x)+D_p\partial_{\tilde\mu}\mathcal H^{(x)}(\tilde\mu,x,V^{(0,1)}(t,\mu,x)),\\
 &\Phi_3(t,x,\mu)={\cal G}^{(1,1)}(t,\mu,x),\quad (t,\mu,x)\in[0,T)\times\mathcal P_2(\mathbb R)\times\mathbb R,
\end{align*}
and apply Theorem \ref{li-result} again to \eqref{HJB-MF-3-0} and obtain that $V^{(1,1)}$ satisfies Property \ref{joint-differentiation}. In view of the uniform estimates in Lemma \ref{key-second-order}, Theorem \ref{lip-estimate-2-1-1}, as well as the Feynman-Kac representation \eqref{master-Feynman-Kac}, we may get a priori estimates on $V^{(1,1)}$.

For $z=(z_1,\ldots,z_d)\in\mathbb R^d$, further taking $D_{z_1}$ in \eqref{HJB-MF-3-0} yields the equation for $V^{(2,1)}(t,\mu,z):=D_{z_1}D_z\partial_\mu V(t,\mu,z)$:
\begin{align}\label{HJB-MF-3}
 \mathcal L^{(2,1)}_{t,\mu}V^{(2,1)}(t,\mu,z)+\mathcal G^{(2,1)}(t,\mu,z)=0,
\end{align}
where
\begin{align*}
 \mathcal L^{(2,1)}_{t,\mu}f(t,\mu,z):=\mathcal L^{(1,1)}_{t,\mu}f(t,\mu,z)+V^{(1,1)}(t,\mu,z)D_p\partial_{\tilde\mu}\mathcal H^{(p)}\big(\tilde\mu,z,V^{(0,1)}(t,\mu,z)\big)f(t,\mu,z),
\end{align*}
and 
\begin{align*}
 \mathcal G^{(2,1)}(t,\mu,z)&=D_{z_1}\mathcal G^{(1,1)}(t,\mu,z)+D_{z_1}\big[\partial_{\tilde\mu}\mathcal H^{(p)}(\tilde\mu,z,V^{(0,1)}(t,\mu,z))\big]V^{(1,1)}(t,\mu,z)\notag\\
 &\quad+V^{(1,1)}(t,\mu,z)D_{z_1}\big[D_p\partial_{\tilde\mu}\mathcal H^{(p)}\big(\tilde\mu,z,V^{(0,1)}(t,\mu,z)\big)\big]V(t,\mu,z)\\
 &\quad+D_{z_1}\big[\partial_p\partial_{\tilde\mu}\mathcal H^{(x)}(\tilde\mu,z,V^{(0,1)}(t,\mu,z))\big]V^{(1,1)}(t,\mu,z).
 \end{align*}
According to the expressions above, the bound of parameters in both $\mathcal L^{(2,1)}_{t,\mu}$ and $\mathcal G^{(2,1)}$ depend only on $|V^{(0,1)}|_\infty$ and $|V^{(1,1)}|_\infty$. But both $V^{(0,1)}$ and $V^{(1,1)}$ have been shown to be uniformly bounded (independent of $\tilde\delta$) on $t\in[T-\tilde\delta,T]$. After applying the Feynman-Kac representation \eqref{master-Feynman-Kac} in Theorem \ref{li-result}, we may show that $V^{(2,1)}$ satisfies Property \ref{joint-differentiation} and is uniformly bounded (independent of $\tilde\delta$) on $t\in[T-\tilde\delta,T]$, which is therefore also true for $D^2_z\partial_\mu V(t,\mu,z)$.
 
For the final step, we take $D_{z_1}$ in \eqref{HJB-MF-3} once more and stop inductively taking derivatives, since we have fulfilled the goal of showing the boundedness of the forth order derivatives of $V$.  Taking $D_y$ in \eqref{HJB-MF-3} yields
\begin{align}\label{HJB-MF-4}
 \mathcal L^{(3,1)}_{t,\mu}V^{(3,1)}(t,\mu,z)+\mathcal G^{(3,1)}(t,\mu,z)=0.
\end{align}
The operator $\mathcal L^{(3,1)}_{t,\mu}$ and the inhomogeneous term $\mathcal G^{(3,1)}$ are obtained inductively from $\mathcal L^{(2,1)}_{t,\mu}$ and $\mathcal G^{(2,1)}$ in \eqref{HJB-MF-3}. For the sake of brevity we omit the exact expression of $\mathcal L^{(3,1)}_{t,\mu}$ and $\mathcal G^{(3,1)}$. But we note that the parameters of both $\mathcal L^{(3,1)}_{t,\mu}$ and $\mathcal G^{(3,1)}$ depend only on $V^{(0,1)}$, $V^{(1,1)}$ and $D^2_z\partial_\mu V(t,\mu,z)$. As is shown in the previous induction, $V^{(0,1)}$, $V^{(1,1)}$, $D^2_z\partial_\mu V(t,\mu,z)$ are uniformly bounded and satisfy Property \ref{joint-differentiation}. Hence $\mathcal L^{(3,1)}_{t,\mu}$ and $\mathcal G^{(3,1)}$ meet the requirements in Theorem \ref{li-result}. Therefore we may apply Theorem \ref{li-result} and the Feynman-Kac representation \eqref{master-Feynman-Kac} therein again to show that $V^{(3,1)}$ satisfies Property \ref{joint-differentiation} and is uniformly bounded (independent of $\tilde\delta$) on $t\in[T-\tilde\delta,T]$.
\end{proof}
Using the same idea as Proposition \ref{further-regularity}, we may inductively obtain further regularity results on
\begin{align*}
 D^{i_1}_{x_1}\cdots D^{i_j}_{x_j}\partial^j_\mu V(t,\mu,x_1,\ldots,x_j),\ 0\leq j\leq 4,\ 0\leq i_1,\ldots,i_j\leq4,\ 0\leq i_1+\cdots+i_j+j\leq 4.
\end{align*}
The above discussion is documented in the following.
\begin{proposition}\label{further-regularity-1}
Suppose the same as in Proposition \ref{further-regularity}. Then \eqref{higher-order-goal} holds on $t\in[T-\tilde\delta,T]$.
\end{proposition}

\section{The global well-posedness of HJB/master equations}\label{global-well-posedness-0}
\subsection{The global well-posedness of HJB equation \eqref{HJB-MF}}\label{global-well-posedness}
Thanks to Proposition \ref{further-regularity-1}, we can now use Lemma \ref{veri-V} repeatedly to obtain the global well-posedness of the master equation \eqref{HJB-MF}. Specifically, our main idea is to show the corresponding short time well-posedness inductively on time interval $[T-(k+1)\delta,T-k\delta]$, $k=0,1,2,\ldots$, where $\delta$ is a fixed constant that is sufficiently small, so that the whole time interval $[0,T]$ is covered by the above intervals. With such intuition, we may prove Theorem \ref{existence-of-decoupling-field}
\begin{proof}[Proof of Theorem \ref{existence-of-decoupling-field}]
 Take $K=U$ in \eqref{1st-step-1}. According to Lemma \ref{veri-V} and Lemma \ref{short-time-well-posedness}, we have proved the claim on $t\in[T-\tilde c,T]$ where the positive constant $\tilde c$ depends only on \eqref{tilde c-dependence} with $K=U$. After obtaining such $V$ on $t\in[T-\tilde c,T]$, we may use Proposition \ref{further-regularity-1} to show the existence of a constant $C$, depending only on \eqref{tilde c-dependence} with $K=U$, such that
 \begin{align}\label{iteration}
  \sum_{\substack{0\leq k\leq 4,0\leq l_1,\ldots,l_k\leq4\\0\leq l_1+\cdots+l_k+k\leq 4}}|D^{l_1}_{x_1}\cdots D^{l_k}_{x_k}\partial^k_\mu V|_\infty\leq C,\quad t\in[T-\tilde c,T].
 \end{align}
Now we may take 
 \begin{align*}
  K(\mu)=V(T-\tilde c,\mu),\quad \mu \in \mathcal P_2(\mathbb R^d),
 \end{align*}
and plug it in \eqref{1st-step-1} to show the claim on $t\in[T-\tilde c-c,T]$. According to Lemma \ref{short-time-well-posedness}, the positive constant $c$ depends only on \eqref{tilde c-dependence} with $K(\cdot)=V(T-\tilde c,\cdot)$. Moreover, in view of \eqref{iteration}, $c$ depends only on \eqref{tilde c-dependence} with $K=U$. Again we have from Proposition \ref{further-regularity-1} that 
\begin{align}\label{iteration-1}
  \sum_{\substack{0\leq k\leq 4,0\leq l_1,\ldots,l_k\leq4\\0\leq l_1+\cdots+l_k+k\leq 4}}|D^{l_1}_{x_1}\cdots D^{l_k}_{x_k}\partial^k_\mu V|_\infty\leq C,\quad t\in[T-\tilde c-c,T],
 \end{align}
 with the same constant in \eqref{iteration}. Using the same method that yields \eqref{iteration-1}, we can inductively take 
 \begin{align*}
  K(\mu)=V(T-\tilde c-kc,\mu),\quad k=1,2,\ldots,
 \end{align*}
 in \eqref{1st-step-1} and show that the same estimates in \eqref{iteration-1} holds on $t\in[T-\tilde c-(k+1)c,T-\tilde c-kc]$. Therefore, when $k$ is sufficiently large, we have shown that $\eqref{Y_t=V(X)}$ is valid on $t\in[0,T]$. Moreover,
 \begin{align}\label{iteration-2}
  \sum_{\substack{0\leq k\leq 4,0\leq l_1,\ldots,l_k\leq4\\0\leq l_1+\cdots+l_k+k\leq 4}}|D^{l_1}_{x_1}\cdots D^{l_k}_{x_k}\partial^k_\mu V|_\infty\leq C,\quad t\in[0,T].
 \end{align}
\end{proof}

\subsection{The global well-posedness of master equation \eqref{HJB-MFGC}}\label{po-mfgc}
In this section we derive the results on potential mean field games of controls (MFGC). Similar to the case with generalized mean field control problems, we describe a typical model for better intuition. In such model, the representative player has the dynamics
\begin{align}\label{mean-field-particle}
   \left\{\begin{aligned}
    &dX_t = \theta_tdt +\sigma dW_t+\sigma_0dW^0_t,\\
&X_0=x\in\mathbb R^d.
\end{aligned}\right.
 \end{align}
 Denote by $\breve\mu_t\in\mathcal P_2(\mathbb R^d\times\mathbb R^d)$ the flow of generalized population, i.e., the joint distribution of certain benchmark optimal path-momentum pair $(\breve X_t,\breve\theta_t)$. The representative player then evolves according to the criterion
 {\begin{align}\label{optim.-1}
   J(\theta,t,\mu,x):=\mathbb E\bigg[\int_t^T\tilde L\big(X_s,\theta_s,\breve\mu_s\big)ds+G(x\sharp\breve\mu_T,X(T))\bigg]\ \longrightarrow\ \text{Min!}
 \end{align}}
 In order to obtain an equilibrium of MFGC \eqref{mean-field-particle}$\sim$\eqref{optim.-1}, we may utilize the dynamic programming principle to formally obtain the master equation \eqref{HJB-MFGC} with
\begin{align*}
 \breve{\cal H}\big(\breve\mu,x,p\big)=\inf_{\theta\in\mathbb R^d}\big\{\tilde L\big(x,\theta,\breve\mu\big)+\theta\cdot p\big\}.
\end{align*}
See e.g. \cite{Mou2022MFGC} for more related discussion on \eqref{HJB-MFGC}.

Next we make use of our previous results on the generalized MFC problems and establish the associated results on \eqref{HJB-MFGC} which describes potential mean field games of controls, where Assumption \ref{potential-game-assumption} is also imposed. In view of Assumption \ref{potential-game-assumption} and Lemma \ref{assum-impli}, Fenchel biconjugation theorem yields
\begin{align}\label{MFGC-Lagrange}
 \tilde L\big(x,\theta,\breve\mu\big)=\sup_{p\in\mathbb R^d}\big\{\breve{\cal H}\big(\breve\mu,x,p\big)-\theta\cdot p\big\}.
\end{align}
In order to establish the well-posedness of \eqref{HJB-MFGC}, we take $D_x$ in \eqref{HJB-MFGC} and consider the equation concerning $\breve V(t,x,\mu):=D_x\mathcal V(t,x,\mu)$ as follows.
\begin{align}\label{potential-SMP}
 \left\{\begin{aligned}
    &\partial_t\breve V(t,x,\mu)+\frac12{\rm tr}\big(D^2_{xx}\breve V(t,x,\mu)(\sigma\sigma^\top+\sigma_0\sigma^\top_0)\big)+D_x\breve{\cal H}\big(\breve\mu,x,\breve V(t,x,\mu)\big)\\
    &+D_p\breve{\cal H}\big(\breve\mu,x,\breve V(t,x,\mu)\big)\partial_x\breve V(t,x,\mu)+\frac12\int_{\mathbb R^d}{\rm tr}\big(D_y\partial_\mu\breve V(t,x,\mu,y)(\sigma\sigma^\top+\sigma_0\sigma^\top_0)\big)\mu(dy)\\
    &+\frac12\int_{\mathbb R^d}\int_{\mathbb R^d}{\rm tr}\big(\partial^2_{\mu\mu}\breve V(t,x,\mu,y,z)\sigma_0\sigma^\top_0\big)\mu(dy)\mu(dz)+\int_{\mathbb R^d}{\rm tr}\big(D_x\partial_\mu\breve V(t,x,\mu,y)\sigma_0\sigma^\top_0\big)\mu(dy)\\
    &+\int_{\mathbb R^d}\partial_\mu\breve V(t,x,\mu,y)\cdot D_y\breve{\cal H}\big(\breve\mu,y,\breve V(t,y,\mu)\big)\mu(dy)=0,\\
    &\breve V(T,x,\mu)=D_xG(x,\mu),\ (t,x,\mu)\in[0,T)\times\mathbb R^d\times\mathcal P_2(\mathbb R^d).
  \end{aligned}\right.
\end{align}
Given Assumption \ref{potential-game-assumption}, we find \eqref{HJB-MF-2} and \eqref{potential-SMP} are equivalent. Hence we may first refer to Theorem \ref{existence-of-decoupling-field} and obtain the well-posedness of \eqref{potential-SMP}, which gives $D_xV(t,x,\mu)$, and then recover $V(t,x,\mu)$. The discussion above constitutes the main idea of proving Theorem \ref{well-posedness-mfgc}.
\begin{proof}[Proof of Theorem \ref{well-posedness-mfgc}]
Given Assumption \ref{potential-game-assumption}, we find that \eqref{HJB-MF-2} and \eqref{potential-SMP} are equivalent. It is easy to verify that the classical solution of \eqref{potential-SMP} with bounded derivatives is the unique decoupling field corresponding to $Y_t$ in \eqref{Y_t=V(X)}. According to Theorem \ref{existence-of-decoupling-field} and Lemma \ref{short-time-well-posedness}, we obtain the existence of a unique global solution to \eqref{potential-SMP} such that $D^k_x\breve V(t,x,\cdot)\in\mathcal C^{3-k}\big(\mathcal P_2(\mathbb R^d)\big)$ for $(t,x,\mu)\in[0,T]\times\mathbb R^d\times\mathcal P_2(\mathbb R^d)$, $k=1,2,3$, where all derivatives are bounded by constants depending only on \eqref{tilde c-dependence} with $K=U$.
 
 Denote by $\breve V$ the solution to \eqref{potential-SMP} and replace the $D_x\mathcal V$ in \eqref{HJB-MFGC} with $\breve V$. We get
 \begin{align}\label{HJB-MFGC-1}
  \left\{\begin{aligned}
    &\partial_t\mathcal V(t,x,\mu)+\frac12{\rm tr}\big(D^2_{xx}\mathcal V(t,x,\mu)(\sigma\sigma^\top+\sigma_0\sigma^\top_0)\big)\\
    &+\breve{\cal H}\big(\breve\mu,x,\breve V(t,x,\mu)\big)+\frac12\int_{\mathbb R^d}{\rm tr}\big(D_y\partial_\mu\mathcal V(t,x,\mu,y)(\sigma\sigma^\top+\sigma_0\sigma^\top_0)\big)\mu(dy)\\
    &+\frac12\int_{\mathbb R^d}\int_{\mathbb R^d}{\rm tr}\big(\partial^2_{\mu\mu}\mathcal V(t,x,\mu,y,z)\sigma_0\sigma^\top_0\big)\mu(dy)\mu(dz)+\int_{\mathbb R^d}D_x\partial_\mu\mathcal V(t,x,\mu,y)\sigma_0\sigma^\top_0\mu(dy)\\
    &+\int_{\mathbb R^d}\partial_\mu\mathcal V(t,x,\mu,y)\cdot D_y\breve{\cal H}\big(\breve\mu,y,\breve V(t,y,\mu)\big)\mu(dy)=0,\\
    &\mathcal V(T,x,\mu)=G(x,\mu),\ (t,x,\mu)\in[0,T)\times\mathbb R^d\times\mathcal P_2(\mathbb R^d).
  \end{aligned}\right.
\end{align}
Given $\breve V$, \eqref{HJB-MFGC-1} above is a linear equation. Since we have shown that $\partial^k_x\breve V(t,x,\cdot)\in\mathcal C^{3-k}\big(\mathcal P_2(\mathbb R^d)\big)$ for $(t,x,\mu)\in[0,T]\times\mathbb R^d\times\mathcal P_2(\mathbb R^d)$, $k=1,2,3$, we have from Theorem \ref{li-result} the existence of a unique classical solution to \eqref{HJB-MFGC-1} that is continuously differentiable with bounded derivatives.

Next we show $\breve V=D_x\mathcal V$. To see that, take $D_x$ in \eqref{HJB-MFGC-1} and get
\begin{align*}
 \left\{\begin{aligned}
    &\partial_t\hat V(t,x,\mu)+\frac12{\rm tr}\big(D^2_{xx}\hat V(t,x,\mu)(\sigma\sigma^\top+\sigma_0\sigma^\top_0)\big)+D_x\breve{\cal H}\big(\breve\mu,x,\breve V(t,x,\mu)\big)\\
    &+D_p\breve{\cal H}\big(\breve\mu,x,\breve V(t,x,\mu)\big)\cdot D_x\breve V(t,x,\mu)+\frac12\int_{\mathbb R^d}{\rm tr}\big(D_y\partial_\mu\hat V(t,x,\mu,y)(\sigma\sigma^\top+\sigma_0\sigma^\top_0)\big)\mu(dy)\\
    &+\frac12\int_{\mathbb R^d}\int_{\mathbb R^d}{\rm tr}\big(\partial^2_{\mu\mu}\hat V(t,x,\mu,y,z)\sigma_0\sigma^\top_0\big)\mu(dy)\mu(dz)+\int_{\mathbb R^d}{\rm tr}\big(D_x\partial_\mu\hat V(t,x,\mu,y)\sigma_0\sigma^\top_0\big)\mu(dy)\\
    &+\int_{\mathbb R^d}\partial_\mu\hat V(t,x,\mu,y)\cdot D_y\breve{\cal H}\big(\breve\mu,y,\breve V(t,y,\mu)\big)\mu(dy)=0,\\
    &\hat V(T,x,\mu)=D_xG(x,\mu),\ (t,x,\mu)\in[0,T)\times\mathbb R^d\times\mathcal P_2(\mathbb R^d).
  \end{aligned}\right.
  \end{align*}
The above linear equation is on $\hat V$. Such linear equation admits two solutions, namely $\partial_x\mathcal V$ and $\breve V$. According to Theorem \ref{li-result}, the uniqueness of solution implies $\breve V=\partial_x\mathcal V$. Plugging $\breve V=\partial_x\mathcal V$ into \eqref{HJB-MFGC-1}, we see that \eqref{HJB-MFGC} admits a classical solution $\partial^k_x\breve V(t,x,\cdot)\in\mathcal C^{3-k}\big(\mathcal P_2(\mathbb R)\big)$ for $(t,x,\mu)\in[0,T]\times\mathbb R\times\mathcal P_2(\mathbb R)$, $k=1,2,3$, where all derivatives are bounded by constants depending only on \eqref{tilde c-dependence} with $K=U$. The uniqueness of solution to \eqref{HJB-MFGC} is implied by the uniqueness of solution to \eqref{potential-SMP}.
\end{proof}

\subsection{The degenerated case}\label{degenerate}
In this section we discuss the classical solutions to \eqref{HJB-MF} where there is no individual noise, i.e., $\sigma=0$. One key observation in previous sections is that the a priori estimates are independent of the diffusion coefficients $\sigma$ and $\sigma_0$. Such uniform estimates lead to the convergence to a classical solution to \eqref{HJB-MF} when $\sigma$ goes to zero.

In this section, for each $\sigma>0$ we denote by $V^\sigma$ the solution to \eqref{HJB-MF}. Our first aim is then to study the convergence of $V^\sigma$ and its derivatives as $\sigma\to0$, i.e., to prove Theorem \ref{degenerate-HJB-MF}.
\begin{proof}[Proof Theorem \ref{degenerate-HJB-MF}]
Let us focus on the compact set $\mathcal A_n\subset\mathcal P_2(\mathbb R^d)$ where
\begin{align*}
 \mathcal A_n:=\big\{\mu\in\mathcal P_2(\mathbb R^d):\mu(\{x:|x|\leq n\})=1\big\}.
\end{align*}
For each $\mathcal A_n,\ n\geq1,$ in view of Theorem \ref{existence-of-decoupling-field}, up to a subsequence, $(V^\sigma)_{\sigma>0}$ converges uniformly on $[0,T]\times\mathcal A_n$. Denote the associated limit by $V^{(n)}$. It is easy to show that $V^{(n+1)}$ and $V^{(n)}$ coincide on $\mathcal A_n$. Since $\cup_{n=1}^{+\infty}\mathcal A_n$ is dense in $\mathcal P_2(\mathbb R^d)$, it is then easy to uniquely extend the domain of $V^{(n)}$ and obtain a Lipschitz function $V$ defined on $[0,T]\times\mathcal P_2(\mathbb R^d)$. In order to show the regularity of $V$ on $[0,T]\times\mathcal P_2(\mathbb R^d)$, it suffices to show the uniform estimates and convergence of 
 \begin{align*}
  D^{i_1}_{x_1}\cdots D^{i_k}_{x_k}\partial^k_\mu V^\sigma(t,\mu,x_1,\ldots,x_k),\ 0\leq k\leq3,\ 0\leq i_1,\ldots,i_k\leq3,\ 1\leq i_1+\cdots+i_k+k\leq3,\ \mu\in\mathcal A_n,
 \end{align*}
as $\sigma\to0$. In view of the uniform estimates in Proposition \ref{further-regularity-1} as well as Arzel\`a-Ascoli lemma, it is easy to see that $D^{i_1}_{x_1}\cdots D^{i_k}_{x_k}\partial^k_\mu V^\sigma(t,\mu,x_1,\ldots,x_k)$ is uniformly bounded and admits uniformly convergent subsequence as $\sigma\to0$. It remains to show the uniqueness of the corresponding limit. Since all the limits are jointly continuous, we lose nothing by focusing on the limits of
\begin{align*}
 D^{i_1}_{x_1}\cdots D^{i_k}_{x_k}\partial^k_\mu V^\sigma\bigg(t,\frac1N\sum_{i=1}^N\delta_{y_i},x_1,\ldots,x_k\bigg)\quad\text{as}\quad\sigma\to0.
\end{align*}
Here we only show the aforementioned uniqueness of $\partial_\mu V\big(t,\frac1N\sum_{i=1}^N\delta_{y_i},x_1\big)$ because the proof for the rest can be done inductively using the same approach. 

At the moment, let us consider any subsequence $\{V^{\sigma_k}\}_{k\geq1}$, converging to the some limit $V$. Consider two limits of different subsequences of $\{\partial_\mu V^{\sigma_k}\}_{k\geq1}$, namely $\partial_\mu V$ and $\partial_\mu\tilde V$. For $(y_1,\ldots,y_N)\in(\mathbb R^d)^N$ with $y_i\neq y_j$, $1\leq i,j\leq N$, denote by
\begin{align*}
 \mu_\lambda:=\frac1N\delta_{y_1+\lambda\alpha}+\frac1N\sum_{i=2}^N\delta_{y_i},\ \lambda\in[0,1],\ \alpha\in\mathbb R^d.
\end{align*}
Sending $\sigma_k$ to $0$ in
\begin{align*}
 \varepsilon^{-1}\big(V^{\sigma_k}(t,\mu_\varepsilon)-V^{\sigma_k}(t,\mu_0)\big)=\varepsilon^{-1}\int_0^\varepsilon\partial_\mu V^{\sigma_k}(t,\mu_\lambda,y_1+\lambda\alpha)\cdot\alpha d\lambda
\end{align*}
along different subsequences, we obtain
\begin{align*}
 \varepsilon^{-1}\int_0^\varepsilon\partial_\mu V(t,\mu_\lambda,y_1+\lambda\alpha)d\lambda=\varepsilon^{-1}\int_0^\varepsilon\partial_\mu\tilde V(t,\mu_\lambda,y_1+\lambda\alpha)\cdot\alpha d\lambda.
\end{align*}
Passing $\varepsilon$ to $0$ yields
\begin{align*}
 \partial_\mu V\bigg(t,\frac1N\sum_{i=1}^N\delta_{y_i},y_1\bigg)=\partial_\mu\tilde V\bigg(t,\frac1N\sum_{i=1}^N\delta_{y_i},y_1\bigg).
\end{align*}
The continuity then implies $\partial_\mu V=\partial_\mu\tilde V$. To see that $V$ solves \eqref{HJB-MF}, we may send $\sigma$ to $0$ in \eqref{HJB-MF} along the subsequence $\{\sigma_k\}_{k\geq1}$ and utilize the uniform convergence.

It now remains to show that any subsequence of $V^\sigma$ converges to the same limit $V$ as $\sigma\to0$. Since any such limit $V$ solves \eqref{HJB-MF}, it suffices to show the uniqueness of the solution to \eqref{HJB-MF}. Consider any classical solution $V$ to \eqref{HJB-MF} for $\sigma=0$ with bounded derivatives. Let $V^\sigma_N$ denote the solution to \eqref{HJB-N} for any $\sigma>0$. For $x=(x_1,\ldots,x_N)$, denote by
\begin{align*}
 \hat v^\sigma_N(t,x):=V(t,\mu_x)-V^\sigma_N(t,x),\ \mu_x:=\frac1N\sum_{i=1}^N\delta_{x_i},\ (t,x)\in[0,T]\times(\mathbb R^d)^N,
\end{align*}
Then
\begin{align*}
  \left\{\begin{aligned}
    &\partial_t\hat v^\sigma_N(t,x)+\frac12\sum_{i=1}^N{\rm tr}\big(D^2_{x_ix_i}\hat v^\sigma_N(t,x)\sigma\sigma^\top\big)+\frac12\sum_{i,j=1}^N{\rm tr}\big(D^2_{x_ix_j}\hat v^\sigma_N(t,x)\sigma_0\sigma^\top_0\big)+h(t,x)\cdot\nabla_x\hat v^\sigma_N(t,x)=\tilde g_N(t,x),\\
    &\hat v^\sigma_N(T,x)=0,\ (t,x)\in[0,T)\times(\mathbb R^d)^N.
  \end{aligned}\right.
\end{align*}
Here
\begin{align*}
 \tilde g_N(t,x)=\frac1{2N^2}\sum_{i=1}^N{\rm tr}\big(\partial^2_{\mu\mu}V(t,\mu_x,x_i,x_i)\sigma\sigma^\top\big)+\frac1{2N}\sum_{i=1}^N{\rm tr}\big(D_x\partial_\mu V(t,\mu_x,x_i)\sigma\sigma^\top\big).
\end{align*}
In view the estimates in Theorem \ref{existence-of-decoupling-field}, we may show that
\begin{align*}
 \big|V(t,\mu_x)-V^\sigma_N(t,x_1,\ldots,x_N)\big|\leq C\sigma^2,\quad\mu_x:=\frac1N\sum_{i=1}^N\delta_{x_i}.
\end{align*}
where the constant $C$ is independent of $\sigma$. Hence we may deduce the uniqueness of solution as well as the uniqueness of limit by sending $\sigma$ to $0$.
\end{proof}
\begin{remark}\label{eigen-0-1-degenerate}
 In view of the proof of Theorem \ref{degenerate-HJB-MF}, when $\sigma=0$, it follows that $V(t,\mu_x)=V_N(t,x_1,\ldots,x_N)$. Using a similar convergence argument to that in the proof of Theorem \ref{degenerate-HJB-MF}, it is easy to see that Theorem \ref{eigen-0-1} still holds when $\sigma=0$. We may also infer from the convergence argument that the derivatives mentioned in Theorem \ref{degenerate-HJB-MF} are all Lipschitz continuous in $(\mu,x_1,\ldots,x_j)$, $1\leq j\leq 3$, where the Lipschitz constant depends only on \eqref{tilde c-dependence} with $K=U$.
\end{remark}
Using Theorem \ref{degenerate-HJB-MF} and a similar argument to the proof of Theorem \ref{well-posedness-mfgc}, we can also obtain the following well-posedness of \eqref{HJB-MFGC} with $\sigma=0$.
 \begin{proof}[Proof of Theorem \ref{well-posedness-mfgc-degenerate}] Thanks to Theorem \ref{degenerate-HJB-MF} and Remark \ref{eigen-0-1-degenerate}, we may take $D_x$ in \eqref{HJB-MF} and show the well-posedness of the resulting equation, and then obtain the solution $\breve V$ to \eqref{HJB-MFGC-1} with $\sigma=0$. Next we can show that $\breve V$ uniquely solves \eqref{HJB-MFGC} following the same argument as in the proof of Theorem \ref{well-posedness-mfgc}. 
 \end{proof}

\section{Approximating $\partial_\mu V$ and Nash equilibria with particle systems}\label{approximate-partial-mu-V}
As is shown in Proposition \ref{prop-propagation-1}, the solution $V_N$ to \eqref{HJB-N} converges to $V$ in \eqref{HJB-MF} at the optimal rate $O(N^{-1})$ under Assumption \ref{assumption}. In fact, it is well known that such convergence, also known as the propagation of chaos, particle approximation or law of large number, holds under more general conditions at a possibly different rate (please see the previously mentioned references on the convergence of value functions and optimizers/equilibriums). One consequence of such convergence is the corresponding numerical methods towards MFC/MFG problems (see e.g. \cite{Carmona2022}). In this section, we further show the propagation of chaos concerning $\partial_\mu V$, or equivalently $D_x\mathcal V$. To be exact, we first show that $N\partial_1 V_N$ can be extended, in an appropriate sense, to a Lipschitz function $\partial_\mu V_N(t,\cdot)$ on $\mathbb R^d\times\mathcal P_2(\mathbb R^d)$. Moreover, $\partial_\mu V_N$ converges to $\partial_\mu V$ as $N$ tends to infinity, which usually implies the convergence to optimal feedback functions for MFC/MFG in various models. Above all, thanks to our previous results, the aforementioned convergence can be quantified with an algebraic convergence rate.

Let us now describe how one could approximate $\partial_\mu V$ from the particle system corresponding to $V_N$. For empirical measure $\mu=\frac1{N-1}\sum_{i=2}^N\delta_{x_i}$, we may define
\begin{align}\label{approximate-1}
\partial_\mu V_N(t,x_1,\mu):=N\int_{\mathbb R}\cdots\int_{\mathbb R}D_{x_1} V_N(t,x_1,\tilde x_2,\ldots,\tilde x_N)\mu(d\tilde x_2)\cdots\mu(d\tilde x_N).
\end{align}
Direct calculation reveals that $\partial_\mu V$ is a Lipschitz function defined on $[0,T]\times\mathbb R^d\times\mathcal P_2(\mathbb R^d)$. Thanks to our previous results, we can show that $\partial_\mu V_N(t,x,\mu)$ could serve as an approximation of $\partial_\mu V(t,x,\mu)$ with an algebraic convergence rate, and thus usually induce an approximate optimizer in various models.
\begin{lemma}\label{approximate-feedback}
 Assume Assumption \ref{assumption}. Then for $(t,x)\in[0,T]\times(\mathbb R^d)^N$, $1\leq k\leq N$,
 \begin{align*}
  \bigg|ND_{x_k}V_N(t,x_1,\ldots,x_N)-\partial_\mu V\bigg(x_k,\frac1N\sum_{j=1}^N\delta_{x_j}\bigg)\bigg|\leq\frac CN,
 \end{align*}
where the constant $C$ depends only on \eqref{tilde c-dependence} with $K=U$.
\end{lemma}
\begin{proof}
Let us adopt the notations in Proposition \ref{prop-propagation-1}, where
 \begin{align*}
 & \hat v_N(t,x)=V\bigg(t,\frac1N\sum_{j=1}^N\delta_{x_j}\bigg)-V_N(t,x_1,\ldots,x_N),\ v_N(t,x)=V(t,\mu_x),\\
  &g_N(t,x)=\frac1{2N^2}\sum_{i=1}^N{\rm tr}\big(\partial^2_{\mu\mu}V(t,\mu_x,x_i,x_i)\sigma\sigma^\top\big),
 \end{align*}
and
 \begin{align*}
  \left\{\begin{aligned}
    &\quad\partial_t\hat v_N(t,x)+\frac12\sum_{i=1}^N{\rm tr}\big(D^2_{x_ix_i}\hat v_N(t,x)\sigma\sigma^\top\big)+\frac12\sum_{i,j=1}^N{\rm tr}\big(D^2_{x_ix_j}\hat v_N(t,x)\sigma_0\sigma^\top_0\big)\\
    &=g_N(t,x)+H_N(x,\nabla V_N)-H_N(x,\nabla v_N),\\
    &\hat v_N(T,x)=0,\quad(t,x)\in[0,T)\times(\mathbb R^d)^N.
  \end{aligned}\right.
 \end{align*}
Here and below, for the sake of brevity, we use $\nabla V_N$, $\nabla v_N$ and so forth to denote $\nabla V_N(t,x)$, $\nabla v_N(t,x)$ respectively. Denote by $x_k=\big(x^{(1)}_k,\ldots,x^{(d)}_k\big)$. Take $\partial_{x^{(l)}_k}$ in the above and denoting by $\hat v^{kl}_N(t,x)=\partial_{x^{(l)}_k}\hat v_N(t,x)$, we have
 \begin{align*}
  \left\{\begin{aligned} &\quad\partial_t\hat v^{kl}_N(t,x)+\frac12\sum_{j=1}^N{\rm tr}\big(D^2_{x_jx_j}\hat v^{kl}_N(t,x)\sigma\sigma^\top\big)+\frac12\sum_{i,j=1}^N{\rm tr}\big(D^2_{x_ix_j}\hat v^{kl}_N(t,x)\sigma_0\sigma^\top_0\big)+\sum_{j=1}^ND_{p_j}H_N(x,\nabla V_N)\cdot D_{x_j}\hat v^{kl}_N\notag\\
    &=\partial_{x^{(l)}_k}g_N(t,x)+\sum_{j=1}^N\big(\partial_{p_j}H_N(x,\nabla V_N)-\partial_{p_j}H_N(x,\nabla v_N)\big)\cdot D_{x_j}\partial_{x^{(l)}_k}\hat v_N+\partial_{x^{(l)}_k}H_N(x,\nabla V_N)-\partial_{x^{(l)}_k}H_N(x,\nabla v_N),\\
    &\hat v^{kl}_N(T,x)=0,\quad(t,x)\in[0,T)\times(\mathbb R^d)^N,\  k=1,\ldots,N.\end{aligned}\right.
 \end{align*}
 It is easy to see
 \begin{align*}
  |\partial_{x_k}H_N(x,\nabla V_N)-\partial_{x_k}H_N(x,\nabla v_N)|\leq C|D_{x_k}\hat v_N|+\frac CN\sum_{i=1}^N|D_{x_i}\hat v_N|.
 \end{align*}
 According to Theorem \ref{existence-of-decoupling-field} and Theorem \ref{degenerate-HJB-MF},
 \begin{align*}
  \bigg|\sum_{j=1}^N\big(D_{p_j}H_N(x,\nabla V_N)-D_{p_j}H_N(x,\nabla v_N)\big)\cdot D_{x_j}\partial_{x^{(l)}_k}\hat v_N\bigg|\leq \frac CN\sum_{j=1}^N|D_{x_j}\hat v_N|,\ |\partial_{x^{(l)}_k}g_N(t,x)|\leq\frac C{N^2}.
 \end{align*}
Here the constants above depend only on \eqref{tilde c-dependence} with $K=U$. Consider the forward process
\begin{align*}
 \left\{\begin{aligned}&dX^i_N(s)=D_{p_i}H_N\big(X_N(s),\nabla v_N(s,X_N(s))\big)ds+\sigma dW_i(s)+\sigma_0 dW_0(s),\\
 &X^i_N(t)=x_i,\quad1\leq i\leq N,\end{aligned}\right.
\end{align*}
and $\hat Y^i_N(s):=D_{x_i}\hat v_N(s,X(s)),\ s\in[t,T],\ 1\leq i\leq N$. Then the equation on $\hat v^{kl}_N$ and above estimates yields
\begin{align*}
 \left\{\begin{aligned}d\hat Y^k_N(s)&=\bigg[D_{x_k}g_N\big(t,X_N(s)\big)+\sum_{i=1}^N\eta_{ki}(s)\hat Y^i_N(s)\bigg]ds+\sum_{i=0}^NZ^i_N(s)dW_i(s),\\
 \hat Y^k_N(T)&=0,\ 1\leq k\leq N.\end{aligned}\right.
\end{align*}
 where for $1\leq k,i\leq N$, $\eta_{ki}$ is a process satisfying $|\eta_{ki}|\leq C\delta_{ki}+\frac CN$. Hence by Gr{\"o}nwall's inequality,
\begin{align*}
  |\hat Y^k_N(s)|\leq\frac C{N^2}+C\mathbb E_s\bigg[\int_s^T\frac 1N\sum_{i=1}^N|\hat Y^i_N(u)|du\bigg],\quad k=1,\ldots,N.
\end{align*}
Hence summing $k$ from $1$ to $N$ and Gr{\"o}nwall's inequality gives
\begin{align*}
 |ND_{x_k}\hat v_N(t,x)|=|N\hat Y^k_N(t)|\leq\frac CN,
\end{align*}
where the constant depends only on \eqref{tilde c-dependence} with $K=U$, and the proof is completed.
\end{proof}
Now we may give an estimate on the aforementioned convergence rate of $\partial_\mu V_N$ to $\partial_\mu V$.
\begin{proposition}\label{particle-appr-control}
Suppose Assumption \ref{assumption}. For $1<p<2$, $\mu\in\mathcal P_2(\mathbb R^d)$, the Lipschitz function $\partial_\mu V_N(t,x,\mu)$ on $[0,T]\times\mathbb R^d\times\mathcal P_2(\mathbb R^d)$ converges to $\partial_\mu V(t,\mu,x)$ uniformly on bounded sets at a rate
 \begin{align*}
  &\quad\big|\partial_\mu V_N(t,x,\mu)-\partial_\mu V(t,x,\mu)\big|\\
  &\leq\frac{C(x_1,\mu)}N+C\times\bigg(\int_{\mathbb R^d}|y|^2\mu(dy)\bigg)^{\frac12}\times\begin{cases}
N^{-1/2}+N^{-(q-p)/q}
& \text{if } p>d/2,\\
N^{-1/2}\log(1+N)+N^{-(q-p)/q}
& \text{if } p=d/2,\\
N^{-p/d}+N^{-(q-p)/q}
& \text{if } p\in(0,d/2) \text{ and } q\neq d/(d-p).
\end{cases}
 \end{align*}
\end{proposition}
\begin{proof}
 In view of Lemma \ref{approximate-feedback},
 \begin{align*}
  \int_{\mathbb R^d}\cdots\int_{\mathbb R^d}\bigg|ND_{x_1}V_N(t,x_1,\tilde x_2\ldots,\tilde x_N)-\partial_\mu V\bigg(x_1,\frac1N\delta_{x_1}+\frac1N\sum_{j=2}^N\delta_{\tilde x_j}\bigg)\bigg|\mu(d\tilde x_2)\cdots\mu(d\tilde x_N)\leq\frac CN.
 \end{align*}
 Therefore
 \begin{align*}
   &\quad\int_{\mathbb R^d}\cdots\int_{\mathbb R^d}\big|ND_{x_1}V_N(t,x_1,\tilde x_2\ldots,\tilde x_N)-\partial_\mu V(x_1,\mu)\big|\mu(d\tilde x_2)\cdots\mu(d\tilde x_N)\notag\\
   &\leq\int_{\mathbb R^d}\cdots\int_{\mathbb R^d}\bigg|ND_{x_1}V_N(t,x_1,\tilde x_2\ldots,\tilde x_N)-\partial_\mu V\bigg(x_1,\frac1N\delta_{x_1}+\frac1N\sum_{j=1}^N\delta_{\tilde x_j}\bigg)\bigg|\mu(d\tilde x_2)\cdots\mu(d\tilde x_N)\notag\\
   &\quad+\int_{\mathbb R^d}\cdots\int_{\mathbb R^d}\bigg|\partial_\mu V\bigg(x_1,\frac1N\delta_{x_1}+\frac1N\sum_{j=2}^N\delta_{\tilde x_j}\bigg)-\partial_\mu V\bigg(x_1,\frac1N\sum_{j=1}^N\delta_{\tilde x_j}\bigg)\bigg|\mu(d\tilde x_2)\cdots\mu(d\tilde x_N)\notag\\
   &\quad+\int_{\mathbb R^d}\cdots\int_{\mathbb R^d}\bigg|\partial_\mu V\bigg(x_1,\frac1N\sum_{j=1}^N\delta_{\tilde x_j}\bigg)-\partial_\mu V(x_1,\mu)\bigg|\mu(d\tilde x_2)\cdots\mu(d\tilde x_N)\notag\\
   &=:I+II+III.
 \end{align*}
We have shown that $I\leq\frac CN$. Moreover, since $\partial^2_{\mu\mu}V$ is bounded,
\begin{align*}
 II\leq\frac CN\int_{\mathbb R}|x_1-\tilde x_1|\mu(d\tilde x_1).
\end{align*}
According to \cite{Fournier15},
\begin{align*}
 III&=\mathbb E\bigg[\partial_\mu V\bigg(x_1,\frac1N\sum_{j=1}^N\delta_{\xi_j}\bigg)-\partial_\mu V(x_1,\mu)\bigg]\leq C\mathbb E\bigg[\mathcal W_p\bigg(\frac1N\sum_{j=1}^N\delta_{\xi_j},\mu\bigg)\bigg]\\
 &\leq C\times\bigg(\int_{\mathbb R^d}|y|^2\mu(dy)\bigg)^{\frac12}\times\begin{cases}
N^{-1/2}+N^{-(q-p)/q}
& \text{if } p>d/2,\\
N^{-1/2}\log(1+N)+N^{-(q-p)/q}
& \text{if } p=d/2,\\
N^{-p/d}+N^{-(q-p)/q}
& \text{if } p\in(0,d/2) \text{ and } q\neq d/(d-p).
\end{cases}
\end{align*}
Combine the above and the proof is now completed.
\end{proof}
One application of solutions to master equations is to construct an approximate Markovian Nash equilibrium for $N$-player games. In view of the previous results, one may construct an approximate Markovian Nash equilibrium for the $N$-player version of \eqref{mean-field-particle}$\sim$\eqref{optim.-1} based on $V_N$. More specifically, the alternative approximate Markovian Nash equilibrium can be constructed as follows. Consider $N$-player games with the following $k$-th player, $k=1,\ldots,N$,
\begin{align}\label{Nash-model-1}
   \left\{\begin{aligned}
    &dX_{N,k}(s) = \theta_{N,k}(s)ds +\sigma dW^k(s)+\sigma_0dW^0_s,\\
&X_{N,k}(t)=x_k\in\mathbb R^d,
\end{aligned}\right.
 \end{align}
and the cost to the $k$-th player
\begin{align}\label{Nash-model-2}
   J_{N,k}(t,x,\theta_{N,1},\ldots,\theta_{N,N}):=\mathbb E\bigg[\int_t^T\tilde L\big(X_{N,k}(s),\theta_{N,k}(s),\hat\mu^{N,k}_s\big)ds+G(x\sharp\hat\mu^{N,k}_T,X(T))\bigg],
\end{align}
where
\begin{align*}
 \hat\mu^{N,k}_s=\frac1{N-1}\sum_{i\neq k}\delta_{\big(X_{N,i}(s),\theta_{N,i}(s)\big)}.
\end{align*}
We note that the framework \eqref{Nash-model-1}$\sim$\eqref{Nash-model-2} can be easily adapted to the case where $X_{N,k}(0)$ are random variables with finite second moments. In such case, according to \cite{Lacker20-1}, the $\varepsilon_N$-Markovian Nash equilibriums consist of those $(\theta_{N,1},\ldots,\theta_{N,N})$ generated by feedback functions in such a way that
\begin{align*}
 \theta_{N,i}(s)=\tilde\theta_{N,i}\big(s,X_N(s)\big),\quad\tilde\theta_{N,i}:\ [0,T]\times(\mathbb R^d)^N\mapsto\mathbb R^d,\ 1\leq i\leq N,
\end{align*}
and 
\begin{align*}
 J_{N,k}(t,x,\theta_{N,1},\ldots,\theta_{N,k}\ldots,\theta_{N,N})\leq J_{N,k}(t,x,\theta_{N,1},\ldots,\beta,\ldots,\theta_{N,N})+\varepsilon_N,
\end{align*}
for any other $\beta$ generated by feedback functions. In this note we also study the deviation $\varepsilon_{N,k}$ such that 
\begin{align}\label{appr-deviation}
 J_{N,k}(t,x,\theta_{N,1},\ldots,\theta_{N,k}\ldots,\theta_{N,N})\leq J_{N,k}(t,x,\theta_{N,1},\ldots,\beta,\ldots,\theta_{N,N})+\varepsilon_{N,k},\ 1\leq k\leq N,
\end{align}
for any other $\beta$ generated by feedback functions.

Next we show that
\begin{align}\label{appr-feedback}
 \hat\theta^*_{N,k}(t,x):=\partial_{\tilde\mu}\mathcal H^{(p)}\big(\tilde\mu_t,x_k,NV^k_N(t,x)\big),\quad1\leq k\leq N,
\end{align}
is an approximate Markovian Nash equilibrium, where the notation $V^k_N(t,x)=D_{x_k}V_N(t,x)$ is from \eqref{abbreviation} and
\begin{align*}
  \tilde\mu_t:=\frac1N\sum_{i=1}^N\delta_{(x_i,NV^i_N(t,x))}.
 \end{align*}
Thanks to the analysis on the particle system in Theorem \ref{eigen-0-1} and Lemma \ref{approximate-feedback}, one might show that such approximate equilibrium has an error $O(N^{-1})$.
\begin{proposition}\label{particle-appr-equil}
Suppose Assumption \ref{assumption} and Assumption \ref{potential-game-assumption}. For the $N$-player Nash games \eqref{Nash-model-1}$\sim$\eqref{Nash-model-2} where the $k$-th player adopts feedback function \eqref{appr-feedback} with initial position $x_k$, \eqref{appr-deviation} holds with
 \begin{align}\label{bound-varepsilon-Nk}
  \varepsilon_{N,k}\leq C\bigg(\frac1{N-1}\sum_{j\neq k}|x_j-x_k|+\frac1{N-1}\bigg),
 \end{align}
where $C$ depends only on $\breve H$ and \eqref{C-depend}. As a result, when each player has i.i.d. initial distribution in $\mathcal P_2(\mathbb R^d)$, the feedback functions in \eqref{appr-feedback} constitute an $\varepsilon_N$-Markovian Nash equilibrium with $\varepsilon_N=O(N^{-1})$.
\end{proposition}
\begin{proof} 
 Let $J^k_N(t,x)$ be the objective function of the $k$-th agent where the feedback function for each agent is in \eqref{appr-feedback}. Then $J^k_N(t,x)$ uniquely solves
\begin{align}\label{particle-appr-equil-2}
  \left\{\begin{aligned}&\partial_tJ^k_N+\frac12\sum_{i=1}^N{\rm tr}\big(D^2_{x_ix_i}J^k_N\sigma\sigma^\top\big)+\frac12\sum_{i,j=1}^N{\rm tr}\big(D^2_{x_ix_j}J^k_N\sigma_0\sigma^\top_0\big)+\sum_{i=1}^N\partial_{\tilde\mu}\mathcal H^{(p)}(\tilde\mu_t,x_i,NV^i_N)\cdot D_{x_i}J^k_N\\
  &\qquad+\tilde L\big(x_k,\partial_{\tilde\mu}\mathcal H^{(p)}(\tilde\mu_t,x_k,N V^k_N),\breve\mu^{-k}_t\big)=0,\\
  &J^k_N(T,x)=G\bigg(\frac1{N-1}\sum_{i\neq k}\delta_{x_i},x_k\bigg),\ (t,x)\in[0,T)\times(\mathbb R^d)^N,\ x=(x_1,\ldots,x_N).\end{aligned}\right.
 \end{align}
 where $\tilde\mu^{-k}_t=\frac1{N-1}\sum_{j\neq k}\delta_{(x_j,NV^j_N)}$ and
 \begin{align*}
 \breve\mu^{-k}_t(dxdp)=\frac1{N-1}\sum_{j\neq k}\delta_{\big(x_j,\partial_{\tilde\mu}\mathcal H^{(p)}(\tilde\mu_t,x_j,NV^j_N)\big)}=\big(x,\partial_{\tilde\mu}\mathcal H^{(p)}(\tilde\mu_t,x,p)\big)\sharp\tilde\mu^{-k}_t(dxdp).
 \end{align*}
At the meant time, the value function $J^{k,*}_N$ for the $k$-th player solves
\begin{align*}
 &\partial_tJ^{k,*}_N+\frac12\sum_{i=1}^N{\rm tr}\big(D^2_{x_ix_i}J^{k,*}_N\sigma\sigma^\top\big)+\frac12\sum_{i,j=1}^N{\rm tr}\big(D^2_{x_ix_j}J^{k,*}_N\sigma_0\sigma^\top_0\big)+\sum_{i\neq k}\partial_{\tilde\mu}\mathcal H^{(p)}(\tilde\mu_t,x_i,NV^i_N)\cdot D_{x_i}J^{k,*}_N\notag\\
  &\qquad+\breve{\mathcal H}\big(x_k,\partial_kJ^{k,*}_N,\breve\mu^{-k}_t\big)=0,
\end{align*}
with the same terminal condition as \eqref{particle-appr-equil-2}. Under Assumption \ref{potential-game-assumption}, we have
\begin{align*}
 \partial_x\mathcal V(t,x,\mu)=\partial_\mu V(t,\mu,x),
\end{align*}
where $V$ and $\mathcal V$ are from \eqref{HJB-MF} and \eqref{HJB-MFGC} respectively. Plugging $\mathcal J^k_N(t,x):=\mathcal V(t,x_k,\mu_N)$ into the above, in view of Theorem \ref{well-posedness-mfgc}, Theorem \ref{well-posedness-mfgc-degenerate} and Lemma \ref{approximate-feedback}, one may show that
\begin{align*}
 &\partial_t\mathcal J^k_N+\frac12\sum_{i=1}^N{\rm tr}\big(D^2_{x_ix_i}\mathcal J^k_N\sigma\sigma^\top\big)+\frac12\sum_{i,j=1}^N{\rm tr}\big(D^2_{x_ix_j}\mathcal J^k_N\sigma_0\sigma^\top_0\big)+\sum_{i\neq k}\partial_{\tilde\mu}\mathcal H^{(p)}(\tilde\mu_t,x_i,NV^i_N)\cdot D_{x_i}\mathcal J^k_N\notag\\
  &\qquad+\breve{\mathcal H}\big(x_k,\partial_k\mathcal J^k_N,\breve\mu^{-k}_t\big)=g^{(1)}_{N,k}(t,x),
\end{align*}
where the terminal condition is the same as \eqref{particle-appr-equil-2} and the error term satisfies
\begin{align*}
 |g^{(1)}_{N,k}(t,x)|\leq C\bigg(\frac1{N-1}\sum_{j\neq k}|x_j-x_k|+\frac1{N-1}\bigg).
\end{align*}
Therefore
\begin{align*}
 |J^{k,*}_N-\mathcal J^k_N|\leq C\bigg(\frac1{N-1}\sum_{j\neq k}|x_j-x_k|+\frac1{N-1}\bigg),
\end{align*}
where the constant depends only on $\breve{\mathcal H}$ and \eqref{C-depend}. According to Assumption \ref{potential-game-assumption},
\begin{align}\label{estimate-L}
 &\quad\tilde L\big(x_k,\partial_{\tilde\mu}\mathcal H^{(p)}(\tilde\mu_t,x_k,NV^k_N),\breve\mu^{-k}_t\big)=\tilde L\big(x_k,\partial_p\breve H(x_k,NV^k_N,\breve\mu_t),\breve\mu_t\big)-g^{(2)}(t,x)\notag\\
 &=\breve{\mathcal H}\big(x_k,NV^k_N,\breve\mu_t\big)-\partial_p\breve{\mathcal H}(x_k,NV^k_N,\breve\mu_t)\cdot NV^k_N-g^{(2)}_{N,k}(t,x)\notag\\
 &=\breve{\mathcal H}\big(x_k,NV^k_N,\breve\mu_t\big)-\partial_{\tilde\mu}\mathcal H^{(p)}(\tilde\mu_t,x_k,NV^k_N)\cdot NV^k_N-g^{(2)}_{N,k}(t,x),
\end{align}
where
\begin{align*}
&g^{(2)}_{N,k}(t,x)=\tilde L\big(x_k,\partial_p\breve H(x_k,NV^k_N,\breve\mu_t),\breve\mu^{-k}_t\big)-\tilde L\big(x_k,\partial_p\breve H(x_k,NV^k_N,\breve\mu_t),\breve\mu_t\big),\\
&\breve\mu_t(dxdp)=\big(x,\partial_{\tilde\mu}\mathcal H^{(p)}(\tilde\mu_t,x,p)\big)\sharp\tilde\mu_t(dxdp).
\end{align*}
After a simple coupling, one may have from the defintion
\begin{align*}
 \mathcal W_1(\breve\mu_t,\breve\mu^{-k}_t)\leq C\mathcal W_1(\tilde\mu_t,\tilde\mu^{-k}_t)\leq C\bigg(\frac1{N-1}\sum_{j\neq k}|x_j-x_k|+\frac1{N-1}\bigg).
\end{align*}
Hence
\begin{align*}
&|g^{(2)}_{N,k}(t,x)|\leq C\mathcal W_1(\breve\mu_t,\breve\mu^{-k}_t)\leq C\mathcal W_1(\tilde\mu_t,\tilde\mu^{-k}_t)\leq C\bigg(\frac1{N-1}\sum_{j\neq k}|x_j-x_k|+\frac1{N-1}\bigg).
\end{align*}
Plug \eqref{estimate-L} into \eqref{particle-appr-equil-2} and notice by Lemma \ref{approximate-feedback}
\begin{align*}
 |NV^k_N-D_{x_k}\mathcal J^k_N|\leq C\bigg(\frac1{N-1}\sum_{j\neq k}|x_j-x_k|+\frac1{N-1}\bigg),
\end{align*}
as well as
\begin{align*}
 &\partial_t\mathcal J^k_N+\frac12\sum_{i=1}^N{\rm tr}\big(D^2_{x_ix_i}\mathcal J^k_N\sigma\sigma^\top\big)+\frac12\sum_{i,j=1}^N{\rm tr}\big(D^2_{x_ix_j}\mathcal J^k_N\sigma_0\sigma^\top_0\big)+\sum_{i=1}^N\partial_{\tilde\mu}\mathcal H^{(p)}(\tilde\mu^{-i}_t,x_i,D_{x_i}\mathcal J^i_N)D_{x_i}\mathcal J^k_N\\
  &\qquad+\breve{\mathcal H}\big(x_k,D_{x_k}\mathcal J^k_N,\breve\mu^{-k}_t\big)=g^{(3)}_{N,k}(t,x),
\end{align*}
with
\begin{align*}
 |g^{(3)}_{N,k}(t,x)|\leq\frac CN.
\end{align*}
Combining the above, we obtain
\begin{align}\label{appr-performance}
 |J^k_N-\mathcal J^k_N|\leq C\bigg(\frac1{N-1}\sum_{j\neq k}|x_j-x_k|+\frac1{N-1}\bigg).
\end{align}
As a result,
\begin{align*}
 |J^k_N-J^{k,*}_N|\leq C\bigg(\frac1{N-1}\sum_{j\neq k}|x_j-x_k|+\frac1{N-1}\bigg).
 \end{align*}
 In other words, we have proved \eqref{bound-varepsilon-Nk}.
 
 Replacing the $x_j$ ($1\leq j\leq N$) in the above with i.i.d. random variables and taking expectation, one may obtain that the feedback functions in \eqref{appr-feedback} constitute an $\varepsilon_N$-Markovian Nash equilibrium with $\varepsilon_N=O(N^{-1})$.
\end{proof}
\begin{remark}
 Given the well-posedness of \eqref{HJB-MFGC}, one may also consider the case where there exists a classical solution $v^{N,i}$ ($1\leq i\leq N$) to the Nash system describing equilibriums of \eqref{Nash-model-1}$\sim$\eqref{Nash-model-2}:
 \begin{align*}
 \left\{\begin{aligned}&\partial_tv^{N,k}+\frac12\sum_{i=1}^N{\rm tr}\big(D^2_{x_ix_i}v^{N,k}\sigma\sigma^\top\big)+\frac12\sum_{i,j=1}^N{\rm tr}\big(D^2_{x_ix_j}v^{N,k}\sigma_0\sigma^\top_0\big)+\sum_{i\neq k}\partial_p\breve{\mathcal H}(x_i,D_{x_i}v^{N,i},\breve\nu^{-i}_t)\cdot D_{x_i}v^{N,i}\notag\\
  &\qquad+\breve{\mathcal H}\big(x_k,D_{x_k}v^{N,k},\breve\nu^{-k}_t\big)=0,\notag\\
  &v^{N,k}(T,x)=G\bigg(\frac1{N-1}\sum_{i\neq k}^N\delta_{x_i},x_k\bigg),\ (t,x)\in[0,T)\times(\mathbb R^d)^N,\ x=(x_1,\ldots,x_N).
  \end{aligned}\right.
\end{align*}
where
\begin{align*}
 \breve\nu^{-k}_t(dxdp)=\big(x,\partial_{\tilde\mu}{\mathcal H}^{(p)}(\tilde\nu^{-k}_t,x,p)\big)\sharp\tilde\nu_t(dxdp),\quad\tilde\nu^{-k}_t=\frac1{N-1}\sum_{j\neq k}\delta_{(x_j,D_{x_j}v^{N,j})},
\end{align*}
and thus
\begin{align*}
 \breve\nu^{-k}_t=\frac1{N-1}\sum_{j\neq k}\delta_{\big(x_j,D_p\breve{\mathcal H}(x_j,D_{x_j}v^{N,j},\breve\nu^{-j}_t)\big)}.
\end{align*}
Following a similar argument to that in \cite{Cardaliaguet19}, one could show that
\begin{align*}
 \mathcal J^k_N-v^{N,k}=O(N^{-1}).
\end{align*}
In view of \eqref{appr-performance}, the feedback functions in \eqref{appr-feedback} generate approximate a Markovian Nash equilibrium whose performance approximates $v^{N,i}$ ($1\leq i\leq N$) with an error $O(N^{-1})$.
\end{remark}

{\bf Acknowledgements.} HL is partially supported by NSFC grant 12522122, NSFC grant 12531009, Liaoning Revitalization Talents Program (Grant  XLYCE2503058), and Hong Kong RGC Grant ECS 21302521. CM acknowledges the support provided by NSFC grant 12522122, NSFC/RGC JRS N\_CityU165/25, GRF 11311422, and GRF 11303223.

\bibliographystyle{plain}

\end{document}